\documentclass[a4paper,10pt]{amsart}
\usepackage[top=0.8in, bottom=0.8in, left=1.25in, right=1.25in]{geometry}
\usepackage{fancybox}
\usepackage{amsmath}
\usepackage{amssymb}
\usepackage{amsthm}
\usepackage{mathrsfs}
\usepackage{epstopdf}
\usepackage{epstopdf}
\usepackage{array} 
\usepackage{amscd}
\usepackage{color}
\usepackage{hyperref}

\usepackage[all]{xypic}
\usepackage{tikz}

\usepackage{amsfonts}
\usepackage{amssymb}
\usepackage{mathrsfs}

\DeclareFontFamily{U}{rsfs}{%
\skewchar\font127}
\DeclareFontShape{U}{rsfs}{m}{n}{%
<-6>rsfs5<6-8.5>rsfs7<8.5->rsfs10}{}
\DeclareSymbolFont{rsfs}{U}{rsfs}{m}{n}
\DeclareSymbolFontAlphabet
{\mathrsfs}{rsfs}
\DeclareRobustCommand*\rsfs{%
\@fontswitch\relax\mathrsfs}

\theoremstyle{plain}
\newtheorem{thm}{Theorem}[section]
\newtheorem{prop}[thm]{Proposition}
\newtheorem{lem}[thm]{Lemma}

\newtheorem{defi}[thm]{Definition}

\newtheorem{rmk}[thm]{Remark}
\newtheorem{cor}[thm]{Corollary}

\newtheorem{ccase}{Case}

\newtheorem{prop-defi}[thm]{Proposition-Definition}
\newtheorem{thm-defi}[thm]{Theorem-Definition}
\newtheorem{lem-defi}[thm]{Lemma-Definition}

\newtheorem{conj}[thm]{Conjecture}
\newtheorem{exam}[thm]{Example}
\newtheorem{set}[thm]{Setting}

\newdimen\argwidth
\def\db[#1\db]{
 \setbox0=\hbox{$#1$}\argwidth=\wd0
 \setbox0=\hbox{$\left[\box0\right]$}
  \advance\argwidth by -\wd0
 \left[\kern.3\argwidth\box0 \kern.3\argwidth\right]}

\newcommand{\aA}{\mathcal{A}}

\newcommand{\eE}{\mathcal{E}}

\newcommand{\hH}{\mathcal{H}}

\newcommand{\oO}{\mathcal{O}}
\newcommand{\pP}{\mathcal{P}}

\newcommand{\Supp}{\mathop{\rm Supp}\nolimits}
\newcommand{\Hom}{\mathop{\rm Hom}\nolimits}

\newcommand{\dotimes}{\stackrel{\textbf{L}}{\otimes}}
\newcommand{\dR}{\mathbf{R}}

\newcommand{\ch}{\mathop{\rm ch}\nolimits}

\newcommand{\Ext}{\mathop{\rm Ext}\nolimits}
\newcommand{\Spec}{\mathop{\rm Spec}\nolimits}
\newcommand{\rank}{\mathop{\rm rank}\nolimits}
\newcommand{\Coh}{\mathop{\rm Coh}\nolimits}

\newcommand{\Per}{\mathop{\rm Per}\nolimits}

\newcommand{\cneq}{\mathrel{\raise.095ex\hbox{:}\mkern-4.2mu=}}
\newcommand{\eqcn}{\mathrel{=\mkern-4.5mu\raise.095ex\hbox{:}}}

\newcommand{\oPPer}{\mathop{\rm ^{0}Per}\nolimits}

\newcommand{\iPPer}{\mathop{\rm ^{-1}Per}\nolimits}

\newcommand{\ppPPer}{\mathop{^{{p}}\rm{Per}}\nolimits}

\newcommand{\modu}{\mathop{\rm mod}\nolimits}
\newcommand{\End}{\mathop{\rm End}\nolimits}

\newcommand{\Imm}{\mathop{\rm Im}\nolimits}

\newcommand{\Ker}{\mathop{\rm Ker}\nolimits}

\newcommand{\RHom}{\mathop{\dR\mathrm{Hom}}\nolimits}

\makeatletter
 
  \@addtoreset{equation}{section}
\makeatother

\renewcommand{\labelenumi}{(\roman{enumi})}

\newcommand{\lkakko}{[\![}
\newcommand{\rkakko}{]\!]}

\title[{Counting perverse coherent systems on CY 4-folds}]
{Counting perverse coherent systems \\ on Calabi-Yau 4-folds}
\date{}
\author{Yalong Cao}
\address{RIKEN Interdisciplinary Theoretical and Mathematical Sciences Program (iTHEMS), 2-1, Hirosawa, Wako-shi, Saitama, 351-0198, Japan}
\email{yalong.cao@riken.jp}
\author{Yukinobu Toda}
\address{Kavli Institute for the Physics and Mathematics of the Universe (WPI),The University of Tokyo Institutes for Advanced Study, The University of Tokyo, Kashiwa, Chiba 277-8583, Japan}
\email{yukinobu.toda@ipmu.jp}

\begin{document}
\begin{abstract}
Nagao-Nakajima introduced counting invariants of stable perverse coherent systems on small resolutions of Calabi-Yau 3-folds and 
determined them on the resolved conifold. Their invariants recover DT/PT invariants and Szendr\"oi's non-commutative invariants in some chambers of stability conditions. In this paper, we study an analogue of their work on Calabi-Yau 4-folds. We define counting invariants for 
stable perverse coherent systems using primary insertions and compute them in all chambers of stability conditions.  
We also study counting invariants of local resolved conifold $\mathcal{O}_{\mathbb{P}^1}(-1,-1,0)$ defined using torus localization and tautological insertions. 
We conjecture a wall-crossing formula for them, which upon dimensional reduction recovers Nagao-Nakajima's wall-crossing formula on resolved conifold.

\end{abstract}
\maketitle



\tableofcontents

\section{Introduction}

\subsection{Background on CY 3-folds}
For a contractible rational curve on a Calabi-Yau 3-fold $Z$, 
we have the flop
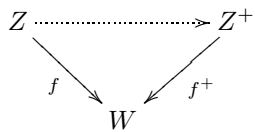
\begin{align*}
\xymatrix{
Z \ar[dr]_-{f} \ar@{.>}[rr] &  & Z^{+} \ar[ld]^-{f^{+}} \\
& W &
}
\end{align*}
where $f$ contracts the rational curve 
 to a Gorenstein singularity and 
$f^{+}$ is a blow-up of $W$ in another way
so that $Z \dashrightarrow Z^{+}$ is not an isomorphism. 
In \cite{Bri1}, Bridgeland introduced perverse coherent sheaves
associated with 3-fold flopping contractions and used them to prove the equivalence of derived categories of $Z$ and $Z^+$ which
was conjectured earlier by Bondal and Orlov \cite{BO}.
Shortly after that, Van den Bergh \cite{VB} constructed non-commutative resolution of $W$ and realized Bridgeland's equivalence through it. 
 
Nagao-Nakajima \cite{NN} introduced counting invariants for stable perverse coherent systems associated with flopping contractions of Calabi-Yau 3-folds. 
They determined their invariants for any chamber of stability conditions on resolved conifold $\mathcal{O}_{\mathbb{P}^1}(-1,-1)$, which recover DT/PT invariants \cite{DT, PT} 
and Szendr\"oi's non-commutative invariants \cite{Sze, Young} (see also Mozgovoy-Reineke \cite{MR}) in some special chambers. 
As a corollary of their wall-crossing formula, they recovered DT/PT correspondence \cite{Bri3, Toda1} and flop formula \cite{Cala, Toda2} in this case.

\subsection{Perverse coherent systems on projective CY 4-folds}
In this paper, we are interested in extending their work to Calabi-Yau 4-folds. Our setting is the following: 
\begin{set}\label{intro setting}
	Let $X$ be a smooth projective Calabi-Yau 4-fold 
	and $f \colon X \to Y$ be a projective birational contraction 
	which contracts an irreducible surface $E \subset X$ to a 
	curve $C \subset Y$. 
	We assume that formal neighborhood at 
	each point $p \in C \subset Y$ is of the form 
	\begin{align*}
		\widehat{\oO}_{Y, p} \cong 
\mathbb{C}\lkakko x, y, z, w, u \rkakko/(xy-zw). \end{align*}
\end{set}
In the above setting, 
one can show that 
$\dR f_{\ast}\oO_X=\oO_Y$, 
the singular locus $C \subset Y$ is a smooth connected curve, 
and the morphism $f|_{E} \colon E\to C$ is a ruled surface whose 
fibers have normal bundle
$\oO_{\mathbb{P}^1}(-1,-1,0)$ in $X$.
In this case, the abelian category $\Per(X/Y)$ of perverse coherent sheaves \eqref{per heart} still makes sense and 
Van den Bergh's work  \cite{VB} applies. Following Nagao-Nakajima, we consider a pair (called \textit{perverse coherent system})
\begin{align}\label{intro:pair}
	(F, s), \,\,
	F \in \Per(X/Y), \,\, 
	s \colon \oO_X \to F. 
\end{align}
For $\Theta=(\theta_0, \theta_1) \in \mathbb{R}^2$, 
we will define $\Theta$-(semi)stability 
for perverse coherent systems (see Definition \ref{pair stab}), 
and construct the coarse moduli space 
\begin{align*}
	P_n^{\Theta}(X, \beta)=\big\{(F, s) : \Theta\mbox{-semistable 
		pairs (\ref{intro:pair}) with }[F]=\beta,\, \chi(F)=n\big\}/\sim
\end{align*}
of $S$-equivalence classes of $\Theta$-semistable perverse coherent systems
(see Theorem~\ref{cons of moduli}).

We are only interested in curve classes $\beta$ such that $f_*\beta=0$, i.e.~classes in fibers\,\footnote{These are analogy of curve classes of 
resolved conifold considered by Nagao-Nakajima \cite{NN}.} of $f$. 
We will classify
walls for $\Theta$-stability, which turn out to 
consist of six types
denoted by 
$L_{\pm}^{\pm}(k)$, 
$L_{\pm}^{\mp}(k)$
for $k\in \mathbb{Z}_{\geqslant 0}$ and $L_{\pm}(\infty)$
(see~Proposition~\ref{prop:wall}). 
This 
 wall-chamber structure for 
 $\Theta$-stability is described in Figure~\ref{figure1}, 
 which is 
 the same as that
  of the resolved conifold in~\cite{NN} (see Lemma~\ref{lem identify wall on cpt and local}). 
\begin{figure}
 \centering
\begin{tikzpicture}[node distance=1cm]
\draw[thick] (-4.6,0)--(4.6,0)  node [pos=0, anchor=east]{\tiny{$\theta_1=0$}} ;\draw[thick](0,4.6)--(0,-4.6)  node [pos=0, anchor=east]{\tiny{$\theta_0=0$}} ;
\draw[thick] (-4,4)--(4,-4)  node [pos=0, anchor=east]{\tiny{$\theta_0+\theta_1=0$}} ;
\draw[thick] (-3.5,4.3)--(3.5,-4.3)  node[pos=0, anchor=east]{\tiny{$m\theta_0+(m-1)\theta_1=0$} } ;
\draw[thick] (-2.3,4.6)--(2.3,-4.6)  node[pos=0, anchor=east]{\tiny{$2\theta_0+\theta_1=0$} } ;
\draw[thick] (-4.3,3.5)--(4.3,-3.5)  node[pos=0, anchor=east]{\tiny{$m\theta_0+(m+1)\theta_1=0$} } ;
\draw[thick] (-4.6,2.3)--(4.6,-2.3)  node[pos=0, anchor=east]{\tiny{$\theta_0+2\theta_1=0$} } ;
\draw[fill] (-3.5,3.6) circle [radius=0.025];
\node [thick, right] at (-3.5,3.6) {\tiny{PT}};
\draw[fill] (-3.6,3.5) circle [radius=0.025];
\node [thick, below] at (-3.6,3.5) {\tiny{DT}};
\draw[fill] (3.8,-3.7) circle [radius=0.025];
\node [thick, above] at (3.8,-3.7) {\tiny{PT}};

\draw[fill] (3.7,-3.8) circle [radius=0.025];
\node [thick, left] at (3.7,-3.8) {\tiny{DT}};

\node at (4.2,-3.7) {\small{$X^+$}};
\node at (-4,3.7) {\small{$X$}};

\draw[fill] (-3.46,2.5) circle [radius=0.015];
\draw[fill] (-3.51,2.4) circle [radius=0.015];
\draw[fill] (-3.56,2.3) circle [radius=0.015];

\draw[fill] (-2.5,3.46) circle [radius=0.015];
\draw[fill] (-2.4,3.51) circle [radius=0.015];
\draw[fill] (-2.3,3.56) circle [radius=0.015];

\draw[fill] (-3.15, 2.97) circle [radius=0.015];
\draw[fill] (-3.2, 2.9) circle [radius=0.015];
\draw[fill] (-3.25, 2.83) circle [radius=0.015];

\draw[fill] (-2.97, 3.15) circle [radius=0.015];
\draw[fill] (-2.9, 3.22) circle [radius=0.015];
\draw[fill] (-2.83, 3.28 ) circle [radius=0.015];

\node at (2.5,2) {$\begin{subarray}{c}\mathrm{empty\,\, chamber} \end{subarray}$ } ;
\node at (-3,-2.3) {$\begin{subarray}{c}\mathrm{non-commutative}  \\ \mathrm{chamber} \end{subarray}$ } ;
\end{tikzpicture} 
\caption{Wall-chamber structures of $\oO_{\mathbb{P}^1}(-1,-1,0)$}
\label{figure1}
\end{figure}
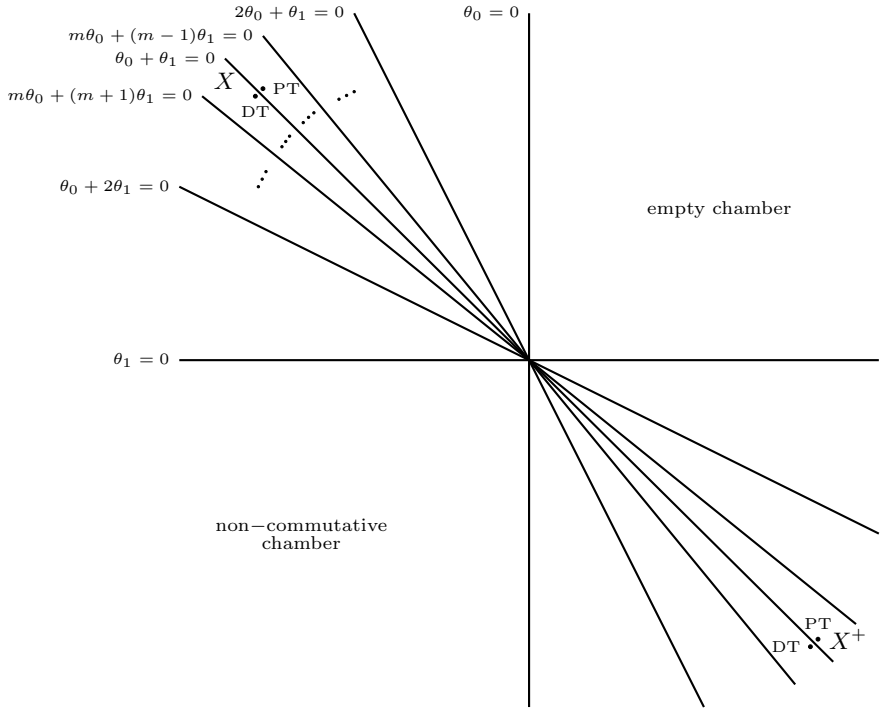
In special chambers, our moduli spaces recover moduli spaces of $Z_t$-stable pairs introduced in \cite{CT1} (therefore also recover PT stable pairs
~\cite{PT, CMT2}),
Hilbert schemes of curves, and perverse Hilbert schemes (see Proposition \ref{Z_t chamber}, \ref{dt/pt chamber}, \ref{nc chamber} respectively).

When $\Theta$ lies outside walls, $P^{\Theta}_n(X,\beta)$ consists of 
only stable objects and admits a 
$(-2)$-shifted symplectic derived scheme structure in the sense of Pantev-To\"en-Vaquie-Vezzosi \cite{PTVV}. Therefore
there exists a virtual class\,\footnote{See \cite{CL1} for construction in some special cases and \cite{OT} for a recent algebro-geometric approach.}
$$[P^{\Theta}_n(X,\beta)]^{\mathrm{vir}}\in H_{2n}(P^{\Theta}_n(X,\beta),\mathbb{Z}), $$ 
in the sense of Borisov-Joyce \cite{BJ}, which depends on the choice of orientation \cite{CGJ} (see Proposition \ref{prop on exist of virt class}).
In order to define counting invariants, we consider
 \textit{primary insertions}:  
\begin{align*}
\tau \colon H^{4}(X,\mathbb{Z})\to H^{2}(P^{\Theta}_n(X,\beta),\mathbb{Z}), \quad 
\tau(\gamma) \cneq (\pi_{P})_{\ast}(\pi_X^{\ast}\gamma \cup\ch_3(\mathbb{F}) ),
\end{align*}
where $\pi_X$, $\pi_P$ are projections from $X \times P^{\Theta}_n(X,\beta)$
onto corresponding factors, $\mathbb{I}=(\pi_X^*\oO_X\to \mathbb{F})$ is the universal pair and $\ch_3(\mathbb{F})$ is the
Poincar\'e dual to the fundamental cycle of $\mathbb{F}$.

The primary counting invariants of $\Theta$-stable perverse 
coherent systems are defined by 
\begin{align*}P^{\Theta}_{n,\beta}(\gamma):=\int_{[P^{\Theta}_n(X,\beta)]^{\rm{vir}}} \tau(\gamma)^n\in \mathbb{Z}. \end{align*}
The first purpose of this paper is to 
completely determine these invariants 
 for all chambers of stability conditions: 
\begin{thm}\emph{(Theorem \ref{cpt main thm})}\label{intro cpt main thm}
Let $f\colon X\to Y$ be as in Setting \ref{intro setting}, $E\subset X$ be the exceptional surface and $[\mathbb{P}^1]\in H_2(X,\mathbb{Z})$ be the fiber class of $f|_{E} \colon E\to C$. Let $\Theta=(\theta_0,\theta_1)\in \mathbb{R}^2$ be outside walls defined in \eqref{all walls}. 
Then for certain choice of orientation, we have 
\begin{align*}
	\sum_{n\in \mathbb{Z}, f_{\ast}\beta=0}
	\frac{P^{\Theta}_{n,\beta}(\gamma)}{n!}\,q^nt^{\beta}=\left\{
	\begin{array}{cl}
		\exp\left(qt^{[\mathbb{P}^1]}\right)^{\int_X\gamma\cup [E] }           &\mbox{ if } \theta_0<0,\, \theta_0+2\theta_1>0, \\
		&  \\
		\exp\left(qt^{[\mathbb{P}^1]}-qt^{-[\mathbb{P}^1]}\right)^{\int_X\gamma\cup [E] }          &\mbox{ if } \theta_0<0,\, \theta_0+2\theta_1<0,  \\
		&  \\
		\exp\left(-qt^{-[\mathbb{P}^1]}\right)^{\int_X\gamma\cup [E] }         &\mbox{ if }  \theta_0>0,\, \theta_0+2\theta_1<0,  \\
		&  \\ 
		1  \quad \quad \quad           &   \mbox{ otherwise}. 
	\end{array} \right. 
\end{align*} 
\end{thm}
We remark that the wall-chamber structure of primary counting invariants 
(see Figure~\ref{figure2}) is different from that of corresponding moduli spaces as in Figure~\ref{figure1}. 
Indeed most walls in Figure~\ref{figure1} (except those in Figure~\ref{figure2}) 
do not contribute to the wall-crossing formula of primary 
invariants\footnote{A heuristic explanation for this is that the primary insertion picks up incident perverse coherent systems which can be thought
as supporting on local resolved conifold. Then one can argue using quiver model that only walls in Figure~\ref{figure2} will contribute to wall-crossing. See the proof of 
Theorem \ref{cpt main thm} for more details.}. 

The result of this theorem in particular proves some of our 
previous conjectures: 
\begin{cor}\emph{(Corollary \ref{cor on verify prev conj})}
The LePotier-pair/GV conjecture \cite[Conjecture~0.2]{CT1}, 
PT/GV conjecture \cite[\S 0.7]{CMT2} and 
DT/PT conjecture \cite[Conjecture~0.3]{CK2} hold for fiber classes in Setting \ref{intro setting}.
\end{cor}
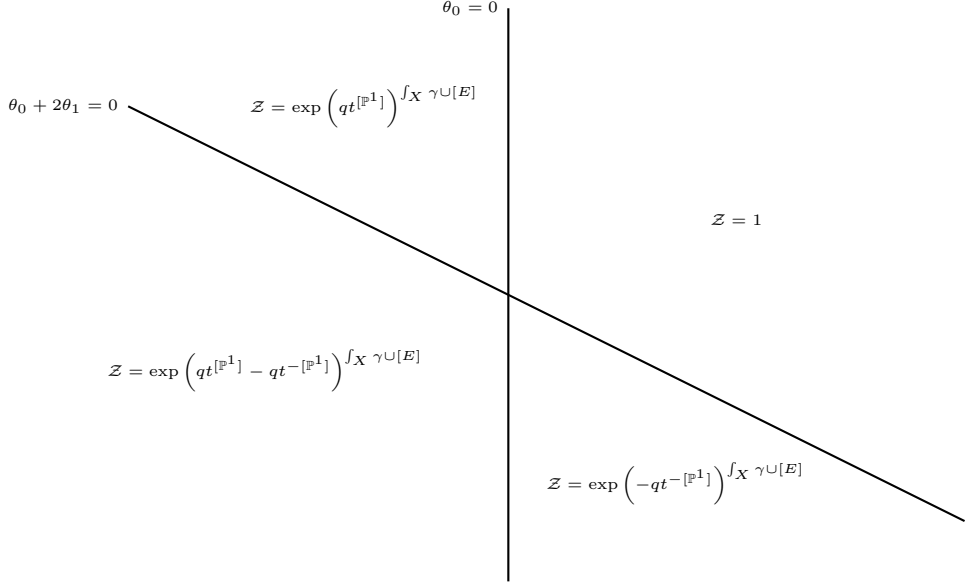
\begin{figure}
\begin{tikzpicture}[node distance=1cm]
\draw[thick](0,3.8)--(0,-3.8)  node [pos=0, anchor=east]{\tiny{$\theta_0=0$}} ;
\draw[thick] (-5,2.5)--(6,-3)  node[pos=0, anchor=east]{\tiny{$\theta_0+2\theta_1=0$} } ;
\node at (3,1) {\tiny{$\mathcal{Z}=1$} } ;
\node at (-1.9,2.5) {\tiny{$\mathcal{Z}=\exp\left(qt^{[\mathbb{P}^1]}\right)^{\int_X\gamma\cup [E] }$}};
\node at (-3.2,-1) {\tiny{$\mathcal{Z}=\exp\left(qt^{[\mathbb{P}^1]}-qt^{-[\mathbb{P}^1]}\right)^{\int_X\gamma\cup [E] }$} } ;
\node at (2.2,-2.5) {\tiny{$\mathcal{Z}=\exp\left(-qt^{-[\mathbb{P}^1]}\right)^{\int_X\gamma\cup [E] }$}};
\end{tikzpicture} 
\caption{Counting invariants for primary insertions---Theorem \ref{intro cpt main thm}}
\label{figure2}
\end{figure}
 
\subsection{Perverse coherent systems on local resolved conifold}
We also consider similar counting problem for the local resolved conifold \begin{align*}
X=\oO_{\mathbb{P}^1}(-1,-1,0).
\end{align*}
Since moduli spaces of $\Theta$-semistable perverse coherent systems
on it are non-compact,
we define counting invariants using torus localization. As there is no compact 4-cycle in $X$, instead of primary insertions, 
we consider tautological insertions as in \cite{CK1, CKM1, CT3}. 
We take a CY torus  
\begin{align*}T_0=\{t=(t_0,t_1,t_2,t_3)\in(\mathbb{C}^*)^4:\,t_0t_1t_2t_3=1\},
\end{align*}
which acts on $X$ preserving the CY 4-form. 
This action lifts to an action on $P^{\Theta}_n(X,d)$ with finitely many reduced points as torus fixed loci (see Proposition \ref{prop torus fixed loci}). 
Therefore we can define equivariant tautological invariants (see Definition \ref{taut inv}):
\begin{align*}
	P^{\Theta}_{n,d}(e^m)
	\cneq \sum_{\begin{subarray}{c}I=(\oO_X\to F) \in P^{\Theta}_n(X,d)^{T_0} \end{subarray}}
e_{T_0}(\chi_X(I,I)^{\frac{1}{2}}_0)\cdot e_{T_0\times \mathbb{C}^*}(\chi_X(F)^{\vee}\otimes e^m)\in \Lambda. 
\end{align*}
Here we consider a trivial $\mathbb{C}^*$-action on moduli spaces and $e^m$ is a trivial line bundle with $\mathbb{C}^*$-equivariant weight $m$. 
$\Lambda$ is the field of rational functions 
of equivariant parameters $m$ and 
$\lambda_i=e_{T_0}(t_i)$.  
The above invariants 
depend on choice of sign for each torus fixed point. 

For $\Theta_{\mathrm{PT}}:=(-1+0^+,1)$, the corresponding 
invariants 
\begin{align*}
P_{n, d}(e^m) \cneq P_{n, d}^{\Theta_{\mathrm{PT}}}(e^m) 
\end{align*}
enumerate PT stable pairs, which have a remarkable conjectural formula.  
\begin{conj}\label{intro conj formula in pt chamber}\emph{(Cao-Kool-Monavari \cite{CKM1})}
There exist choices of signs such that 
\begin{align*}
	\sum_{n,d}P_{n,d}(e^m)q^nt^d=\prod_{k\geqslant 1}\left(1-q^k t\right)^{k\cdot \frac{m}{\lambda_3}},
	\end{align*}
where $-\lambda_3$ is the equivariant parameter of $\oO_{\mathbb{P}^1}$ in $X$.
\end{conj}
\begin{rmk}
In fact, the choices of signs such that the above formula holds are conjectured to be unique and checked for small $n,d$ $($ref.~\cite[Conjecture~0.16, Proposition~B.1]{CKM1}$)$.
\end{rmk}
The second purpose of this paper is to 
give an interpretation of Conjecture~\ref{intro conj formula in pt chamber}
in terms of wall-crossing of $\Theta$-stable perverse coherent systems. 
Suppose that $\Theta$ lies on 
 one of the walls in Figure~\ref{figure1} except 
 the DT/PT wall, and $\Theta_{\pm}$ lies on 
 its adjacent chambers.
We consider the flip type diagram of $T_0$-fixed loci of good moduli spaces: 
\begin{align}\label{diagram:wall}
\xymatrix{
\bigcup_{n,d} P^{\Theta_{-}}_n(X,d)^{T_0} \ar[rd]_{\pi^-} & & \ar[ld]^{\pi^+} \bigcup_{n,d}P^{\Theta_{+}}_n(X,d)^{T_0} \\
& \bigcup_{n,d} P^{\Theta}_n(X,d)^{T_0}. & }
\end{align}
Here $P^{\Theta}_n(X,d)^{T_0}$ consists of $\Theta$-polystable perverse coherent systems of type 
\begin{align*}
	I_0\oplus S_{k-1}^{\oplus r}[-1], \quad r\geqslant 0,
	\end{align*}
where $I_0$ is a $T_0$-fixed $\Theta$-stable perverse coherent system,
$S_{k-1}$ is a $T_0$-fixed $\Theta$-stable
perverse coherent sheaf with $\Theta(S_{k-1})=0$, 
and $r$ can be computed from Chern character of $I_0$. 
$S_{k-1}$ is determined by the type of wall, e.g. 
$S_{k-1}=\oO_{\mathbb{P}^1}(k-1)$
if $\Theta \in L_-^-(k)$ (see~(\ref{def:Sk}) for details). 

When $m=\lambda_3$, there exists a dimensional reduction which relates our invariants with Nagao-Nakajima's invariants on
the 3-fold $\oO_{\mathbb{P}^1}(-1,-1)$ (see 
Proposition \ref{dim red}). 
In \cite[Theorem~3.12]{NN}\footnote{In this paper, we always use the numbering in the arxiv version of \cite{NN}.}, they proved a wall-crossing formula 
by stratifying $\pi^{\pm}$ into Grassmannian bundles and showed that the
 difference of invariants under wall-crossing is independent 
of the choice of $I_0$. Motivated by the idea of their
 wall-crossing formula, 
we conjecture a similar phenomenon holds 
 for our 4-fold invariants on $\oO_{\mathbb{P}^1}(-1,-1,0)$:
\begin{conj}\label{intro wall cross conj}\emph{(Conjecture \ref{wall cross conj})}
	Let $\Theta$ lie on 
	one of the 
	walls $L_{\pm}^{-}(k)$, $L_{\pm}^{+}(k)$ in \eqref{all walls}. 
	For a $T_0$-fixed $\Theta$-stable perverse coherent system 
	$I_0$, we consider the following 
	sequence of $\Theta$-polystable objects with $r\geqslant 0$: 
	\begin{align*}
		P_{k-1, r}^{I_0}:=\left\{I_0\oplus S_{k-1}^{\oplus r}[-1]
		\right\}\in \bigcup_{n,d}P^{\Theta}_n(X,d)^{T_0}. 
		\end{align*}
\begin{itemize}
\item If $\Theta=(\theta_0,\theta_1)\in L_{-}^-(k)$ or $L_{+}^-(k)$ $($$k\geqslant1$$)$ and 
$\Theta_{\pm}=(\theta_0\mp 0^+,\theta_1)$, then there exist choices of signs such that
$$\frac{\sum_{r}t^r\sum_{I\in \pi_+^{-1}(P^{I_0}_{k-1,r})}e_{T_0}(\chi_X(I,I)^{\frac{1}{2}}_0)\cdot e_{T_0\times \mathbb{C}^*}(\chi_X(F)^{\vee}\otimes e^m)}{ 
\sum_{r}t^r\sum_{I\in \pi_-^{-1}(P^{I_0}_{k-1,r})}e_{T_0}(\chi_X(I,I)^{\frac{1}{2}}_0)\cdot e_{T_0\times \mathbb{C}^*}(\chi_X(F)^{\vee}\otimes e^m)}
=(1-t)^{k\frac{m}{\lambda_3}}. $$
\item If $\Theta=(\theta_0,\theta_1)\in L_{-}^+(k)$ or $L_{+}^+(k)$ $($$k\geqslant 0$$)$ and 
$\Theta_{\pm}=(\theta_0\mp 0^+,\theta_1)$, then there exist choices of signs such that
$$\frac{\sum_{r}t^r\sum_{I\in \pi_+^{-1}(P^{I_0}_{k-1,r})}e_{T_0}(\chi_X(I,I)^{\frac{1}{2}}_0)\cdot e_{T_0\times \mathbb{C}^*}(\chi_X(F)^{\vee}\otimes e^m)}{ 
\sum_{r}t^r\sum_{I\in \pi_-^{-1}(P^{I_0}_{k-1,r})}e_{T_0}(\chi_X(I,I)^{\frac{1}{2}}_0)\cdot e_{T_0\times \mathbb{C}^*}(\chi_X(F)^{\vee}\otimes e^m)}
=(1-t^{-1})^{k\frac{m}{\lambda_3}}.$$
\end{itemize}
\end{conj}
The point of those formulae in Conjecture~\ref{intro wall cross conj} is that the 
quotient series in the LHS are independent of the choice of $I_0$. 
So by taking the summation for all $T_0$-fixed $\Theta$-stable 
perverse coherent systems $I_0$, 
we obtain the following wall-crossing formula of tautological invariants: 
\begin{prop}\emph{(Proposition \ref{wall crossing formula})}
Assuming Conjecture \ref{intro wall cross conj},
then we have the following:   
\begin{itemize}
\item If $\Theta=(\theta_0,\theta_1)\in L_{-}^-(k)$ or $L_{+}^-(k)$ $($$k\geqslant1$$)$ and 
$\Theta_{\pm}=(\theta_0\mp 0^+,\theta_1)$, then there exist choices of signs such that
$$\frac{\sum_{n,d}P^{\Theta_+}_{n,d}(e^m)q^nt^d}{\sum_{n,d}P^{\Theta_-}_{n,d}(e^m)q^nt^d}
=(1-q^{k}t)^{k\frac{m}{\lambda_3}}. $$
\item If $\Theta=(\theta_0,\theta_1)\in L_{-}^+(k)$ or $L_{+}^+(k)$ $($$k\geqslant 0$$)$ and 
$\Theta_{\pm}=(\theta_0\mp 0^+,\theta_1)$, then there exist choices of signs such that
$$\frac{\sum_{n,d}P^{\Theta_+}_{n,d}(e^m)q^nt^d}{\sum_{n,d}P^{\Theta_-}_{n,d}(e^m)q^nt^d}
=(1-q^{k}t^{-1})^{k\frac{m}{\lambda_3}}. $$
\end{itemize}
\end{prop}
In particular, by applying the above proposition from empty chamber to PT chamber, 
we obtain 
a wall-crossing interpretation of Conjecture \ref{intro conj formula in pt chamber}. It also provides a conjectural 
formula for non-commutative tautological invariants 
of $\oO_{\mathbb{P}^1}(-1, -1, 0)$ (see Corollary \ref{cor on nc}).

When $m=\lambda_3$, Conjecture \ref{intro wall cross conj} reduces to Nagao-Nakajima's wall-crossing formula (Proposition~\ref{dim red}). 
Apart from this, we give several further evidence of our conjecture: 
\begin{thm}\emph{(Theorem \ref{thm on js}, Proposition \ref{check conj muti C case}, \ref{check conj single C case})}
Conjecture \ref{intro wall cross conj} holds for $L_{-}^-(k)$ when
\begin{itemize}
\item $I_0=\oO_X$
\item $I_0=I_{l\mathbb{P}^1}$, $l=1$, $k=2$, up to degree $t^{16}$,
\item $I_0=I_{l\mathbb{P}^1}$, $l=2$, $k=2$, up to degree $t^{10}$,
\item $I_0=I_{l\mathbb{P}^1}$, $l=3,4$, $k=2$, up to degree $t^9$,
\item $I_0=I_{l\mathbb{P}^1}$, $l=5$, $k=2$, up to degree $t^8$,
\item $I_0=I_{l\mathbb{P}^1}$, $l=6$, $k=2$, up to degree $t^7$,
\item $I_0=I_{l\mathbb{P}^1}$, $l=7,8,9,10$, $k=2$, up to degree $t^6$,
\item $I_0=I_{l\mathbb{P}^1}$, any $l$, $k=2$, up to degree $t^5$.
\item  $I_0=I_{\mathbb{P}^1}$, $k=3$, up to degree $t^5$,
\item  $I_0=I_{\mathbb{P}^1}$, $k=4,5$, up to degree $t^2$,
\item $I_0=I_{\mathbb{P}^1}$,  $k\leqslant 12$, up to degree $t^1$.
\end{itemize}
Here $I_{l\mathbb{P}^1}:=\left(\oO_X\twoheadrightarrow \oO_{\mathbb{P}^1}\otimes \sum_{j=0}^{l-1} t_3^{j}\right)$
is the ideal sheaf of thickened $\mathbb{P}^1$ into $\oO_{\mathbb{P}^1}$-direction in $X$. 
\end{thm}
The first case is proved using its compact analogue and Atiyah-Bott localization formula. 
Other cases are done with the help of a computer program, which usually involves checking very nontrivial combinatoric identities (see Example \ref{combi identi}).
We remark that we use consistent sign rule as discussed in Remark \ref{sign rmk} to check our conjecture,
and it is an interesting question to link our conjectural wall-crossing formula with the recent wall-crossing proposal of \cite{GJT}.

\subsection{Acknowledgement}
Y. C. is grateful to Martijn Kool and Sergej Monavari for previous collaboration \cite{CKM1} which gives a motivation of this paper.
We thank a responsible referee for very careful reading of our paper whose suggestions improve the exposition of the paper.
This work is partially supported by the World Premier International Research Center Initiative (WPI), MEXT, Japan.
Y. C. is partially supported by RIKEN Interdisciplinary Theoretical and Mathematical Sciences
Program (iTHEMS), JSPS KAKENHI Grant Number JP19K23397 and Royal Society Newton International Fellowships Alumni 2019 and 2020.
Y. T. is supported by Grant-in Aid for Scientific Research grant (No. 19H01779) from MEXT, Japan.

\section{Perverse coherent systems on projective CY 4-folds}
Built on the renowned work of Bridgeland \cite{Bri1} on perverse coherent sheaves for 3-fold flopping contractions, 
Nagao-Nakajima \cite{NN} introduced perverse coherent systems and their counting invariants. We study 
an analogy of their work in the setting of projective CY 4-folds.

\subsection{Perverse coherent sheaves}

In this section, we assume the following setting: 
\begin{set}\label{setting}
Let $X$ be a smooth projective Calabi-Yau 4-fold 
and $f \colon X \to Y$ be a projective birational contraction 
which contracts an irreducible surface $E \subset X$ to a 
curve $C \subset Y$. 
We assume that formal neighborhood at 
each point $p \in C \subset Y$ is of the form 
\begin{align*}
\widehat{\oO}_{Y, p} \cong 
\mathbb{C}\lkakko x, y, z, w, u \rkakko/(xy-zw).
\end{align*}
\end{set}
Under the above setting, 
one can show that 
$\dR f_{\ast}\oO_X=\oO_Y$, 
the singular locus $C \subset Y$ is a smooth connected curve, 
the morphism $f|_{E} \colon E\to C$ is a ruled surface whose 
fibers have normal bundle
$\oO_{\mathbb{P}^1}(-1,-1,0)$ in $X$. 
\begin{exam}Let $g\colon Z\to W$ be a 3-fold flopping contraction of a $(-1,-1)$ curve on a smooth projective CY 3-fold $Z$. 
For an elliptic curve $E$, the $E$-copy of $g$ gives a contraction above. \end{exam}
As in Bridgeland \cite{Bri1},
for $p\in \mathbb{Z}$ we consider the following heart of \textit{perverse} t-structure on $D^b\Coh(X)$:
\begin{align}\label{per heart}
	\ppPPer(X/Y) \cneq 
\left\{E\in D^b\Coh(X): \begin{array}{l}
\dR f_*E\in \Coh(Y), \\
\Hom(E,\mathscr{C}^{>p})=\Hom(\mathscr{C}^{<p},E)=0
\end{array}  \right\},
\end{align}
where 
\begin{align*}
\mathscr{C}^{>p} &:=\{F\in D^b\Coh(X): \dR f_*F=0,\,\, \mathcal{H}^{\leqslant p}(F)=0 \}, \\
\mathscr{C}^{<p} &:=\{F\in D^b\Coh(X): \dR f_*F=0,\,\, \mathcal{H}^{\geqslant p}(F)=0 \}. 
\end{align*}
In this paper, we mainly use the $p=-1$ perversity
\begin{align*}
	\Per(X/Y) \cneq \iPPer(X/Y). 
	\end{align*}
It is easy to see that $\oO_X\in \Per(X/Y)$.
By the renowned result of Van den Bergh \cite{VB}, there exists a \textit{local projective generator} of $\Per(X/Y)$
\begin{align}\label{local proj gene}\mathcal{P}=\oO_X \oplus \mathcal{P}_0, \end{align}
which exists as a vector bundle on $X$, a sheaf $\mathcal{A}_Y:=f_*\mathcal{E}nd(\mathcal{P})$ of non-commutative algebras on $Y$
and a derived equivalence
\begin{align}\label{psi equi cptX}
&\Phi \colon D^b\Coh(X) \stackrel{\sim}{\to} D^b \Coh(\aA_Y), \quad 
(-) \mapsto \RHom(\pP, -),
\end{align}
which restricts to an equivalence 
between $\Per(X/Y)$ and $\Coh(\aA_Y)$. 

The morphism $f \colon X \to Y$ is a flopping contraction, 
and we have the flop 
\begin{align*}
	\xymatrix{
		X \ar[dr]_-{f} \ar@{.>}[rr]^-{\phi} 
		&  & X^+ \ar[ld]^-{f^{+}} \\
		& Y. &
	}
\end{align*}
The flopping contraction $f^+ \colon X^+ \to Y$ also satisfies 
Setting~\ref{setting}. 
By Bridgeland~\cite{Bri1} and Van den Bergh~\cite{VB}, 
there exists an equivalence
\begin{align}\label{equiv:flop}
\Upsilon \colon D^b \Coh(X^+) \stackrel{\sim}{\to}
D^b \Coh(X)
\end{align}
which restricts to an equivalence 
between 
$\oPPer(X^+/Y)$ and $\iPPer(X/Y)$.

We are mainly interested in perverse coherent sheaves which are supported on fibers of $f$. We define the following categories:
\begin{align*}
\Coh_{\leqslant 1}(X) &:=
\left\{E\in \Coh(X): \, \dim \Supp(E)\leqslant 1 \right\}, \\
\Coh_{\leqslant 1}(X/Y) &:=\left\{E\in \Coh_{\leqslant 1}(X): \, 
\dim \Supp(\dR f_*E)=0 \right\}, \\
D^b\Coh_{\leqslant 1}(X/Y) &:=\left\{E\in D^b\Coh(X): \,  \mathcal{H}^*(E)\in 
\Coh_{\leqslant 1}(X/Y) \right\}, \\
\Per_{\leqslant 1}(X/Y) &:=\Per(X/Y)\cap D^b\Coh_{\leqslant 1}(X/Y). 
\end{align*} 
We will use the following lemma: 
\begin{lem}\label{lem on supp on fibers}
We take $\beta \in H_2(X, \mathbb{Z})$
with $f_{\ast}\beta=0$. 
Then any $F\in \Per(X/Y)$ with $\ch(F)=(0,0,0, \beta, n)$ is supported on fibers of $f$, and hence 
$F\in \Per_{\leqslant 1}(X/Y)$.
\end{lem}
\begin{proof}
Let $\Phi$ be the equivalence in \eqref{psi equi cptX}. 
The object $\Phi(F)$ is given by 
\begin{align*}
\Phi(F)=\dR f_*(F)\oplus \dR f_*(F\otimes \mathcal{P}_0^{\vee}), 
\end{align*}
which is a coherent sheaf on $Y$. 
We claim that $\Phi(F)$ is zero dimensional. 
In fact, choose an ample divisor $H$ on $Y$. 
By the adjunction and Riemann-Roch formula, we have 
\begin{align*}
\chi_Y(\oO_Y,\dR f_*F\otimes \oO_Y(mH))&=\chi_X(\oO_X,F\otimes f^*\oO_Y(mH))=n,
\end{align*}
which is independent of $m$.
Therefore $\dR f_*F$ has a zero dimensional support. 
The same argument also shows that 
$\dR f_*(F\otimes \mathcal{P}_0^{\vee})$ has a zero 
dimensional support. 

Let $S=\{p_1,\cdots, p_n\}\subset Y$ be the set-theoretic support of $\Phi(F)$. 
Then we have $\Phi(F)|_{Y\setminus S}=0$. 
Since the derived equivalence $\Phi$ in (\ref{psi equi cptX})
is compatible with the base change to open subsets, we have 
\begin{align*}
\Phi(F|_{X\setminus f^{-1}(S)})=0. 
\end{align*}
Therefore $F|_{X\setminus f^{-1}(S)}=0$, 
i.e. $F$ is supported on fibers of $f$.
\end{proof}
Next we give another description of $\Per_{\leqslant 1}(X/Y)$ using tilting theory of Happel-Reiten-Smal\o.
Below we fix a $\mathbb{Q}$-ample divisor $\omega$ on $X$
which is degree one on fibers of $f|_{E} \colon E \to C$.  
For $F\in D^b\Coh_{\leqslant 1}(X/Y)$, we set
\begin{align*}
	d(F) \cneq [F] \cdot \omega \in \mathbb{Z}
\end{align*}
where $[F]$ is the fundamental one cycle of $F$. 
In other word, $d(F)$ is determined by 
$[F]=d(F)[\mathbb{P}^1]$
in $H_2(X, \mathbb{Z})$, 
where $[\mathbb{P}^1]$ is the fiber class of 
$f|_{E} \colon E \to C$. 
For $F \in \Coh_{\leqslant 1}(X/Y)$, its slope is defined to be 
\begin{align}\label{u:slope}
\mu_{\omega}(F):=\frac{\chi(F)}{d(F)} \in \mathbb{Q} \cup \{\infty\}
\end{align}
where we set $\mu_{\omega}(F)=\infty$ if $d(F)=0$. 
The above slope function defines $\mu_{\omega}$-semistable sheaves on 
$\Coh_{\leqslant 1}(X/Y)$ in the usual way. 
We define extension closed subcategories: 
\begin{align*}
\mathcal{T}_{\omega} &:=\langle F\in \Coh_{\leqslant 1}(X/Y): \, F\,\, \emph{is}\,\, \mu_{\omega}\emph{-semistable}\,\, \emph{with}\,\, \mu_{\omega}(F)> 0 \rangle_{\mathrm{ex}}, \\
\mathcal{F}_{\omega} &:=\langle F\in \Coh_{\leqslant 1}(X/Y): \, F\,\, \emph{is}\,\, \mu_{\omega}\emph{-semistable}\,\, \emph{with}\,\, \mu_{\omega}(F) \leqslant 0 \rangle_{\mathrm{ex}}.
\end{align*}
Here $\langle - \rangle_{\mathrm{ex}}$ is the extension closure. 
By the Harder-Narasimhan filtration, they form a torsion pair and we can define the tilting category in the sense of \cite{HRS}:
\begin{align*}
\Coh^{+}_{\leqslant 1}(X/Y):=\langle\mathcal{F}_{\omega}[1], \mathcal{T}_{\omega} \rangle_{\mathrm{ex}}.
\end{align*}
This is the heart of a bounded t-structure 
in $D^b \Coh_{\leqslant 1}(X/Y)$, 
and is in particular an abelian category \cite{BBD}.
\begin{prop}\label{prop on perv0}
As abelian subcategories of $D^b\Coh_{\leqslant 1}(X/Y)$, we have 
\begin{align*}
\Coh^{+}_{\leqslant 1}(X/Y)= \Per_{\leqslant 1}(X/Y). 
\end{align*}
\end{prop}
\begin{proof}
Note that any object in $\Coh_{\leqslant 1}(X/Y)$ is supported on 
points or fibers of $f|_{E} \colon E \to C$. 
By taking 
the Harder-Narasimhan and Jordan-H\"older filtrations, we have 
\begin{align*}
\mathcal{T}_{\omega} &=\langle \oO_x, \,\oO_{f^{-1}(c)}(a): x\in X, c\in C, 
a\geqslant 0   \rangle_{\mathrm{ex}}, \\
\mathcal{F}_{\omega} &=\langle \oO_{f^{-1}(c)}(a): 
c \in C, a<0
\rangle_{\mathrm{ex}}.
\end{align*}
It is straightforward to check $\mathcal{T}_{\omega}, \mathcal{F}_{\omega}[1]\subseteq \Per_{\leqslant 1}(X/Y)$,
therefore $\Coh^{+}_{\leqslant 1}(X/Y)\subseteq\Per_{\leqslant 1}(X/Y)$.   
Both sides are hearts of bounded 
t-structures
on $D^b \Coh_{\leqslant 1}(X/Y)$, they must be the same. 
\end{proof}

For $\Theta=(\theta_0, \theta_1) \in \mathbb{R}^2$
and $F \in D^b \Coh_{\leqslant 1}(X/Y)$, we set
\begin{align*}
\Theta(F) \cneq \theta_0  \cdot \chi(F)+\theta_1 \cdot 
(\chi(F)-d(F)) \in \mathbb{R}.
\end{align*}
Following Nagao-Nakajima \cite[Definition 2.1]{NN}, 
we introduce the notion of (semi)stable perverse coherent systems. Here we always consider the 
rank one case. 
\begin{defi}\label{pair stab}
A perverse coherent system is a pair 
\begin{align*}
(F,s), \ F\in \Per_{\leqslant 1}(X/Y), \ s \colon 
\oO_X\to F.
\end{align*}
For $\Theta=(\theta_0,\theta_1)\in \mathbb{R}^2$, a perverse coherent system $(F,s)$ is $\Theta$-(semi)stable if 
\begin{itemize}
\item for any non-zero subobject $0\neq F'\subseteq F$ in $\Per_{\leqslant 1}(X/Y)$, we have 
$\Theta(F') <(\leqslant)0$. 
\item for any proper subobject $F'\subsetneq F$ in $\Per_{\leqslant 1}(X/Y)$ such that $\mathrm{Im}(s)\subseteq F'$, we have 
$\Theta(F') <(\leqslant) \Theta(F)$. 
\end{itemize}
\end{defi}
Following arguments from \cite[Theorem 1.10]{NN}, \cite[Proposition 1.6.1]{Yos}, one can construct moduli spaces of such semistable perverse coherent systems.
\begin{thm}\label{cons of moduli}\emph{(\cite[Theorem 1.10]{NN}, \cite[Proposition 1.6.1]{Yos})} \\
Let $f\colon X\to Y$ be as in Setting \ref{setting}, and take 
$\beta\in H_2(X,\mathbb{Z})$ with $f_{\ast}\beta=0$, 
$n\in \mathbb{Z}$, 
$\Theta\in \mathbb{R}^2$. Then there is a projective coarse moduli scheme $P^{\Theta}_n(X,\beta)$ which parametrizes $S$-equivalence classes of $\Theta$-semistable perverse coherent systems $(F,s)$ with $\ch(F)=(0,0,0,\beta,n)$.
\end{thm}

\subsection{Perverse coherent sheaves on local model}
Let $X_0=\oO_{\mathbb{P}^1}(-1,-1,0)$ and 
take the affinization $f_0 \colon X_0 \to Y_0$, where $Y_0$ is 
given by 
\begin{align}\label{def:Y0}
	Y_0 =\Spec \mathbb{C}[x, y, z, w, u]/(xy-zw).
\end{align}
By the assumption in Setting~\ref{setting}, 
the morphism $f \colon X \to Y$ is 
identified with $f_0 \colon X_0 \to Y_0$ formally locally on $Y$. 
By~\cite{VB}, there is a projective generator for
$\Per(X_0/Y_0)$
given by 
\begin{align*}
	\pP_0 =\oO_{X_0} \oplus \oO_{X_0}(1),
\end{align*} 
and derived equivalences
\begin{align}\label{psi equi cpt}
	&\Phi_0 \colon D^b\Coh(X_0) \stackrel{\sim}{\to} D^b \modu (A_{Y_0}), \quad 
	(-) \mapsto \RHom(\pP_0, -), \\
	\notag
	&\Psi_0 \colon D^b \modu(A_{Y_0}) \stackrel{\sim}{\to} D^b \Coh(X_0), \quad
	(-) \mapsto (-) \dotimes_{A_{Y_0}} \pP_0,
\end{align}
which restrict to equivalences between 
$\Per(X_0/Y_0)$ and $\modu(A_{Y_0})$. 
Here $A_{Y_0}\cneq \End(\pP_0)$ is a non-commutative algebra.

The morphism $f_0 \colon X_0 \to Y_0$ is a flopping contraction, 
and we have the flop 
\begin{align*}
	\xymatrix{
		X_0 \ar[dr]_-{f_0} \ar@{.>}[rr]^-{\phi_0} 
		&  & X_0^+ \ar[ld]^-{f_0^{+}} \\
		& Y_0. &
	}
\end{align*}
The flop 
$X_0^+$ is also isomorphic to 
$\oO_{\mathbb{P}^1}(-1, -1, 0)$.  
By setting $\pP_0^+=\oO_{X_0^+} \oplus \oO_{X_0^+}(-1)$, we have 
a derived equivalence 
\begin{align*}
	\Phi_0^+=\RHom(\pP_0^+, -)
	\colon D^b \Coh(X_0^+) \stackrel{\sim}{\to}
	D^b \modu(A_{Y_0}). 
\end{align*}
Here we have used the isomorphism induced by the strict transform
\begin{align*}
	(\phi_{0})_{\ast} \colon 
	A_{Y_0}=\End(\pP_0) \stackrel{\cong}{\to} \End(\pP_0^+).
\end{align*}
By composing with the equivalence $\Psi_0$ in 
(\ref{psi equi cpt}), we obtain the flop equivalence
\begin{align}\label{fequiv}
	\Upsilon_0 \cneq \Psi_0 \circ \Phi_0^{+}
	\colon D^b \Coh(X_0^+) \stackrel{\sim}{\to}
	D^b \Coh(X_0),
\end{align}
giving a local model of the equivalence (\ref{equiv:flop}). 

The module category over the non-commutative algebra $A_{Y_0}$ is
described in terms of representations of 
a quiver with relations as follows. 
Let $(Q,I)$ be the following quiver with relations:
\begin{align}\label{cy4 quiver}
	\xymatrix{ \bullet_{\textbf{0}} \ar@(dl,ul)[]^{c}     \ar@/^1.5pc/[rr]^{a_1}  \ar@/^0.5pc/[rr]^{a_2} 
		&& \bullet_{\textbf{1}}\ar@/^0.5pc/[ll]^{b_1}  \ar@/^1.5pc/[ll]^{b_2}  \ar@(dr,ur)[]_{d} } \quad \,\,  \\  \nonumber
	a_2b_ia_1=a_1b_ia_2, \quad b_2a_ib_1=b_1a_ib_2,   \\  \nonumber
	da_i=a_ic, \quad cb_i=b_id, \quad  i=1,2.   \end{align}
\begin{lem}\label{quiver of local resolved coni}
	We have an 
	equivalence 
	\begin{align*}
		\modu(A_{Y_0}) \stackrel{\sim}{\to}
		\modu(\mathbb{C}{Q}/I). 
	\end{align*}
\end{lem}
\begin{proof}
	We write $\pi \colon \oO_{\mathbb{P}^1}(-1,-1,0)\to \mathbb{P}^1$ as the composition of projections
	$$\oO_{\mathbb{P}^1}(-1,-1,0)\stackrel{\pi_1}{\to} \oO_{\mathbb{P}^1}(-1,-1) \stackrel{\pi_2}{\to}  \mathbb{P}^1. $$
	Since $\pP_0=\pi_1^{\ast}\eE$ for  $\eE=\pi_2^*\oO_{\mathbb{P}^1}\oplus\pi_2^*\oO_{\mathbb{P}^1}(1)$, we have 
	\begin{align*}
		\Hom(\mathcal{P}_0,\mathcal{P}_0)
		&\cong \Hom(\eE,\eE\otimes \pi_{1*}\oO) \\
		&\cong \Hom(\eE,\eE)\otimes \mathbb{C}[t]. \end{align*}
	We write $B:=\Hom(\eE,\eE)$. By the above isomorphism, 
	an $A_{Y_0}$-module $M$ can be viewed as a $B$-module, which is an representation of the following 
	quiver with relations (see~\cite[\S 2.1]{Sze}):
	\begin{align*} 
		\xymatrix{ \bullet_{\textbf{0}}      \ar@/^1.5pc/[rr]^{a_1}  \ar@/^0.5pc/[rr]^{a_2} 
			&& \bullet_{\textbf{1}}\ar@/^0.5pc/[ll]^{b_1}  \ar@/^1.5pc/[ll]^{b_2}    } \quad \quad \quad \quad \\  
		a_2b_ia_1=a_1b_ia_2, \quad b_2a_ib_1=b_1a_ib_2.   \end{align*}
	Based on the action of idempotent elements at vertex $0$ and $1$, we can write 
	\begin{align*}
		M=M_0\oplus M_1. 
	\end{align*}
	The $\mathbb{C}[t]$-module structure on $M$
	gives an action on $M$
	\begin{align*}
		\times t \colon M_i\to M_i, \quad i=0,1, 
	\end{align*}
	which we denote by loops $c,d$ in \eqref{cy4 quiver}.
	This action commutes with $B$-module action, so we have commutative relations $da_i=a_ic$, $cb_i=b_id$.
\end{proof}

For $\Theta=(\theta_0, \theta_1) \in \mathbb{R}^2$
and a finite dimensional representation $V=(V_0, V_1)$ of the quiver 
(\ref{cy4 quiver}),
we denote
\begin{align*}
	\Theta(V) \cneq 
	\theta_0 \dim V_0+ \theta_1 \dim V_1, \quad
	\mu_{\Theta}(V) \cneq \frac{\theta_0 \dim V_0+\theta_1 \dim V_1}{\dim V_0+\dim V_1}.
\end{align*} 
The $\Theta$-stability for representations of 
the quiver (\ref{cy4 quiver}) is defined as follows: 
\begin{defi}
A finite dimensional representation $V$ of the quiver (\ref{cy4 quiver})
is $\Theta$-(semi)stable if for any 
subrepresentation $0\neq V' \subsetneq V$ we have 
$\mu_{\Theta}(V')<(\leqslant) \mu_{\Theta}(V)$. 
\end{defi}
A simple extension of Lemma~\ref{quiver of local resolved coni} (e.g. \cite[Proposition 3.3]{NN}) shows that 
perverse coherent systems on $X_0$ are in one-to-one correspondence with representations 
of the \textit{framed} quivers $\widetilde{Q}$ with relations $I$:
\begin{align}\label{framed cy4 quiver}
\xymatrix{\bullet^{\mathbf{ \infty}} \ar[d]  &&  \\   
 \bullet_{\textbf{0}} \ar@(dl,ul)[]^{c}     \ar@/^1.5pc/[rr]^{a_1}  \ar@/^0.5pc/[rr]^{a_2} 
&& \bullet_{\textbf{1}}\ar@/^0.5pc/[ll]^{b_1}  \ar@/^1.5pc/[ll]^{b_2}  \ar@(dr,ur)[]_{d} } \quad \,\,  \\  \nonumber
 a_2b_ia_1=a_1b_ia_2, \quad b_2a_ib_1=b_1a_ib_2,   \\  \nonumber
 da_i=a_ic, \quad cb_i=b_id, \quad  i=1,2.   \end{align}
More specifically, for a perverse coherent system $(F,s)$, we have the associated 
vector spaces 
\begin{align*}
V_0=\Hom(\oO_{X_0}, F), \quad V_1=
\Hom(\oO_{X_0}(1), F),
\end{align*}
at the vertex $0$ and $1$.
When $F$ has compact support, we set 
$d(F) \cneq \rank(\pi_{\ast}F)$ where 
$\pi \colon X_0 \to \mathbb{P}^1$ is the projection. 
In this case, we have 
\begin{align*}
(\dim V_0, \dim V_1)=(\chi(F), \chi(F)-d(F)).
\end{align*}
Below when we consider a perverse coherent system $(F, s)$ on $X_0$, we always 
assume that $F$ is compactly supported. 
The stability of perverse coherent systems 
in Definition \ref{pair stab} for 
$X_0$ 
translates into the following King's stability~\cite{King} of representations of $(\widetilde{Q},I)$.
\begin{defi}\label{def:stability:framed}
	A representation $(V_0,V_1, V_{\infty}=\mathbb{C})$ of $(\widetilde{Q},I)$ is $\Theta$-(semi)stable if 
	\begin{itemize}
		\item for any non-zero subrepresentation 
		$(V_0',V_1', 0)$, we have 
		$\Theta(V_0', V_1') <(\leqslant) 0$. 
		\item for any proper subrepresentation
		 $(V_0',V_1',V_{\infty}'=\mathbb{C})$, we have 
		 $\Theta(V_0', V_1') <(\leqslant) \Theta(V_0, V_1)$. 
	\end{itemize}
\end{defi}


\subsection{Wall-chamber structures for local resolved conifold}\label{sect on wall on local conifold}
In this section, we study wall-chamber structures for moduli spaces of (compactly supported) stable perverse coherent systems on 
the following local model: 
\begin{align*}
f_0 \colon  X_0=\oO_{\mathbb{P}^1}(-1,-1,0)\to Y_0=\big\{(x,y,z,w)\in\mathbb{C}^4: \, xy=zw\big\}\times \mathbb{C}.
 \end{align*}
As mentioned above, we are reduced to study the wall-chamber structures for stability of finite dimensional representations 
of \eqref{framed cy4 quiver}. 

In order to 
classify all walls, we need to find out $\Theta$ such that there exists a strictly $\Theta$-semistable representation\,\footnote{Here $V_{\infty}=\mathbb{C}$ as we are only interested in rank one perverse coherent systems (see Definition \ref{pair stab}).}
\begin{align*}
\widetilde{V}=(V_0,V_1,V_{\infty}=\mathbb{C})
\end{align*} of the framed quiver
$(\widetilde{Q}, I)$ in \eqref{framed cy4 quiver}. 
We have the Jordan-H\"older filtration:
\begin{align*}
\widetilde{V}=\widetilde{V}^0\supset \widetilde{V}^1\supset \cdots \supset \widetilde{V}^l=0, \quad l\geqslant 2,
\end{align*}
such that $\widetilde{V}^i/\widetilde{V}^{i+1}$'s are $\Theta$-stable representations of $(\widetilde{Q}, I)$.
Since $\dim V_{\infty}=1$, there must be some $\widetilde{V}^i/\widetilde{V}^{i+1}$ which has zero dimension vector
at the vertex $\infty$ (the condition $l\geqslant2$ is crucial here), i.e. $\widetilde{V}^i/\widetilde{V}^{i+1}$ is a  finite dimensional $\Theta$-stable representation of the unframed quiver $(Q, I)$ \eqref{cy4 quiver}, 
satisfying $\Theta(\widetilde{V}^i/\widetilde{V}^{i+1})=0.\,$\footnote{Based on Definition \ref{def:stability:framed}, equivalently Definition \ref{pair stab}, the slope 
function of $\widetilde{V}$ is assumed to be zero (compared with \cite[\S 1.3]{NN}). The equality follows from the slope property of Jordan-H\"older filtration of $\widetilde{V}$.}
Therefore we see that a strictly $\Theta$-semistable representation of  \eqref{framed cy4 quiver} produces a  
$\Theta$-stable representation of the unframed quiver \eqref{cy4 quiver}.
So in order to classify all walls, we are reduced to classify all $\Theta$ such that there exists a non-zero finite dimensional $\Theta$-stable representations of $(Q, I)$.
The following lemma is proved 
along with the similar argument of~\cite[Lemma 3.4]{NN}. Here we recall the key point to make us self-contained. 

\begin{lem}\label{lem on wall}
Let $V=(V_0,V_1)$ be a non-zero finite dimensional $\Theta$-stable representation of quiver \eqref{cy4 quiver}. 
Then one of the following conditions 
hold: 
\begin{enumerate}
	\renewcommand{\labelenumi}{(\arabic{enumi})}
	\item $\dim V_0=\dim V_1=1$, 
	\item $a_1=a_2=0$, 
	\item $b_1=b_2=0$. 
	\end{enumerate}
\end{lem}
\begin{proof}
By replacing $\Theta=(\theta_0, \theta_1)$ with 
$(\theta_0-\mu_{\Theta}(V), \theta_1-\mu_{\Theta}(V))$, we may assume  
$\Theta(V)=0$. 
We may also 
assume that
 $V_0,V_1\neq 0$, as otherwise (2) or (3) holds trivially. 
 For a fixed $(i, j)$, we set 
\begin{align*}
S_0=\Ker(b_ja_i), \quad S_1=\Ker(a_ib_j), \quad T_0=\mathrm{Im}(b_ja_i), \quad
T_1=\mathrm{Im}(a_ib_j). 
\end{align*}
It is easy to check that $(S_0,S_1)$ and $(T_0,T_1)$ are subrepresentations of $V$. 
 Therefore the $\Theta$-stability of $V$ implies 
$$\theta_0\dim S_0+\theta_1\dim S_1\leqslant 0, \quad \theta_0\dim T_0+\theta_1\dim T_1\leqslant 0. $$
These two inequalities must be equalities as we also have 
$$\theta_0\dim(V_0)+\theta_1\dim(V_1)=0, \quad \dim S_i+\dim T_i=\dim V_i, \,\,\, (i=0,1). $$
So either $(S_0,S_1)=(0,0)$ or $(S_0,S_1)=(V_0,V_1)$ holds. Note that if the first case happens, then $a_i$ and $b_j$ are injective, hence
they give isomorphisms $V_0 \cong V_1$. 
If the second case happens, then $b_ja_i=a_ib_j=0$.
%

From the above argument, we may assume that either one of the followings holds:
\begin{enumerate}
	\renewcommand{\labelenumi}{(\Alph{enumi})}
	\item $a_i b_j=b_j a_i=0$ for all $i$, $j$, or 
	\item $a_1$, $b_1$ are isomorphisms with 
	$\dim V_0=\dim V_1=1$.
	\end{enumerate}
In the case of (A), we first assume $\theta_0\geqslant 0$. 
By taking the subrepresentation 
$(\Ker(a_1)\cap\Ker(a_2),0)$ of $V$, 
the $\Theta$-stability yields
 $\Ker(a_1)\cap\Ker(a_2)=0$. Using $a_ib_j=0$, we have 
$$\mathrm{Im}(b_1),\,\, \mathrm{Im}(b_2) \subseteq \Ker(a_1)\cap\Ker(a_2)=0, $$
therefore $b_1=b_2=0$. Similarly when $\theta_0\leqslant  0$, i.e. $\theta_1\geqslant 0$, we conclude that $a_1=a_2=0$.

In the case of (B), note that $b_1a_1, b_2a_1,b_1a_2, b_2a_2, c$ are pairwise commutating.
Let us take a common eigenvector $0\neq v_0\in V_0$. Then 
$$
(S_0', S_1')=(\langle v_0 \rangle, \langle 
 a_1(v_0), a_2(v_0) \rangle ) $$
is a subrepresentation of $V$.
Here $\langle - \rangle$ is the linear span of $-$. 
 By stability, we obtain the inequality 
$$\theta_0\dim S'_0+\theta_1\dim S'_1\leqslant 0. $$
Note that we have $\theta_0+\theta_1=0$. If $\theta_0<0$, the above inequality is equivalent to  
$\dim S'_1\leqslant \dim S'_0=1$. 
Since $a_1$ is an isomorphism,
we have $a_1(v_0)\neq 0$.
Therefore the equality holds and 
$$(S'_0,S'_1)=(V_0,V_1), \quad \dim V_0=\dim V_1=1. $$
If $\theta_0>0$ (i.e. $\theta_1<0$), by symmetry between vertex $0$ and $1$ and considering pairwise commutating paths 
$a_1b_1, a_1b_2, a_2b_1, a_2b_2, d$, we obtain the same conclusion.
\end{proof}

For $a \in \mathbb{C}$, we denote by $j_a$ the 
closed immersion 
\begin{align*}
j_a \colon
\mathbb{P}^1 \hookrightarrow
 \oO_{\mathbb{P}^1}(-1,-1)\times \{a\} \hookrightarrow
 \oO_{\mathbb{P}^1}(-1,-1,0), 
\end{align*}
where the first arrow is the zero section. 
We have the following classification of $\Theta$-stable 
representations. 

\begin{prop}\label{classify:unframed}
A non-zero finite dimensional representation $V$ of quiver
 \eqref{cy4 quiver} is $\Theta$-stable if  
it is either one of the following: 
\begin{enumerate}
\item $\Phi_0(j_{a\ast}\oO_{\mathbb{P}^1}(m-1))$
for $a \in \mathbb{C}$ and $m\geqslant 1$, 
\item $\Phi_0(j_{a\ast}\oO_{\mathbb{P}^1}(-m-1)[1])$
for $a \in \mathbb{C}$ and $m\geqslant 0$,
\item $\Phi_0(\oO_x)$
for $x \in X_0$, 
\item $\Phi_0^+(j_{a\ast}\oO_{\mathbb{P}^1}(-m-1)[1])$
for $a \in \mathbb{C}$ and $m\geqslant 1$, 
\item $\Phi_0^+(j_{a\ast}\oO_{\mathbb{P}^1}(m-1))$
for $a \in \mathbb{C}$ and $m\geqslant 0$, 
\item $\Phi_0^+(\oO_x)$ for $x \in X_0^+$.
\end{enumerate}
\end{prop}
\begin{proof}
Suppose that $V$ satisfies (2) of Lemma \ref{lem on wall}.
Then $V$ is a representation of the following quiver with relation: 
\begin{align}\label{quiver of P1C} 
\xymatrix{ \bullet_{\textbf{0}} \ar@(dl,ul)[]^{c}    
&& \bullet_{\textbf{1}}  \ar@/^0.5pc/[ll]_{b_2}  \ar@/_0.5pc/[ll]_{b_1}  \ar@(dr,ur)[]_{d} }   \\   \nonumber
 cb_i=b_id, \,\,   i=1,2. \quad \quad   \end{align}
Geometrically, this quiver corresponds to $\mathbb{P}^1\times \mathbb{C}$ from the tilting bundle. More precisely, let $\eE:=\oO_{\mathbb{P}^1}\oplus \oO_{\mathbb{P}^1}(1)$ be the tilting bundle of $\mathbb{P}^1$ whose endomorphism algebra 
$K:=\End(\eE) $
gives rise to the Kronecker quiver:  
\begin{align*} 
\xymatrix{ \bullet_{\textbf{0}}  
&& \bullet_{\textbf{1}}  \ar@/^0.5pc/[ll]_{b_2}  \ar@/_0.5pc/[ll]_{b_1}   }    \end{align*}
Then the endomorphism algebra $L:=\End(\pi^*\eE)$ gives rise to the quiver \eqref{quiver of P1C},
where $\pi \colon \mathbb{P}^1\times \mathbb{C}\to \mathbb{P}^1$ denotes the projection. 
By the projection formula, we have 
\begin{align*}
L &\cong \Hom(\eE,\eE\otimes \pi_*\oO_{\mathbb{P}^1\times \mathbb{C}}) \\
&\cong K\otimes \mathbb{C}[t]. \end{align*}
From the natural embedding
$\mathbb{C}[t]\to L$, 
we can treat $L$-module $V$ as a $\mathbb{C}[t]$-module. 
Since $V$ is a stable $L$-module, we have 
$\End(V)\cong \mathbb{C}$. 
Therefore
 as a $\mathbb{C}[t]$-module, we have  
\begin{align*}
\End(V)\cong \mathbb{C}[t]/(t-a),
\end{align*}
for some $a\in \mathbb{C}$. 
In particular the $L$-module structure on $V$
descends to the $L/(t-a)\cong K$-module structure, 
so we can also treat $V$ as a stable $K$-module. 

By \cite[Lemma 2.12]{NY}, we can classify all $\Theta$-stable $K$-modules, which under the equivalence
\begin{align*}
\RHom(\eE, -) \colon 
 D^b\Coh(\mathbb{P}^1)
\stackrel{\sim}{\to} D^b \modu(K)
\end{align*}
correspond to 
$\oO_{\mathbb{P}^1}(m-1)$ for $m\geqslant 1$
or $\oO_{\mathbb{P}^1}(-m-1)[1]$ for $m\geqslant 0$.
Now we view such $K$-modules as representations of the quiver 
\eqref{cy4 quiver}.
Then similarly to \cite[Remark 3.6]{NN},
via the equivalence (\ref{psi equi cpt})
 they correspond to 
the following objects
in $D^b \Coh(X_0)$
\begin{align}\label{stb obj}j_{a\ast}\oO_{\mathbb{P}^1}(m-1) \,\, (m\geqslant 1), \quad j_{a\ast}\oO_{\mathbb{P}^1}(-m-1)[1] \,\, (m\geqslant 0). 
 \end{align}
Therefore $V$ is either of type (i) or (ii) in the proposition. 

If $V$ satisfies (3) of Lemma \ref{lem on wall}, as 
in \cite[Remark 3.6]{NN} it corresponds to one of the geometric objects
\eqref{stb obj} in the flop side. 
So $V$ is either of type (iv) or (v) in the proposition. 
Finally if $V$ satisfies (1) of Lemma \ref{lem on wall},  
a similar argument as above shows that we can treat 
$V$
as a stable representation of the quiver associated with the resolved conifold $\oO_{\mathbb{P}^1}(-1,-1)$, unique 
up to a choice of $a\in \mathbb{C}$. Combining with \cite[Remark 3.6]{NN}, we know it 
corresponds to a structure sheaf of a point in $X_0$ or that of a
 flop $X_0^+$.
So $V$ is either of type (iii) or (vi) in the proposition. 
\end{proof}
To sum up, walls for $\Theta$-stability of representations 
of the framed quiver \eqref{framed cy4 quiver} can be classified as follows:
\begin{prop}\label{prop:wall}
For each $\Theta$-stable representation in 
Proposition~\ref{classify:unframed}, 
the corresponding wall is given as follows\,\footnote{Here we use different notations for walls compared with \cite[pp.~18]{NN}. Our $L^{-}_{-}(m), L^{+}_{-}(m), L^{-}_{+}(m), L^{+}_{+}(m)$
are those $L^{-}_{+}(m), L^{-}_{-}(m), L^{+}_{+}(m), L^{+}_{-}(m)$ in \cite{NN}. }: 
\begin{align}\label{all walls}
	L^{-}_{-}(m)&:=\big\{(\theta_0,\theta_1)\in \mathbb{R}^2:\,\theta_0<\theta_1,\,\, m\theta_0+(m-1)\theta_1=0 \big\}, \quad (m\geqslant 1), \\ \nonumber
	L^{+}_{-}(m)&:=\big\{(\theta_0,\theta_1)\in \mathbb{R}^2:\,\theta_0<\theta_1,\,\, m\theta_0+(m+1)\theta_1=0 \big\}, \quad (m\geqslant 0), \\
	\nonumber
	L_{-}(\infty)&:=\big\{(\theta_0,\theta_1)\in \mathbb{R}^2:\,\theta_0<\theta_1,\,\, \theta_0+\theta_1=0 \big\},  \\ \nonumber
	L^{-}_{+}(m)&:=\big\{(\theta_0,\theta_1)\in \mathbb{R}^2:\,\theta_0>\theta_1,\,\, m\theta_0+(m-1)\theta_1=0 \big\}, \quad (m\geqslant 1), \\\nonumber
	L^{+}_{+}(m)&:=\big\{(\theta_0,\theta_1)\in \mathbb{R}^2:\,\theta_0>\theta_1,\,\, m\theta_0+(m+1)\theta_1=0 \big\}, \quad (m\geqslant 0), \\\nonumber
	L_{+}(\infty)&:=\big\{(\theta_0,\theta_1)\in \mathbb{R}^2:\,\theta_0>\theta_1,\,\, \theta_0+\theta_1=0 \big\}. 
\end{align}
\begin{proof}
	Let $V=(V_0, V_1)$ be a 
	$\Theta$-stable representation
	(i) in Proposition~\ref{classify:unframed}.
	Then it has the dimension vector
	$(m,m-1)$, so 
	the condition $\Theta(V)=0$ 
	yields the equation of the wall 
	$m\theta_0+(m-1)\theta_1=0$. 
	Since it has a subrepresentation $(V_0, 0)$, when $m>1$,
	it is $\Theta$-stable 
	only if the inequality $m \theta_0<0$
	holds. So $\theta_0<0$ and $\theta_1> 0$ follows from the defining equation of the wall. This gives rise 
	to the wall $L_{-}^-(m)$ when $m>1$. When $m=1$, the equation of wall is $\theta_0=0$ and we simply put $\theta_1> 0$ to define $L_{-}^-(1)$.
	Similarly, the other stable representations 
	in Proposition~\ref{classify:unframed} 
	give rise to other walls 
	in (\ref{all walls}). 
	\end{proof}
\end{prop}

A connected component of the complement of walls in $\mathbb{R}^2$
is called a \textit{chamber}. 
Below we discuss some distinguished chambers. 
When $\theta_0,\theta_1>0$, we call this chamber 
the \textit{empty chamber}:
\begin{prop}\label{empty chamber}
When $\theta_0,\theta_1>0$, there is no non-zero finite dimensional $\Theta$-stable representation of the framed quiver \eqref{framed cy4 quiver}.
\end{prop}
\begin{proof}
Assume that 
there is such a
representation $V=(V_0,V_1,V_{\infty}=\mathbb{C})$. 
By considering the sub-representation $(V_0,V_1,0)$, we obtain 
$
\theta_0\dim V_0+\theta_1\dim V_1<0$, 
which contradicts with the assumption.
\end{proof}

The chambers adjacent to the wall $L_{-}(\infty)$
are the so-called \textit{DT/PT chambers}. Following the proof \cite[Proposition 2.10, 2.11]{NN} in the resolved conifold case, it is 
easy to see: 
\begin{prop}\label{dt/pt chamber}
Let $\Theta^{\pm}=(-1\,\mp \,0^+,1)$. Then under the derived equivalences
 in (\ref{psi equi cpt}), we have the following: 
\begin{itemize}
\item finite dimensional $\Theta^{+}$-stable representations of the framed quiver \eqref{framed cy4 quiver} correspond exactly 
to ideal sheaves of compactly supported subschemes in $X_0$,
\item finite dimensional $\Theta^{-}$-stable representations of the framed quiver \eqref{framed cy4 quiver} correspond exactly 
to PT stable pairs $(\oO_{X_0}\to F)$, i.e. $F$ is compactly supported pure one dimensional sheaf and $\mathrm{Coker}(s)$ is zero dimensional.
\end{itemize}
\end{prop}
\begin{rmk}
Similarly for the wall $L_{+}(\infty)$, finite dimensional $(-\Theta^{\pm})$-stable representations correspond to
those objects in the flop $X_0^+$ of $X_0$.  
\end{rmk}
When $\theta_0,\theta_1<0$, we are in the \textit{non-commutative} chambers, where stable representations correspond to \textit{perverse Hilbert schemes} 
in the sense of Bridgeland \cite{Bri1}. 
\begin{prop}\label{nc chamber}
When $\theta_0,\theta_1<0$, finite dimensional 
$\Theta$-stable representations of the framed quiver \eqref{framed cy4 quiver} are exactly 
those cyclic representations, i.e. representations generated by $V_{\infty}=\mathbb{C}$ as $\mathbb{C}\widetilde{Q}/I$-modules.
\end{prop}
\begin{proof}
Let $V=(V_0, V_1, V_{\infty}=\mathbb{C})$ 
be a $\Theta$-stable representation of $(\widetilde{Q}, I)$. 
Suppose it has a subrepresentation 
$V'=(V'_0,V'_1, V_{\infty}'=\mathbb{C})$, 
the $\Theta$-stability yields
\begin{align*}
\theta_0\dim V'_0+\theta_1\dim V'_1
\leqslant 
\theta_0 \dim V_0+\theta_1 \dim V_1. 
\end{align*}
However $\theta_0,\theta_1<0$, so the above inequality must be equality and $V'_i=V_i$ for $i=0,1$.

Conversely, let us take a cyclic representation $V=(V_0,V_1, V_{\infty}=\mathbb{C})$. Then a non-zero proper sub-representation of it must be of the form 
$(V'_0,V'_1,0)$, and we have 
$$\theta_0\dim V'_0+\theta_1\dim V'_1<0, $$
by the condition $\theta_0,\theta_1<0$, so $V$ is $\Theta$-stable.
\end{proof}

In Section~\ref{subsec:Zt}, we will discuss chambers in the region 
$\theta_0<0$, $\theta_1>0$
in details, 
and show that they are in one-to-one correspondence with  
chambers for \textit{$Z_t$-stable pairs} introduced 
in \cite[Definition 1.5]{CT1}
(see Proposition~\ref{Z_t chamber}). 

\subsection{Classification of stable framed representations with $\dim V_0=1$}
In this subsection, we classify 
$\Theta$-stable representations 
$V=(V_0, V_1, V_{\infty}=\mathbb{C})$ 
of the framed quiver \eqref{framed cy4 quiver}
such that $V_{\infty} \to V_0$ is an isomorphism
and $V_1 \neq \mathbb{C}$. This will be used in the proof of our main theorem \ref{cpt main thm}.
Note that by Proposition~\ref{empty chamber}, we can assume 
$\theta_0<0$ or $\theta_1<0$.
\begin{figure}
	\begin{tikzpicture}[node distance=1cm]
		\draw[thick] (-4.2,0)--(4.2,0)  node [pos=0, anchor=east]{\tiny{$\theta_1=0$}} ;
		\draw[thick](0,3.2)--(0,-3.2)  node [pos=0, anchor=east]{\tiny{$\theta_0=0$}} ;
		\draw[thick] (-4,2)--(4.2,-2.1)  node[pos=0, anchor=east]{\tiny{$\theta_0+2\theta_1=0$} } ;
		\node at (-1.7,2.2) {I};
		\node at (-3,0.7) {I\hspace{-.1em}I};
		\node at (-2,-2) {I\hspace{-.1em}I\hspace{-.1em}I } ;
		\node at (1.7,-2) {I\hspace{-.1em}V};
		\node at (3, -0.7) {V};
	\end{tikzpicture} 
	\caption{Chambers in Proposition~\ref{prop:classify}}
	\label{figure3}
\end{figure}
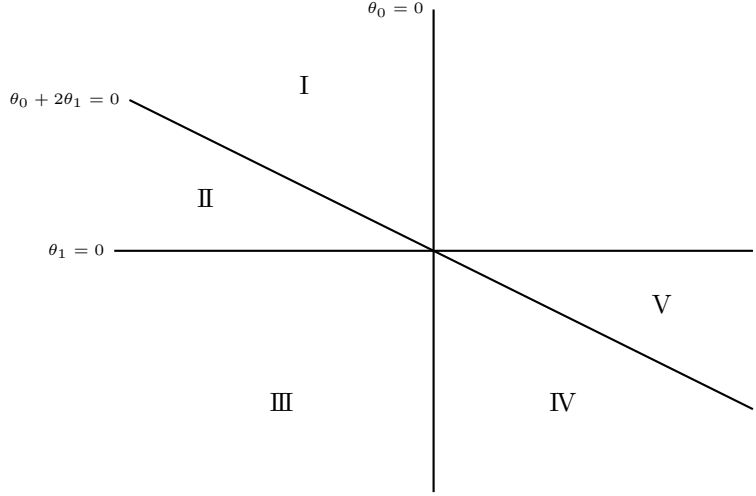

\begin{prop}\label{prop:classify}
\label{prop:onedim}
Let
$V=(V_0, V_1, V_{\infty}=\mathbb{C})$ 
be a representation of the framed quiver \eqref{framed cy4 quiver}
such that
$V_{\infty} \to V_0$ is an isomorphism and $V_1 \neq \mathbb{C}$. 
Let $\Theta$ lie in one of the chambers in Figure~\ref{figure3}. 
Then it is $\Theta$-stable if and only if $(V_0, V_1)$ is the following
for some $a \in \mathbb{C}$:
\begin{align}\label{classify:V}
(V_0, V_1)= 
\begin{cases}
\Phi_0(j_{a\ast}\oO_{\mathbb{P}^1}), & 
	\mbox{ if } \Theta \in \rm{I}, \\
\Phi_0(j_{a\ast}\oO_{\mathbb{P}^1}) \mbox{ or } 
	\Phi_0(j_{a\ast}\oO_{\mathbb{P}^1}(-2)[1]), & 
	\mbox{ if } \Theta \in \rm{I\hspace{-.1em}I} \\
\Phi_0(j_{a\ast}\oO_{\mathbb{P}^1}) \mbox{ or }
	\Phi_0^+(j_{a\ast}\oO_{\mathbb{P}^1}), & 
	\mbox{ if } \Theta \in \rm{I\hspace{-.1em}I\hspace{-.1em}I}, \\
\Phi_0^+(j_{a\ast}\oO_{\mathbb{P}^1}), 
	& 
	\mbox{ if } \Theta \in \rm{I\hspace{-.1em}V}.  	
	\end{cases}
\end{align}
And there is no such $(V_0, V_1)$ if $\Theta \in \rm{V}$. 
\end{prop}
\begin{proof}
If $(V_0, V_1)$ is of the form described in (\ref{classify:V}), 
then $\dim V_0=1$. 
By taking the $(\widetilde{Q}, I)$-representation 
$V=(V_0, V_1, V_{\infty}=\mathbb{C})$
so that $V_{\infty}\to V_0$ is an isomorphism, 
it 
is straightforward to check that $V$ is $\Theta$-stable. 

In what follows, we show the converse direction, i.e. 
if $V=(V_0, V_1, V_{\infty}=\mathbb{C})$ satisfies 
the assumption of the proposition, then 
$(V_0, V_1)$ must be of the form described in (\ref{classify:V}). 
If $V_1=0$, then $(V_0, V_1)=\Phi_0(j_{a\ast}\oO_{\mathbb{P}^1})$
for some $a \in \mathbb{C}$. 
In this case, $V$ is $\Theta$-stable if and only if 
$\Theta(V_0, V_1)=\theta_0<0$. 
Below we may assume that $\dim V_1>0$. 
Note that $(V_0, V_1, 0)$ is a subrepresentation of $V$,
so  
the $\Theta$-stability yields
\begin{align}\label{ineq:theta}
\theta_0+\theta_1 \dim V_1<0. 	
\end{align}
\begin{ccase}
	$\theta_0<0$, $\theta_1>0$. 
\end{ccase}	
In this case for the sub-representation $(0,\mathrm{Ker}(b_1)\cap \mathrm{Ker}(b_2),0)$ of $V$, 
the $\Theta$-stability yields
$$\dim(\mathrm{Ker}(b_1)\cap \mathrm{Ker}(b_2))\cdot \theta_1\leqslant 0. $$
Since $\theta_1>0$, we obtain $\mathrm{Ker}(b_1)\cap \mathrm{Ker}(b_2)=0$.
If 
	$\mathrm{Ker}(b_1)=0$ or $\mathrm{Ker}(b_2)=0$, 
	then $\dim V_1=1$
	so we can assume that 
	$\mathrm{Ker}(b_1) \neq 0$ and $\mathrm{Ker}(b_2) \neq 0$. 
As 
$\mathrm{Ker}(b_1)\cap \mathrm{Ker}(b_2)=0$, we have  
\begin{align}\label{span by two vec}\mathrm{Ker}(b_1)\oplus \mathrm{Ker}(b_2) \subseteq V_1.  \end{align}
Since $\mathrm{Im}(b_i)=V_0$ and it is one dimensional, 
we have 
$$\dim \mathrm{Ker}(b_i)=\dim V_1-1. $$
Combined with \eqref{span by two vec}, we obtain $\dim V_1=2$. 
If this happens, $\Theta$
must satisfy $\theta_0+2\theta_1<0$ 
by (\ref{ineq:theta}). 

We show that $(V_0, V_1)$ is $\Theta$-stable. 
Since $\Ker(b_1) \cap \Ker(b_2)=0$, 
the only possible non-zero proper sub-representations of $(V_0,V_1)$ are
either $(V_0,0)$ or $(V_0,\mathbb{C})$. 
By the inequalities
\begin{align*}\theta_0< \frac{\theta_0+2\theta_1}{3}, \quad \frac{\theta_0+\theta_1}{2}<\frac{\theta_0+2\theta_1}{3}, 
	\end{align*}
we conclude that $(V_0, V_1)$ is $\Theta$-stable. 
By Proposition~\ref{classify:unframed}
together with $(b_1, b_2) \neq (0, 0)$, 
we conclude that
$(V_0, V_1)$ is
$\Phi_0(j_{a\ast}\oO_{\mathbb{P}^1}(-2)[1])$
for some $a \in \mathbb{C}$. 
Therefore we proved the case of $\Theta \in \rm{I}$ or $\Theta \in \rm{I\hspace{-.1em}I}$. 

\begin{ccase}
$\theta_1<0$.
\end{ccase}
In this case, we set 
$$V_1':=\langle \Imm a_1, \Imm a_2 \rangle \subseteq V_1,$$
to be the spanned vector subspace.
Using the relation $da_i=a_ic$ 
of the quiver (\ref{cy4 quiver}), 
we have $d(V_1')\subseteq V_1'$.
Therefore $(V_0, V_1', V_{\infty})$
is a subrepresentation of $V$. 
The $\Theta$-stability yields
\begin{align*}
	\theta_0+\theta_1 \dim V_1' \leqslant \theta_0+\theta_1 \dim V_1.
\end{align*}
Therefore we have $V_1'=V_1$, hence 
$\dim V_1 \leqslant 2$. 
So $\dim V_1=2$ and by (\ref{ineq:theta})
this is possible when 
$\theta_0+2\theta_1<0$. 
Note that $V_1'=V_1$ implies that 
$(a_1, a_2) \colon V_0^{\oplus 2} \to V_1$
is an isomorphism. 
So if we also take $\theta_0>0$ then 
it is easy to see that $(V_0, V_1)$ is a
$\Theta$-stable representation of the unframed quiver 
$(Q, I)$ in \eqref{cy4 quiver}.
By Proposition~\ref{classify:unframed} together with
$(a_1, a_2) \neq (0, 0)$, 
we conclude that $(V_0, V_1)$ is $\Phi_0^+(j_{a\ast}\oO_{\mathbb{P}^1})$
for some $a \in \mathbb{C}$. 
Therefore (\ref{classify:V}) holds when 
$\Theta \in \rm{I\hspace{-.1em}I\hspace{-.1em}I}$, 
$\Theta \in \rm{I\hspace{-.1em}V}$, and 
there is no such $(V_0, V_1)$ when $\Theta \in \rm{V}$. 
\end{proof}

\subsection{Counting invariants}
We go back to the compact setting where 
$$f \colon X\to Y$$ 
is the contraction as in Setting \ref{setting}. 
Consider the coarse moduli space $P^{\Theta}_n(X,\beta)$ of $\Theta$-semistable perverse coherent systems $(F,s)$ with 
$\ch(F)=(0,0,0,\beta,n)$ and $f_*\beta=0$. 
The set of walls for this moduli space coincides with that studied in Section \ref{sect on wall on local conifold}.
\begin{lem}\label{lem identify wall on cpt and local}
Suppose that $\Theta$ lies outside all walls defined in \eqref{all walls}. Then $P^{\Theta}_n(X,\beta)$ with $f_*\beta=0$
depends only on the connected components where $\Theta$ locates.
\end{lem}
\begin{proof}
The argument is the same as in the beginning of Section \ref{sect on wall on local conifold}. 
A wall appears if there exists a strictly $\Theta$-semistable perverse coherent system. By taking the 
Jordan-H\"older filtration, 
there exists a $\Theta$-stable perverse coherent sheaf $V$ such that $\Theta(V)=0$. 
As $\End(V)=\mathbb{C}$, the support of $V$ is connected.
Therefore
by Lemma \ref{lem on supp on fibers},
the support of $V$ is contained in a fiber of $E \to C$. 
By the assumption 
of Setting~\ref{setting}, it sits inside the 
local resolved conifold $X_0=\oO_{\mathbb{P}^1}(-1,-1,0)$, i.e. reducing to the local case.
\end{proof}
For a general choice of $\Theta\in \mathbb{R}^2$ such that it does not lie on a wall, the moduli space $P^{\Theta}_n(X,\beta)$
has a universal family and consists of only 
stable objects. 
In this case, we have the following: 

\begin{prop}\label{prop on exist of virt class}
When $\Theta\in \mathbb{R}^2$ lies outside the walls in 
\eqref{all walls}, the moduli space 
$P^{\Theta}_n(X,\beta)$ can be given a $(-2)$-shifted symplectic derived scheme structure in the sense of Pantev-To\"en-Vaquie-Vezzosi \cite{PTVV}.
\end{prop}
\begin{proof}
The $\Theta$-stability gives an open condition for any family of 
 objects in $D^b\Coh(X)$. Therefore as in \cite[Lemma 1.3]{CMT2}, the existence of $(-2)$-shifted symplectic structure is reduced to \cite[Theorem 0.1]{PTVV}. 
\end{proof}

In the above case, by~\cite[Corollary 1.17]{CGJ},
 we know $P^{\Theta}_n(X,\beta)$ is orientable, 
 hence it admits a Borisov-Joyce 
 virtual class ~\cite{BJ}:
$$[P^{\Theta}_n(X,\beta)]^{\mathrm{vir}}\in H_{2n}(P^{\Theta}_n(X,\beta),\mathbb{Z}), $$ 
which depends on the choice of  
orientation~\cite{CGJ, CL2}.
The virtual dimension of $P_n^{\Theta}(X, \beta)$ 
is in general non-zero, and
we need to involve some insertions to obtain enumerative invariants. 
As in~\cite{CMT1, CMT2, CT1, CK2}, 
 we consider \textit{primary insertions} 
\begin{align*}
\tau \colon H^{4}(X,\mathbb{Z})\to H^{2}(P^{\Theta}_n(X,\beta),\mathbb{Z}), \quad 
\tau(\gamma) \cneq (\pi_{P})_{\ast}(\pi_X^{\ast}\gamma \cup\ch_3(\mathbb{F}) ),
\end{align*}
where $\pi_X$, $\pi_P$ are projections from $X \times P^{\Theta}_n(X,\beta)$
onto corresponding factors, $\mathbb{I}=(\pi_X^*\oO_X\to \mathbb{F})$ is the universal pair and $\ch_3(\mathbb{F})$ is the
Poincar\'e dual to the fundamental cycle of $\mathbb{F}$.
\begin{defi}
The primary counting invariants of $\Theta$-stable perverse coherent systems are 
\begin{align*}P^{\Theta}_{n,\beta}(\gamma):=\int_{[P^{\Theta}_n(X,\beta)]^{\rm{vir}}} \tau(\gamma)^n\in \mathbb{Z}. \end{align*}
\end{defi}
The following is the main result of this section. 
\begin{thm}\label{cpt main thm}
Let $f \colon X\to Y$ be as in Setting \ref{setting}, $E\subset X$ be the exceptional surface and $[\mathbb{P}^1]\in H_2(X,\mathbb{Z})$ be the fiber class of $f|_{E} \colon E\to C$. Let $\Theta=(\theta_0,\theta_1)\in \mathbb{R}^2$ be outside walls defined in \eqref{all walls}. 
Then for certain choice of orientation, we have  
\begin{align*}
	\sum_{n\in \mathbb{Z}, f_{\ast}\beta=0}
	\frac{P^{\Theta}_{n,\beta}(\gamma)}{n!}\,q^nt^{\beta}=\left\{
	\begin{array}{cl}
		\exp\left(qt^{[\mathbb{P}^1]}\right)^{\int_X\gamma\cup [E] }           &\mbox{ if } \theta_0<0,\, \theta_0+2\theta_1>0, \\
		&  \\
		\exp\left(qt^{[\mathbb{P}^1]}-qt^{-[\mathbb{P}^1]}\right)^{\int_X\gamma\cup [E] }          &\mbox{ if } \theta_0<0,\, \theta_0+2\theta_1<0,  \\
		&  \\
		\exp\left(-qt^{-[\mathbb{P}^1]}\right)^{\int_X\gamma\cup [E] }       &\mbox{ if }  \theta_0>0,\, \theta_0+2\theta_1<0,  \\
		&  \\ 
		1  \quad \quad \quad           &   \mbox{ otherwise}. 
	\end{array} \right. 
\end{align*} 
Here the first three cases correspond to $\Theta$ lies in the chamber (a) I, (b) II and III, (c) IV in Figure \ref{figure3} respectively.
\end{thm}
\begin{proof}
We only need to consider curve classes $\beta=d\,[\mathbb{P}^1]$ 
for $d\in \mathbb{Z}$.  
Our aim is to evaluate 
\begin{align}\label{equ need to verify}P^{\Theta}_{n,d}(\gamma)=\int_{[P^{\Theta}_n(X,d [\mathbb{P}^1])]^{\rm{vir}}} \tau(\gamma)^n. \end{align}
We assume $\gamma\cdot [E]\geqslant 0$ (otherwise consider $-\gamma$ instead).
We take $\{S_i\}_{i=1}^n$ to be $n$-different homological cycles which represent the class $\gamma\in H^4(X,\mathbb{Z})$ such that 
the intersections $(S_i\cap E)$'s are transverse, in general position and are disjoint for different choices of $1\leqslant i\leqslant n$.
For simplicity, we assume 
$$S_i\cap E=\big\{P_{i,1},P_{i,2},\cdots, P_{i,n_{0,1}(\gamma)}\big\}, \quad i=1,2,\cdots,n,$$ 
where all points are with positive signs. 
Here $n_{0,1}(\gamma):=\gamma\cdot [E]$ is the genus 0, degree 1 Gopakumar-Vafa type invariant defined by Klemm-Pandharipande \cite{KP}. In the case when there is a point with negative sign, 
we can pair it with another point with positive sign, then it is easy to argue the pair will not contribute to \eqref{equ need to verify}.

For any $(\oO_X\to F)\in P^{\Theta}_n(X,\beta)$, by Lemma \ref{lem on supp on fibers}, 
$F$ is supported on fibers of $f|_{E} \colon E\to C$. We can decompose 
$F$ into the direct sum
\begin{align}\label{decom:F}
	F=\bigoplus_{i=1}^k F_i, 
	\end{align}
such that $\mathrm{supp}(F_i)$'s are connected.
Then any $(\oO_X\to F_i)$ is supported on a formally 
local chart of $X \to Y$, so can be regarded as a perverse coherent system 
for $X_0 \to Y_0$. 
Therefore we can 
present it as a finite dimensional $\Theta$-stable representation $V=(V_0,V_1, V_{\infty}=\mathbb{C})$ 
of the framed quiver \eqref{framed cy4 quiver}.

Note that we have 
$\chi(F_i)=\dim \Hom(\oO_X, F_i) \neq 0$. 
Indeed otherwise 
$(0 \to F_i)$ is a direct summand of $(\oO_X \to F)$
which violates the $\Theta$-stability. 
Therefore 
\begin{align}\label{n compare k}n=\chi(F)=\sum_{i=1}^k \chi(F_i)\geqslant k. \end{align}
In particular
$n=k$ if and only if $\chi(F_i)=1$
for all $i$, and $n>k$ otherwise. 

Since $\gamma=[S_i]$, the 
class $\tau(\gamma)$ is represented by 
a divisor of $P_n^{\Theta}(X, \beta)$
supported on
pairs $(\oO_X\to F)$
such that, under the decomposition (\ref{decom:F}), 
there is a unique $1\leqslant j(i) \leqslant k$
satisfying 
\begin{align}\label{impose}
	\dim \Supp(F_{j(i)})=1, \quad 
	S_i \cap \Supp(F_{j(i)}) \neq \emptyset.
	\end{align} 
The multiplicity of this divisor at $(\oO_X \to F)$
is given by $d(F_{j(i)})$. 
Note that by our generic choice of $S_i$, 
we have $S_{i'} \cap \Supp(F_{j(i)}) =\emptyset$
for $i' \neq i$. 
Therefore if the pair $(\oO_X \to F)$
satisfies the above condition for $i$ and $i'$, 
we have $j(i) \neq j(i')$. 

Now the cycle $\tau(\gamma)^n$
imposes conditions (\ref{impose}) for each $1\leqslant i\leqslant n$, 
so
it is represented by 
a codimension $n$ cycle 
supported on 
 pairs $(\oO_X \to F)$
such that 
$n\leqslant k$, hence $n=k$, 
and $d(F_i) \neq 0$ for all $i$. 
It follows that each $F_i$
satisfies 
$\chi(F_i)=1$
and $d(F_i) \neq 0$, 
so it corresponds to a $\Theta$-stable 
representation $V$ 
of the form
\begin{align*}
V=(V_0, V_1, V_{\infty}=\mathbb{C}), \,\, 
V_1 \neq \mathbb{C}, \,\, V_{\infty} \stackrel{\cong}{\to} V_0.
\end{align*} 
Here the latter isomorphism follows since $\dim V_0=\chi(F_i)=1$
and $V_{\infty}\to V_0$ is non-zero. 
Therefore 
by Proposition~\ref{prop:onedim}, 
$F_i$ is either one of the following objects
\begin{align*}
F_i = \begin{cases}
	j_{a\ast} \oO_{\mathbb{P}^1}, & \mbox{ if }\Theta \in \rm{I}, \\
	j_{a\ast} \oO_{\mathbb{P}^1} \mbox{ or }
j_{a\ast} \oO_{\mathbb{P}^1}(-2)[1], & \mbox{ if }\Theta \in \rm{I\hspace{-.1em}I}, \\
	j_{a\ast} \oO_{\mathbb{P}^1} \mbox{ or }
\Upsilon(j_{a\ast}^+ \oO_{\mathbb{P}^1}), & \mbox{ if }\Theta \in \rm{I\hspace{-.1em}I\hspace{-.1em}I}, \\
\Upsilon(j_{a\ast}^+ \oO_{\mathbb{P}^1}), & \mbox{ if }\Theta \in \rm{I\hspace{-.1em}V}. 
\end{cases}  
\end{align*}
There is no such $F_i$ in other cases and $\Upsilon$ is the flop equivalence (\ref{equiv:flop}), 
$j_a$ for $a \in C$ is the composition
\begin{align*}
	j_a \colon \mathbb{P}^1=(f|_{E})^{-1}(a) \hookrightarrow E \hookrightarrow 
	X,
\end{align*}
and $j_a^+$ is similarly defined for the flop side. 

Below we prove the desired formula in the case of $\Theta \in \rm{I\hspace{-.1em}I}$. 
Other cases are similarly obtained. 
We call an object of the form 
$j_{a\ast} \oO_{\mathbb{P}^1}, j_{a\ast} \oO_{\mathbb{P}^1}(-2)[1]$
as type (i), (ii) respectively. 
 By a computation of the numerical classes, 
the number of objects in $\{F_1, \ldots, F_n\}$
of type (i), (ii) is 
$(n+d)/2$, $(n-d)/2$ respectively. 
The number of such pairs is finite, 
so $\tau(\gamma)^n$ is represented by
a zero cycle. 	
	The total degree of the zero cycle $\tau(\gamma)^n$
is 
calculated as follows. 
We first choose one of the points in $S_i\cap E$ for each $i=1,2,\cdots, n$ and then choose $(n+d)/2$ in $n$
for a choice of type (i) objects. 
Since the type (ii) objects contribute 
to $-1$, the total degree is
\begin{align}\label{total}
	(-1)^{(n-d)/2}
	 {n \choose \frac{n+d}{2}} \cdot (n_{0,1}(\gamma))^n.
	\end{align}
The contribution from the virtual class is determined as follows: consider an open immersion
\begin{align*}
	U \cneq 
(\mathrm{Sym}^{(n+d)/2}(C) \times \mathrm{Sym}^{(n-d)/2}(C))\setminus 
\Delta_{\rm{Big}}
\hookrightarrow P_n^{\Theta}(X, \beta),
\end{align*}
where $\Delta_{\rm{Big}}$ is the big diagonal, sending 
$(a_1, \ldots, a_{(n+d)/2}, b_1, \ldots, b_{(n-d)/2})$
to the object
\begin{align*}
\oO_X \to \bigoplus_{i=1}^{(n+d)/2}j_{a_i \ast}\oO_{\mathbb{P}^1} 
\oplus \bigoplus_{i=1}^{(n-d)/2}j_{b_i \ast}\oO_{\mathbb{P}^1}(-2)[1]. 
\end{align*}
It is straightforward to check that $\Ext^2(-, -)$ of the above 
object is zero. Therefore by Lemma~\ref{ass on virt class}, 
for some choice of orientation, 
the virtual class (up to invert $2$ in the coefficient) is written as
\begin{align*}
	[P_n^{\Theta}(X, \beta)]^{\rm{vir}}
	=[\overline{U}]+
	\sum_{i \in I} c_i [C_i],
\end{align*}
where
$c_i \in \mathbb{Z}[1/2]$, 
 each $C_i$ is an irreducible
 closed subscheme of $P_n^{\Theta}(X, \beta)$ with dimension $n$
such that $C_i \neq \overline{U}$. 
Since the zero cycle which represents $\tau(\gamma)^n$ is contained 
in $U$, 
it follows that the integral (\ref{equ need to verify})
coincides with (\ref{total}).
Therefore we obtain the desired 
expression of the generating series
in the case $\Theta \in \rm{I\hspace{-.1em}I}$. 
\end{proof}

\begin{rmk}
In the non-commutative chamber, our formula shares a similar shape as Szendr\"oi's formula \cite{Sze}, 
which is a product of counting invariants on
$X$ and its flopping side $X^+$. See the RHS of the formula in Corollary \ref{cor on nc} $($taking $m=\lambda_3$$)$ for an expression of Szendr\"oi's formula.
\end{rmk}
In the above theorem, we used the following technical lemma on Borisov-Joyce virtual classes, which we now prove using 
the recent work of Oh-Thomas~\cite{OT} lifting the virtual classes 
in Chow groups (up to invert $2$ in the coefficient). 
\begin{lem}\label{ass on virt class}
	Let $M$ be a projective fine
	 moduli scheme of simple objects in $D^b\Coh(X)$ of a Calabi-Yau 4-fold $X$, which can be given 
	a $(-2)$-shifted symplectic derived scheme structure. 
	Let 
	$[F]\in M$ be a point such that 
	$\Ext^2(F, F)=0$, and take 
	the unique irreducible component $M' \subset M$ which contains $[F]$. 
	Then for some choice of orientation, the
	Borisov-Joyce virtual class is written as 
	\begin{align*}
		[M]^{\rm{vir}}=[M']+\sum_{i\in I} c_i [C_i], \quad 
		c_i \in \mathbb{Z}[1/2]
	\end{align*}
	in $H_{2n}(M, \mathbb{Z}[1/2])$.
	Here $2n$ is the (real) virtual dimension of $M$, and 
	each $C_i \subset M$ is an irreducible $n$-dimensional 
	subscheme such that $C_i \neq M'$. 
\end{lem}
\begin{proof}
	By~\cite{OT}, the BJ virtual class is lifted to 
	an element of the Chow group
	with $\mathbb{Z}[1/2]$-coefficient (which we call Oh-Thomas virtual class below):
	\begin{align*}
		[M]^{\rm{vir}}_{\rm{OT}} \in A_{n}(M, \mathbb{Z}\left[1/2 \right]). 
	\end{align*}
	We briefly review their construction. Let 
	\begin{align*}
		(\dR \pi_{M\ast} \dR \hH om(\eE, \eE)_0[1])^{\vee} \to \mathbb{L}_M
	\end{align*}
	be the obstruction theory for $M$, 
	where $\eE$ is a universal object on $X \times M$
	and $\pi_M \colon X \times M \to M$ is the projection. 
	Let $M \hookrightarrow A$ be a closed immersion into 
	a smooth scheme $A$ with defining ideal $I \subset \oO_A$. 
	It is proved in~\cite[Proposition~4.1]{OT}
	that the above obstruction theory is represented by a map of 
	complexes
	\begin{align}\label{3term}
		(T \to E \to T^{\vee}) \to (0 \to I/I^2 \to \Omega_A|_{M}).
	\end{align}
	Here $E$, $T$ are vector bundles on $M$
	such that $E$ is equipped with a nondegenerate quadratic form, 
	satisfying some compatibility with Serre duality pairing. 
	The stupid truncation of the above map
	\begin{align*}
		(E \to T^{\vee}) \to (I/I^2 \to \Omega_A|_M)
	\end{align*}
	is a Behrend-Fantechi perfect obstruction theory \cite{BF}, 
	so we have the intrinsic normal 
	cone $\mathfrak{C}_M \subset [E^{\vee}/T]$. 
	By pulling it back to $E^{\vee} \cong E$, 
	we obtain the cone 
	$C_{E^{\bullet}} \subset E$. 
	Then Oh-Thomas virtual class 
	\begin{align*}
		[M]^{\rm{vir}}_{\rm{OT}}:=
		\sqrt{0_E^{!}}[C_{E^{\bullet}}]
		\in A_n(M, \mathbb{Z}[1/2])
	\end{align*}
	is given by the square root Gysin pull-back (\cite[Definition~3.3]{OT})
	of the zero section $0_E \colon M \to E$.
	Here an orientation is required in the definition. 
	
The moduli space $M$ is smooth at $[F] \in M$
by the assumption $\Ext^2(F, F)=0$. 
	Hence there is an irreducible smooth Zariski 
	open subset $U \subset M$
	which contains $[F]$, so $M':=\overline{U}$ is
	the unique irreducible component which contains $[F]$. 
	Then obviously  
	\begin{align*}
		[M]^{\rm{vir}}_{\rm{OT}}
		=c' [M']+\sum_{i\in I} c_i [C_i]
	\end{align*}
	for some $c', c_i \in \mathbb{Z}[1/2]$, 
	where $C_i \subset M$ is irreducible with dimension $n$
	and $C_i \neq M'$.  
	
	We are left to show $c'=1$. 
	From the construction of $\sqrt{0_E^{!}}$, 
	it is easy to 
	see it commutes with pull-back by an open immersion $U \hookrightarrow M$. 
	So we have 
	\begin{align}\label{equ ot vir class}
		[M]_{\rm{OT}}^{\rm{vir}}\big|_{U}=
		\sqrt{0_E^!\big|_U}[C_{E^{\bullet}}|_{U}] \in A_n(U, \mathbb{Z}[1/2])
		=\mathbb{Z}[1/2][U].
	\end{align}
	Here the last identity holds as $U$ is an irreducible 
	smooth scheme of dimension $n$. 
	By~\cite[Equation~(56)]{OT}
	the class \eqref{equ ot vir class} is independent of the choice of the 3-term complex
	(\ref{3term}). So on $U$, we can replace (\ref{3term}) by
	$T_U \to 0 \to T_U^{\vee}$. The resulting 
	virtual class on $U$ is then $\pm [U]$. By choosing a suitable 
	orientation, we can take $c'=1$. 
\end{proof}

\subsection{Comparison with $Z_t$-stable pairs}\label{subsec:Zt}
For a birational contraction $f \colon X \to Y$
as in Setting~\ref{setting}, 
recall that we have fixed a $\mathbb{Q}$-ample divisor $\omega$ on $X$
with degree one on the fibers of $f|_{E} \colon E \to C$, 
and the associated slope function
 is defined by (\ref{u:slope}). 
Here we recall the definition of $Z_t$-stability:
\begin{defi}\label{def Zt sta}\emph{(\cite[Lemma 1.7]{CT1})}
	Let $F$ be a one dimensional coherent sheaf and $s \colon \oO_X\to F$ 
	be a section.
	We say $(F,s)$ is a $Z_t$-(semi)stable pair for $t\in \mathbb{R}$ if 
	\begin{enumerate}
		\renewcommand{\labelenumi}{(\roman{enumi})}
		\item for any subsheaf $0\neq F' \subseteq F$, we have 
		$\mu_{\omega}(F')<(\leqslant)t$,
		\item for any
		subsheaf $ F' \subsetneq F$ 
		such that $s$ factors through $F'$, 
		we have 
		$\mu_{\omega}(F/F')>(\geqslant)t$. 
	\end{enumerate}
\end{defi}

We only consider $Z_t$-stable pairs $(F, s)$ such that 
$[F]=\beta$ satisfies $f_{\ast}\beta=0$, i.e. 
$F$ is supported on fibers of $f \colon X \to Y$. 
Then 
the wall-chambers of $Z_t$-(semi)stable pairs are classified as follows.
\begin{lem}\label{lem on zt pair}
	The set of walls for $Z_t$-stability of pairs $(F,s)$ on 
	$X$ is given by $\mathbb{Z} \subset \mathbb{R}$. 
	Moreover, there exists a $Z_t$-stable pair $(F,s)$ with $[F]\neq 0$ only 
	if the following inequalities hold: 
	\begin{align}\label{ineq:tchi}
	t>\frac{\chi(F)}{d(F)} \geqslant 1. 
	\end{align}
\end{lem}
\begin{proof}
	The first claim holds since 
	any one dimensional stable sheaf on $X$
	supported on fibers of $f \colon X \to Y$ is of the form 
	$j_{a\ast}\oO_{\mathbb{P}^1}(k)$ for some $k \in \mathbb{Z}$
	and $a \in C$, whose slopes are integers.  
	We claim that if there is a 
	$Z_t$-stable pair $(F,s)$, we have
	the inequalities (\ref{ineq:tchi}).
	Let $Z \subset X$ be the closed subscheme 
	such that $\mathrm{Im}(s)=\oO_{Z}$. 
	By the $Z_t$-stability, 
	$\oO_Z$ is a non-zero subsheaf of $F$.
	If $\oO_Z \neq F$, we have 
	\begin{align*}
		\mu(F/\oO_Z)=\frac{\chi(F)-\chi(\oO_Z)}{d(F)-d(\oO_Z)}>t> \frac{\chi(F)}{d(F)}, \end{align*}
	which implies that 
	\begin{align*}
		\frac{\chi(\oO_Z)}{d(\oO_Z)}< \frac{\chi(F)}{d(F)}. 
	\end{align*}
	This is an equality if $F=\oO_Z$. Finally, using the fact that any Cohen-Macaulay curve $Z$ in $X$ supported on fibers of $f$
	 satisfies
	$\chi(\oO_Z)\geqslant d(\oO_Z)$, we are done. 
\end{proof}
The following proposition gives a comparison between 
stable perverse coherent systems and $Z_t$-stable pairs: 
\begin{prop}\label{Z_t chamber}
	Let $m\geqslant 2$
	and take $\Theta=(-m+1+0^+,m)$, i.e. 
	$\Theta$ lies in the chamber between walls 
	$L^{-}_{-}(m-1)$ and $L^{-}_{-}(m)$. 
	Then a $\Theta$-stable perverse coherent system 
	on $X$ is a $Z_t$-stable pair
	for $t=m-0^+$, 
	i.e. $t$ lies in the chamber $(m-1,m)\subset \mathbb{R}$, 
	and vice versa. 
\end{prop}
\begin{proof}
	Let $(F,s)$ be a $\Theta$-stable perverse coherent system
	on $X$ supported on fibers of $f$. 
	There is an exact sequence in $\Per(X/Y)$:
	\begin{align*}
		0 \to \mathcal{H}^{-1}(F)[1] \to F \to \mathcal{H}^{0}(F) \to 0. 
	\end{align*}
	Note that
	$\chi(\mathcal{H}^{-1}(F)[1])\geqslant 0$ as 
	$\dR f_{\ast}(\hH^{-1}(F)[1])$ is a zero dimensional sheaf. 
	Assume that
	$F$ is not a sheaf (so $\mathcal{H}^{-1}(F)[1]\neq 0$). 
	Then we have  
	\begin{align*}
		0 \geqslant (\theta_0+\theta_1)\cdot \chi(\mathcal{H}^{-1}(F))> \theta_1\cdot d(\mathcal{H}^{-1}(F)),
	\end{align*}
	where the second inequality uses the $\Theta$-stability of $F$, i.e. 
	$$\Theta(\mathcal{H}^{-1}(F)[1])=\theta_0\cdot \chi(\mathcal{H}^{-1}(F)[1])+\theta_1\cdot\left(\chi(\mathcal{H}^{-1}(F)[1])-d(\mathcal{H}^{-1}(F)[1])\right)<0. $$ 
This implies that 
	$d(\mathcal{H}^{-1}(F))$
	is negative, a contradiction.
	Therefore $F$ is a one dimensional sheaf, 
	and it is easy to see that
	$\Theta$-stability is equivalent to $Z_t$-stability by choosing 
	$t=\theta_1/(\theta_0+\theta_1)$.
	
	Conversely given a $Z_t$-stable pair $(F,s)$ for $t=m-0^+$, 
	we show that it is a $\Theta$-stable perverse coherent system. 
	Let $\mathrm{Im}(s)=\oO_Z \subseteq F$ for a closed subscheme 
	$Z \subset X$. Then applying $\dR f_{\ast}$ 
	to the exact sequence in $\Coh(X)$
	\begin{align*}
		0\to I_Z\to \oO_X\to \oO_Z\to 0, 
	\end{align*}
	we obtain $\dR^1 f_*\oO_Z=0$. For any $A\in \Coh(X)$ such that $\dR f_*A=0$, we have an exact sequence
	\begin{align*}
		0\to \Hom(\oO_Z, A)\to \Hom(\oO_X, A)=0. 
	\end{align*}
Therefore $\Hom(\oO_Z, A)=0$, so by~\cite[Lemma 3.2]{Bri1}
 we have $\oO_Z\in \Per(X/Y)$.
	From the $Z_t$-stability, we know that
	 any Harder-Narasimhan factor of $F/\oO_Z$ satisfies $\mu(F/\oO_Z)\geqslant t>0$.
	By Proposition \ref{prop on perv0}, $F/\oO_Z\in \Per(X/Y)$. 
	Therefore it follows that $F\in \Per(X/Y)$. 
	
	Next we verify the $\Theta$-stability of the pair $(F, s)$. 
	Let us take an exact sequence in $\Per(X/Y)$
	\begin{align}\label{given exa seq}0 \to F_1\to F\to F_2\to 0.
	\end{align}
Since $F$ is a sheaf, by taking the cohomology long 
exact sequence we see that $F_1$ is also a sheaf. 
	We have an exact sequence in $\Per(X/Y)$:
	\begin{align}\label{exa of F2}
		0\to\mathcal{H}^{-1}(F_2)[1]\to F_2\to \mathcal{H}^{0}(F_2) \to 0. 
	\end{align}
	By combining \eqref{given exa seq} with \eqref{exa of F2}, we obtain a distinguished triangle 
	\begin{align}\label{def of F3}F_1\to F_3\to \mathcal{H}^{-1}(F_2)[1], \end{align}
	where $F_3$ fits into a distinguished triangle
	$$F_3\to F\to \mathcal{H}^{0}(F_2). $$
	By \eqref{def of F3}, $\mathcal{H}^{1}(F_3)=0$, so the above triangle is an exact sequence
	in $\Coh(X)$. Then the $Z_t$-stability gives 
	$\mu(F_3) \leqslant t$. 
	Note also \eqref{def of F3} is equivalent to an exact sequence in $\Coh(X)$
	$$0\to \mathcal{H}^{-1}(F_2)\to F_1\to F_3\to 0. $$
	Since $\mathcal{H}^{-1}(F_2)[1]\in \Per(X/Y)$, then $\mathcal{H}^{-1}(F_2)$ belongs to the category $\mathcal{F}_{\omega}$ in Proposition \ref{prop on perv0},
	hence we know $\chi(\mathcal{H}^{-1}(F_2))\leqslant 0$, so 
	\begin{align*}
		&\chi(F_3)=\chi(F_1)-\chi(\mathcal{H}^{-1}(F_2))\geqslant \chi(F_1), \\
		&d(F_3)=d(F_1)-d(\mathcal{H}^{-1}(F_2))\leqslant d(F_1). 
	\end{align*}
	Therefore $\mu(F_1)\leqslant \mu(F_3)$. Together with 
	$\mu(F_3) \leqslant t$, 
	we conclude that $\mu(F_1)\leqslant t$ and 
	it is easy to see it is a strict inequality if $0\neq F_1\neq F$. Choosing $\Theta$ such that $t=\theta_1/(\theta_0+\theta_1)$, we 
	have proved the first condition in Definition \ref{pair stab}. 
	Similar argument also shows that the second condition of Definition \ref{pair stab} and Definition \ref{def Zt sta} are equivalent.
\end{proof}

Let $P_n^t(X, \beta)$ be the moduli space of $Z_t$-stable pairs 
$(F, s)$ with $\ch(F)=(0, 0, 0, \beta, n)$. 
For a generic $t \in \mathbb{R}$, the moduli space 
$P_n^t(X, \beta)$ is a projective scheme, and 
the following invariant
for $\gamma \in H^4(X, \mathbb{Z})$ was defined in~\cite{CT1}:
\begin{align*}
	P_{n, \beta}^t(\gamma) \cneq 
	\int_{[P_n^t(X, \beta)]^{\rm{vir}}}
	\tau(\gamma)^n \in \mathbb{Z}. 
\end{align*}
In the $t \to \infty$ limit, $P_n^t(X, \beta)$ recovers
the moduli space $P_n(X, \beta)$ of PT stable pairs.

Let $I_n(X, \beta)$ be the moduli 
space of ideal sheaves $I_Z=(\oO_X \twoheadrightarrow \oO_Z)$ of one dimensional subschemes $Z$
such that $([Z], \chi(\oO_Z))=(\beta, n)$. 
We have (primary) DT/PT invariants
\begin{align*}
	I_{n, \beta}(\gamma)\cneq 
	\int_{[I_n(X, \beta)]^{\rm{vir}}}
	\tau(\gamma)^n, \quad
	P_{n, \beta}(\gamma) \cneq 
	\int_{[P_n(X, \beta)]^{\rm{vir}}}
	\tau(\gamma)^n.	
\end{align*}
Combining Theorem \ref{cpt main thm} with Proposition \ref{Z_t chamber}, \ref{dt/pt chamber},
we prove some of our previous conjectures, which give sheaf theoretic interpretations of Gopakumar-Vafa type invariants defined by Klemm-Pandharipande \cite{KP} (see also
\cite{CMT1, CT2} for other approaches).
\begin{cor}\label{cor on verify prev conj}
Let $f\colon X\to Y$ be as in Setting \ref{setting}, $E\subset X$ be the exceptional surface and $[\mathbb{P}^1]\in H_2(X,\mathbb{Z})$ be the fiber class of $f|_{E} \colon E\to C$. 
For any $n\in \mathbb{Z}$, $\beta \in H_2(X, \mathbb{Z})$ with 
$f_{\ast}\beta=0$, a generic $t>n/\deg(\beta)$ and $\gamma \in H^4(X, \mathbb{Z})$,
we have identities 
\begin{align*}
	I_{n, \beta}(\gamma)=P_{n, \beta}(\gamma)=P_{n, \beta}^t(\gamma), 
	\end{align*}
for certain choice of orientation.
Moreover, their generating series satisfies 
\begin{align*}
\sum_{n\in \mathbb{Z},
		f_{\ast}\beta=0}\frac{P_{n, \beta}(\gamma)}{n!}q^n t^{\beta} =\exp\left(qt^{[\mathbb{P}^1]}\right)^{\int_X\gamma\cup [E] }.
	\end{align*}
Therefore the LePotier-pair/GV conjecture \cite[Conjecture 0.2]{CT1}, PT/GV conjecture \cite[\S 0.7]{CMT2} and 
DT/PT conjecture \cite[Conjecture 0.3]{CK2} hold in this case.
\end{cor}
Here the first equality is the correspondence ``DT=PT=LePotier-pair" and the second equality gives the ``PT/GV" correspondence,
where the power in the RHS is the only nontrivial GV invariant.

\section{Perverse coherent systems on local resolved conifold}\label{sect on local resolved coni}
In the previous section, we studied counting invariants of perverse coherent systems on projective CY 4-folds. In this 
section, we focus on the local model 
\begin{align*}
	X:=X_0=
	\oO_{\mathbb{P}^1}(-1,-1,0),
	\end{align*}
with a contraction $f \colon X \to Y$
for $Y=Y_0$ in (\ref{def:Y0}). 
We define counting invariants
of perverse coherent systems on $X$ using tautological insertions 
as in \cite{CK1, CKM1} and torus localization formulae as in \cite{CK2, CMT2, CT1}.
\subsection{Moduli spaces}
Recall the framed quiver $\widetilde{Q}$ with relation $I$ associated with $\oO_{\mathbb{P}^1}(-1,-1,0)$:
\begin{align*} 
\xymatrix{\bullet^{\mathbf{ \infty}} \ar[d]  &&  \\   
 \bullet_{\textbf{0}} \ar@(dl,ul)[]^{c}     \ar@/^1.5pc/[rr]^{a_1}  \ar@/^0.5pc/[rr]^{a_2} 
&& \bullet_{\textbf{1}}\ar@/^0.5pc/[ll]^{b_1}  \ar@/^1.5pc/[ll]^{b_2}  \ar@(dr,ur)[]_{d} } \quad \,\,  \\  \nonumber
 a_2b_ia_1=a_1b_ia_2, \quad b_2a_ib_1=b_1a_ib_2,   \\  \nonumber
 da_i=a_ic, \quad cb_i=b_id, \quad  i=1,2.   
\end{align*}
For a 
dimension vector $\textbf{d}=(d_0,d_1)\in (\mathbb{Z}_{\geqslant 0})^2$, 
let $V_i$ be vector
spaces with $\dim V_i=d_i$. 
The space of representations of the framed quiver $\widetilde{Q}$ is 
$$R_\textbf{d}(\widetilde{Q}):=\Hom(V_0,V_1)^{\oplus 2}\oplus \Hom(V_1,V_0)^{\oplus 2}\oplus \Hom(V_0,V_0)\oplus \Hom(V_1,V_1)
\oplus V_0. $$
We have the closed subscheme 
\begin{align*}
	R_\textbf{d}(\widetilde{Q},I)\subset R_\textbf{d}(\widetilde{Q}),
	\end{align*}
corresponding to representations which preserve the relation $I$. 
For $\Theta \in \mathbb{R}^2$, the 
$\Theta$-semistable $(\widetilde{Q}, I)$-representations 
(see Definition~\ref{def:stability:framed}) give an open subscheme of $R_\textbf{d}(\widetilde{Q},I)$, 
denoted by 
\begin{align*}
	R^{ss}_\textbf{d}(\widetilde{Q},I)\subset R_\textbf{d}(\widetilde{Q},I).
	\end{align*}
The good moduli space of $\Theta$-semistable representations with dimension vector $\textbf{d}$ is given by the GIT quotient
\begin{align*}
M^{\Theta}_{\textbf{d}}(\widetilde{Q},I):=R^{ss}_\textbf{d}(\widetilde{Q},I)
/\hspace{-.3em}/(GL(V_0)\times GL(V_1)).
\end{align*}
 If $\Theta\in \mathbb{R}^2$ lies outside walls \eqref{all walls}, 
 it is a fine moduli space consisting of $\Theta$-stable representations.
 The equivalence in (\ref{psi equi cpt}) induces an 
 isomorphism   
\begin{align}\label{iso of mod spa}M^{\Theta}_{\textbf{d}}(\widetilde{Q},I)
	\stackrel{\cong}{\to} P^{\Theta}_n(X,\beta),
\quad (\beta, n)=((d_0-d_1)[\mathbb{P}^1], d_0), 
 \end{align}
where $P^{\Theta}_n(X,\beta)$ is the moduli space of $\Theta$-stable (compactly supported) perverse coherent systems 
as in Theorem \ref{cons of moduli}.

\subsection{Torus action}

We consider the torus $(\mathbb{C}^*)^6$ which acts
on the six edges, $a_1,a_2,b_1,b_2,c,d$ diagonally by scaling. 
 It induces an action on the path algebra $\mathbb{C}Q$.
In order to preserve the relation $I$, we need the actions on edge $c$ and $d$ are the same, so we consider the subtorus 
\begin{align}\label{five dim torus}(\mathbb{C}^*)^5:=\left\{(q_1,q_2,q_3,q_4,q_5,q_6)\in (\mathbb{C}^*)^6: q_5=q_6 \right\}.  \end{align}
Note that $\mathbb{C}^*=\left\{(q,q,q^{-1},q^{-1},1,1)\in (\mathbb{C}^*)^6\right\}$ acts trivially on isomorphism 
classes of representations of $(Q,I)$, so 
we will consider the action of the quotient torus 
$$\bar{T}:=(\mathbb{C}^*)^5/\mathbb{C}^*$$ 
on moduli spaces of representations.
The above torus action does not preserve the CY4 structure, 
as in \cite[\S 2.2]{Sze}, we consider the 3-dimensional subtorus 
$$\bar{T_0}:=\big\{t\in T: \, q_1q_2q_3q_4q_5=1 \big\}. $$ 
Both $\bar{T_0}$ and $\bar{T}$ lift to actions on moduli spaces $M^{\Theta}_{\textbf{d}}(\widetilde{Q},I)$. Their fixed loci 
are the same and consist of finite number of reduced points.
\begin{prop}\label{prop torus fixed loci}
Let $\Theta\in \mathbb{R}^2$ be outside walls \eqref{all walls}. Then
we have 
\begin{align}\label{id:fixed}
	M^{\Theta}_{\textbf{d}}(\widetilde{Q},I)^{\bar{T_0}}=
	M^{\Theta}_{\textbf{d}}(\widetilde{Q},I)^{\bar{T}}
	\end{align}
and it is a finite set.
Moreover the Zariski tangent space of any element has no $\bar{T_0}$-fixed subspace.  
\end{prop}
\begin{proof}
The proof is an easy adaption of the resolved conifold case \cite{Sze, NN}. 
Let $A=\mathbb{C}Q/I$ be the quotient of the path algebra by the ideal of relations of the quiver \eqref{cy4 quiver}. 
We first consider the case 
that $\Theta$ lies in the non-commutative chamber, i.e. $\theta_0,\theta_1<0$. 
By Proposition \ref{nc chamber}, any element in $M^{\Theta}_{\textbf{d}}(\widetilde{Q},I)$ is a cyclic module $(M,m)$, 
where $m\in M$ is based at vertex $0$. We consider the surjection
$$\overline{m}: A\to M, \quad 1\mapsto m, $$
whose kernel is denoted by $J:=\ker(\overline{m})$. Let $\langle 1 \rangle$ denote the idempotent element of $\mathbb{C}Q$ at vertex $1$.
Then $A\langle 1 \rangle$ consists of paths starting from vertex $1$ which surely annihilates $m$. So we write 
$$J=J_0\oplus A\langle 1 \rangle. $$
We claim that if $(M, m)$ is $\bar{T_0}$-fixed, then 
$J_0$ is a monomial ideal. In fact, $J_0$ is a $\bar{T_0}$-fixed ideal whose generators are of the form $f(a_1,a_2,b_1,b_2,c,d)\cdot W$,
where $f$ is a monomial in those variables and $W$ is a weight zero $\bar{T_0}$-eigenvector. Note that a weight zero $\bar{T_0}$-eigenvector should have 
the same start and end point, so $\mathbb{C}^*=\left\{(q,q,q^{-1},q^{-1},1,1)\in (\mathbb{C}^*)^6\right\}$ acts trivially. So  
weight zero $\bar{T_0}$-eigenvectors are the same as weight zero $\widetilde{T_0}$-eigenvectors, where 
$\widetilde{T_0}\subset (\mathbb{C}^*)^5$ is the lift of $\bar{T_0}$ to \eqref{five dim torus}. The set of weight 
zero $\widetilde{T_0}$-eigenvectors is generated by 
$$cb_2a_2b_1a_1, \quad da_1b_1a_2b_2\in A. $$
Therefore, generators of $J_0$ are of the form $f(a_1,a_2,b_1,b_2,c,d)\cdot p(cb_2a_2b_1a_1)$, where $p$ is a polynomial with nonzero constant term.

Let $Z(A)\subset A$ be the center of $A$. 
It is easy to see that 
$$Z(A)=\langle a_ib_j+b_ja_i,\, c+d \rangle. $$
Applying it to the idempotent element at the vertex $0$, we have 
$$Z(A)\langle 0\rangle=\langle b_ja_i,\,c \rangle \cong Z(A). $$
Let $K:=J_0\cap Z(A)\langle 0\rangle$, which is an ideal in $Z(A)\langle 0\rangle\cong Z(A)$. Since 
$J_0$ is $\bar{T_0}$-fixed, the zero set of $K$ is supported on the origin of $\Spec(Z(A)\langle 0\rangle)\cong Y$ (e.g. \cite[Lemma 3.1]{CK1}).
This is disjoint from the zero set of $p(cb_2a_2b_1a_1)\in Z(A)\langle 0\rangle$. 
By the Nullstellensatz, $\langle p, K\rangle=Z(A)\langle 0\rangle$, hence $f\in J_0$.
Therefore $J_0$ is a monomial ideal, 
so it is $\bar{T}$-fixed. 
It follows that the identity (\ref{id:fixed}) holds and
both sides are finite sets. 

Next we study the $\bar{T_0}$-fixed subspace of the Zariski tangent space of $(M,m)\in M^{\Theta}_{\textbf{d}}(\widetilde{Q},I)^{\bar{T_0}}$. 
Under the derived equivalence in (\ref{psi equi cpt}), a cyclic module $(M,m)$ (resp. $A\langle 0\rangle$) corresponds to a pair $I=(\oO_X\to F)$ (resp. $\oO_X$). We have canonical isomorphisms
$$\Ext^1_X(I,I)_0\cong \Hom_X(I,F)\cong \Hom_A(J_0,M), $$
where the first isomorphism can be proved as \cite[pp. 14]{CMT2},
and the second one follows from the exact sequence of $A$-modules
$$ 0\to J_0\to A\langle 0\rangle \to M \to 0. $$
We claim that $\Hom_A(J_0,M)^{\bar{T_0}}=0$. 
It is enough to show that
 under the edge torus $(\mathbb{C}^*)^6$ on 
 $\Hom_A(J_0, m)$, no weight is a multiple of $(1,1,1,1,1,1)$.

Suppose that $\phi \colon J_0\to M=A/J$ is an eigenvector of weight $w(1,1,1,1,1,1)$ with $w\in \mathbb{Z}$.
If $w\geqslant 0$, as the $(1,1,1,1,1,1)$-eigenspace of $A\langle 0\rangle$ is spanned by $cb_2a_2b_1a_1$, 
we have
$$\phi(a)\equiv (cb_2a_2b_1a_1)^w\cdot a \equiv 0 \,(\mathrm{mod}\, J), $$
for any $a \in J_0$, i.e. $\phi=0$. 

Next suppose that $w<0$. 
Note that the $A$-module $M$ is also a coherent $\oO_Y$-module, 
supported on the origin $0 \in Y$
as it is $\bar{T_0}$-fixed. 
 Therefore the actions of $b_1 a_1$ and $c b_2 a_2$ on $M$ are nilpotent. 
Let $\alpha$ be the smallest positive integer such that $(b_1a_1)^{\alpha}\in J_0$ and $\beta$ be
the smallest positive integer such that $(cb_2a_2)^{\beta}(b_1a_1)^{\alpha-1}\in J_0$. 
 As $\phi$ has weight $w(1,1,1,1,1,1)$, we have   
$$\phi((cb_2a_2)^{\beta}(b_1a_1)^{\alpha-1})\equiv (cb_2a_2)^{\beta+w}(b_1a_1)^{\alpha-1+w} \quad (\mathrm{mod}\, J).$$ 
By the commutativity between $b_2a_2$ and $cb_1a_1$, we have 
$$\phi((cb_2a_2)^{\beta}(b_1a_1)^{\alpha})\equiv (b_1a_1)\phi((cb_2a_2)^{\beta}(b_1a_1)^{\alpha-1})
\equiv (cb_2a_2)^{\beta+w}(b_1a_1)^{\alpha+w}  \quad (\mathrm{mod}\, J).$$ 
Since there is no monomial in $A\langle 0\rangle$ with negative torus weights, 
$$\phi((cb_2a_2)^{\beta}(b_1a_1)^{\alpha})\equiv (cb_2a_2)^{\beta}\phi((b_1a_1)^{\alpha}) \equiv 0 \quad (\mathrm{mod}\, J). $$
By combining the above two expressions, we conclude 
$$(cb_2a_2)^{\beta+w}(b_1a_1)^{\alpha+w}\in J. $$
Since $w<0$, we have $(cb_2a_2)^{\beta+w}(b_1a_1)^{\alpha-1}\in J$,
which contradicts to the definition of $\beta$.

Since the wall-chamber structures 
of $\oO_{\mathbb{P}^1}(-1,-1,0)$ and $\oO_{\mathbb{P}^1}(-1,-1)$ are the same, for other choices of $\Theta$, we 
can follow the approach of \cite[\S 4]{NN} and identify the moduli space $M^{\Theta}_{\textbf{d}}(\widetilde{Q},I)$
with the moduli space of cyclic representations of some other quiver (as introduced in Chuang-Jafferis \cite{CJ}) and reduce to a similar argument as above.
\end{proof}
In actual computations, we will first fix torus action on $X=\oO_{\mathbb{P}^1}(-1,-1,0)$: let   
\begin{align}\label{repara T0}T_0=\{t=(t_0,t_1,t_2,t_3)\in(\mathbb{C}^*)^4:\,t_0t_1t_2t_3=1\}, \end{align}
which acts on $X$ in local coordinates  
such that the normal bundle of the zero section satisfies
$$N_{\mathbb{P}^1/X}=\oO_{\mathbb{P}^1}(-Z_{\infty})\otimes t_1^{-1}\oplus \oO_{\mathbb{P}^1}(-Z_{\infty})\otimes t_2^{-1}\oplus
\oO_{\mathbb{P}^1}\otimes t_3^{-1}, $$
where $Z_{0}:=[0:1]$, $Z_{\infty}:=[1:0]\in \mathbb{P}^1$ are torus fixed points.
The torus lifts to an action on $\Per_{\leqslant 1}(X/Y)$ and moduli spaces $P^{\Theta}_n(X,\beta)$. 
By Lemma \ref{quiver of local resolved coni},
it also acts on representations of quiver \eqref{cy4 quiver} as described at the beginning of this section (up to reparametrizations),
which preserves the equivalence \eqref{psi equi cpt} and the isomorphism \eqref{iso of mod spa}.

\subsection{Tautological invariants}

By Proposition \ref{prop torus fixed loci}, we can define the tautological
counting invariants of $P^{\Theta}_n(X,\beta)$ using the isomorphism \eqref{iso of mod spa} and torus localization.
We first recall the following notion of square roots.
\begin{defi}
Let $K^{T_0}(pt)\cong \mathbb{Z}[t_0^{\pm},t_1^{\pm},t_2^{\pm},t_3^{\pm}]/(t_0t_1t_2t_3-1)$ denote the $T_0$-equivariant $K$-theory of one point. A square root $V^{\frac{1}{2}}$ of $V\in K^{T_0}(pt)$
is an element in $K^{T_0}(pt)$ such that 
$$V^{\frac{1}{2}}+\overline{V^{\frac{1}{2}}}=V. $$
Here $\overline{(\cdot)}$ denotes the involution on $K^{T_0}(pt)$ induced by $\mathbb{Z}$-linearly extending the map 
$$t_0^{w_0}t_1^{w_1}t_2^{w_2}t_3^{w_3} \mapsto t_0^{-w_0}t_1^{-w_1}t_2^{-w_2}t_3^{-w_3}, $$
where $t_i$'s denote torus weights in notation \eqref{repara T0}.
\end{defi}
For a $T_0$-equivariant pair $I=(\oO_X\stackrel{s}{\to} F)$ with compactly supported $F\in \Per(X/Y)$, by Serre duality, the following 
 square root exists:
$$\chi_X(I,I)_0^{\frac{1}{2}}:=-\chi_X(F)+\chi_X(F,F)^{\frac{1}{2}}\in K^{T_0}(pt), $$
where $\chi_X(-,-)$ denotes the Euler pairing on $X$ and $\chi_0(-,-)$ denotes its trace-free part,  
$\chi_X(-):=\chi_X(\oO_X,-)$. 
Here see Remark~\ref{sign rmk}
for a choice of $\chi_X(F,F)^{\frac{1}{2}}$, 
which is not unique, though its Euler class is unique up to a sign. 
If $(F,s)$ is $\Theta$-stable for a generic $\Theta\in \mathbb{R}^2$,
then 
$\Ext^1(I,I)_0$ has no $T_0$-fixed subspace
 by Proposition \ref{prop torus fixed loci}.
 Therefore its equivariant Euler class is non-zero, 
 so the equivariant Euler class of 
$\chi_X(I,I)_0^{\frac{1}{2}}$ is well-defined. 
For a different choice of square root, the corresponding Euler class may differ by a sign.

Let $\Lambda$ be the field of rational functions 
defined by 
\begin{align*}
	\Lambda \cneq 
	\frac{\mathbb{Q}(\lambda_0,\lambda_1,\lambda_2,\lambda_3,m)}{(\lambda_0+\lambda_1+\lambda_2+\lambda_3)}
	\cong 
	\mathbb{Q}(\lambda_0, \lambda_1, \lambda_2, m).
	\end{align*}
Here
 $\lambda_i=e_{T_0}(t_i)$'s are equivariant parameters of $T_0$ 
 in \eqref{repara T0}. 
As in \cite{CK1, CKM1}, we use tautological insertions to define invariants.
\begin{defi}\label{taut inv}
Let $X=\oO_{\mathbb{P}^1}(-1,-1,0)$ and $\Theta\in \mathbb{R}^2$ be outside walls \eqref{all walls}. 
Consider a trivial $\mathbb{C}^*$-action on moduli spaces such that $e^m$ is a trivial line bundle with $\mathbb{C}^*$-equivariant weight $m$
We define the 
  tautological invariant to be 
\begin{align*}
	P^{\Theta}_{n,d}(e^m):=\sum_{\begin{subarray}{c}I=(\oO_X\to F) \in P^{\Theta}_n(X,d)^{T_0} \end{subarray}}
e_{T_0}(\chi_X(I,I)^{\frac{1}{2}}_0)\cdot e_{T_0\times \mathbb{C}^*}(\chi_X(F)^{\vee}\otimes e^m)\in 
\Lambda. 
\end{align*}
The above invariant
depends on the choice of sign for each torus fixed point.
\end{defi}
When $\Theta=(-1+0^+,1)$ (i.e. $\Theta$ lies in the PT chamber), the
invariants in Definition~\ref{taut inv}
recover the cohomological invariants studied in \cite[\S 0.4]{CKM1}.
\begin{rmk}\label{sign rmk}
In actual computations, we fix the Fano 3-fold $Y=\oO_{\mathbb{P}^1}(-1,0)$ such that the normal bundle of the zero section satisfies
$$N_{\mathbb{P}^1/Y}=\oO_{\mathbb{P}^1}(-Z_{\infty})\otimes t_1^{-1}\oplus \oO_{\mathbb{P}^1}\otimes t_3^{-1}. $$
We take 
$$\chi_X(F,F)^{\frac{1}{2}}:=\chi_Y(F,F), $$
where RHS is defined by pushforward $F$ to $\mathbb{P}^1$ followed by taking inclusion to $Y$ via zero section.

We then put an extra sign as follows:
\begin{align}\label{sign rule}e_T(\chi_X(I,I)_0^{\frac{1}{2}})=(-1)^{\chi(F)+\deg(F)+\mathrm{sign}(F)}\cdot e_{T}(-\chi_X(F)+\chi_Y(F,F)),  \end{align}
where $\mathrm{sign}(F)\in \mathbb{Z}$. 
When $F$ is scheme theoretically supported on $Y$, motivated by \cite[Equ. (0.1)]{CaoFano}, we take
$\mathrm{sign}(F)=0$. 
In the thickened case, we will explain how to choose it in examples computed in Section \ref{sect on computation}
\footnote{It is also an interesting question to link them with global orientations obtained in \cite{boj}.}.
\end{rmk}
\begin{rmk}\label{rmk on computation}
In the computations of
 $\chi_Y(F,F)$, $\chi_X(F)$ and their equivariant Euler classes,
we use the adjunction formula 
$$\chi_Y(F_i,F_j)=\chi_{\mathbb{P}^1}(F_i,F_j)-\chi_{\mathbb{P}^1}(F_i,F_j\otimes N_{\mathbb{P}^1/Y})+\chi_{\mathbb{P}^1}(F_i,F_j\otimes \wedge^2 N_{\mathbb{P}^1/Y}), $$
and equivariant Riemann-Roch formula 
\begin{align*}
\ch\Big(\chi\big(\oO_{\mathbb{P}^1}(aZ_0+bZ_{\infty})\big)\Big)=\frac{e^{-a\lambda_0}}{1-e^{\lambda_0}}+\frac{e^{b\lambda_0}}{1-e^{-\lambda_0}}\
=\frac{e^{(b+1)\lambda_0}-e^{-a\lambda_0}}{e^{\lambda_0}-1}.
\end{align*}
\end{rmk}

Let $Z$ be the resolved conifold 
embedded into $X$:
\begin{align}\label{Z:iota}
	\iota \colon Z:=\oO_{\mathbb{P}^1}(-1,-1)\times \{0\}
	\hookrightarrow X. 
	\end{align}
We say that $I \in P_n^{\Theta}(X, d)$ 
is scheme theoretically supported on $Z$
if it is of the form $(\oO_X \to \iota_{\ast}F)$. 
Motivated by the dimensional reduction and cohomological limit 
in~\cite{CKM1}, we show the following:
\begin{prop}\label{dim red}
Let us take $\Theta\in \mathbb{R}^2$ which lies 
outside walls in \eqref{all walls}. We have the following: \\
(1) $($Dimensional reduction$)$ For each $I\in P^{\Theta}_n(X,d)^{T_0}$, 
we have 
\begin{align*}
	 e_{T_0}(\chi_X(I,I)^{\frac{1}{2}}_0)\cdot &e_{T_0\times \mathbb{C}^*}(\chi_X(F)^{\vee}\otimes e^m)\big|_{m=\lambda_3} \\
	&=\left\{\begin{array}{cl}
		e_{(\mathbb{C}^*)^3}(\chi_Z(I,I)_0), & 
		\mbox{if } I \mbox{ is scheme theoretically supported on } Z, \\
		& \\
		0, & \mbox{otherwise}.
		\end{array} \right. 
	\end{align*}
Here $(\mathbb{C}^*)^3=T_0|_{Z}$ is the restricted torus.  \\
(2) $($Insertion-free limit$)$
$$\lim_{\begin{subarray}{c}Q\, \mathrm{fixed} \\ m\to \infty \end{subarray}}\left(\sum_{n,d}P^{\Theta}_{n,d}(e^m)q^nt^d\big|_{Q=qm}\right)=
\sum_{n,d}Q^n t^d\sum_{\begin{subarray}{c}I=(\oO_X\to F) \in P^{\Theta}_n(X,d)^{T_0} \end{subarray}}e_{T_0}(\chi_X(I,I)^{\frac{1}{2}}_0). $$
\end{prop}
\begin{proof}
(1) We first assume that 
$I$ is scheme theoretically supported on $Z$,
so that it is written as 
 $I=(\oO_X\to \iota_*F)$.
By adjunction, we get 
\begin{align*}\chi_X(I,I)_0^{\frac{1}{2}} & \cneq -\chi_X(\iota_*F)+\chi_X(\iota_*F,\iota_*F)^{\frac{1}{2}} \\
&=-\chi_Z(F)+\chi_Z(F,F). \end{align*}
By the Serre duality for $Z$, we have 
\begin{align*}\chi_Z(I,I)_0=\chi_Z(F,F)-\chi_Z(F)+\chi_Z(F)^{\vee}\otimes t_3. \end{align*}
Then we have identities
\begin{align*}e_{T_0}(\chi_X(I,I)^{\frac{1}{2}}_0)\cdot e_{T_0\times \mathbb{C}^*}(\chi_X(\iota_*F)^{\vee}\otimes e^m)|_{m=\lambda_3}
&=\frac{e_{T_0}(\chi_Z(F,F))\cdot e_{T_0}(\chi_Z(F)^{\vee}\otimes t_3)}{e_{T_0}(\chi_Z(F))} \\
&=e_{T_0}(\chi_Z(I,I)_0)\\
&=e_{(T_0|_{Z})}(\chi_Z(I,I)_0). \end{align*}
Therefore (1) holds when $I$ is scheme theoretically 
supported on $Z$. 

Next we consider pairs $I=(\oO_X\to F)\in P^{\Theta}_n(X,d)^{T_0}$
thickened into normal direction of $Z$ inside $X$.  
We first deal with the case that $F$ is a sheaf, 
e.g. when $\Theta$ lies in a $Z_t$-stable pair chamber
(see Proposition \ref{Z_t chamber}). 
By Proposition \ref{prop torus fixed loci}, $I$ is also fixed by the full torus $(\mathbb{C}^*)^4$, hence fixed by the subtorus 
$(\mathbb{C}^{\ast})^3 \subset (\mathbb{C}^{\ast})^4$
acting on the fibers of the 
projection $\pi \colon X \to \mathbb{P}^1$. 
We have decompositions into 
$(\mathbb{C}^{\ast})^3$-weight spaces  
\begin{align*}
	\pi_*F =\bigoplus_{(i_1,i_2,i_3)\in \Delta\subset \mathbb{Z}^3} F^{i_1,i_2,i_3},  \quad
\pi_*\oO_X =\bigoplus_{(i_1,i_2,i_3)\in \mathbb{Z}_{\geqslant 0}^3} L_1^{-i_1}\otimes L_2^{-i_2}\otimes L_3^{-i_3}.
\end{align*}
Here 
$$L_1=\oO_{\mathbb{P}^1}(-Z_{\infty})\otimes t_1^{-1},\,\, L_2=\oO_{\mathbb{P}^1}(-Z_{\infty})\otimes t_2^{-1},\,\, 
L_3=\oO_{\mathbb{P}^1}\otimes t_3^{-1} $$
are equivariant line bundles on $\mathbb{P}^1$ and $\Delta$ is defined by 
$$\Delta:=\left\{(i_1,i_2,i_3)\in \mathbb{Z}^3:\,F^{i_1,i_2,i_3}\neq 0  \right\}. $$
The torus invariant section $s$ is determined 
by a collection of morphisms 
$$s^{-i_1,-i_2,-i_3}: L_1^{-i_1}\otimes L_2^{-i_2}\otimes L_3^{-i_3}\to F^{-i_1,-i_2,-i_3}, \quad i_1,i_2,i_3\geqslant 0. $$
Since $s$ is an $\oO_X$-module homomorphism, 
the above morphisms 
fit into a commutative diagram
\begin{align*}
\xymatrix{
L_1^{-i_1}\otimes L_2^{-i_2}\otimes L_3^{-i_3}\otimes L_1^{-1}   \ar[r]^-{=} \ar[d]^{s^{-i_1,-i_2,-i_3}} & L_1^{-i_1-1}\otimes L_2^{-i_2}\otimes L_3^{-i_3}  \ar[d]^{s^{-i_1-1,-i_2,-i_3}}\\
F^{-i_1,-i_2,-i_3}\otimes L_1^{-1} \ar[r]^-{\phi} &  F^{-i_1-1,-i_2,-i_3}. }
\end{align*}
We have 
similar commutative diagrams by replacing the role of 
$L_1$ with $L_2$ or $L_3$.

By the stability, $s$ is not identically zero, so there exists
$(i_1, i_2, i_3)$
such that $s^{-i_1,-i_2,-i_3} \neq 0$. 
Since $F$ is thickened into the normal 
direction of $Z$ inside $X$, 
we may assume $i_3\geqslant 1$.
From the above commutative diagrams, we obtain 
$$s^{-i_1+1,-i_2,-i_3},\,\, s^{-i_1,-i_2+1,-i_3},\,\, s^{-i_1,-i_2,-i_3+1} \neq 0. $$
By inductions, we know $F^{0,0,-1}\neq 0$ and $s^{0,0,-1}\neq 0$. So we have 
\begin{align*}
	F^{0,0,-1}=\oO_{\mathbb{P}^1}(aZ_0+bZ_{\infty})\otimes t_3, 
\end{align*}
for some $a, b \geqslant 0$.
Therefore we have identities
\begin{align*}e_{T_0}(\chi(F)^{\vee}\otimes t_3)
&=\sum_{(i_1,i_2,i_3)\in \Delta\subset \mathbb{Z}^3} e_{T_0}(\chi(F^{i_1,i_2,i_3})^{\vee}\otimes t_3) \\
&= e_{T_0}(\chi(F^{0,0,-1})^{\vee}\otimes t_3+\cdots) \\
&=e_{T_0}(1+\cdots)=0.
\end{align*} 
Here we have used the fact that $\chi(F)=H^0(F)$ so that $\chi(F)$ does not 
contain elements with negative signs (e.g. $-t_0^{w_0}t_1^{w_1}t_2^{w_0}t_3^{w_3}$) in the weight space decomposition.

A similar argument works if $F=F'[1]\in \Per(X/Y)$ for a one dimensional sheaf $F'$.
In general, we have a short exact sequence in $\Per(X/Y)$: 
\begin{align}\label{deco exa seq}0\to \mathcal{H}^{-1}(F)[1]\to F\to \mathcal{H}^{0}(F)\to 0, \end{align}
where $\mathcal{H}^{*}(F)$ are one dimensional sheaves. Applying $\Hom(\oO_X,-)$, we obtain the exact sequence
\begin{align*}0\to \Hom(\oO_X,\mathcal{H}^{-1}(F)[1])\to \Hom(\oO_X, F)\to \Hom(\oO_X,\mathcal{H}^{0}(F))\to 0
	\end{align*}
together with the vanishings
\begin{align*}
	\Hom^{i\neq 0}(\oO_X,\mathcal{H}^{-1}(F)[1])=0, \quad 
	\Hom^{i\neq 0}(\oO_X,\mathcal{H}^{0}(F))=0. \end{align*}
Since $(F,s)$ is $(\mathbb{C}^*)^4$-fixed, we pushforward $F$ and \eqref{deco exa seq} to $\mathbb{P}^1$ and do weight space decomposition as before. It is easy to see the above argument applies. 

(2) We have 
\begin{align*}
&\quad\,  \lim_{\begin{subarray}{c}Q\, \mathrm{fixed} \\ m\to \infty \end{subarray}}\left(\sum_{n,d}P^{\Theta}_{n,d}(e^m)q^nt^d\big|_{Q=qm}\right) \\
&=\lim_{\begin{subarray}{c}m\to \infty \end{subarray}}\left(\sum_{n,d}\frac{P^{\Theta}_{n,d}(e^m)}{m^n}Q^nt^d\right)\\
&=\lim_{\begin{subarray}{c}m\to \infty \end{subarray}}\left(\sum_{n,d}Q^nt^d\sum_{\begin{subarray}{c}I \in P^{\Theta}_n(X,d)^{T_0} \end{subarray}}
e_{T_0}(\chi_X(I,I)^{\frac{1}{2}}_0)\cdot \frac{e_{T_0\times \mathbb{C}^*}(\chi_X(F)^{\vee}\otimes e^m)}{m^n}\right) \\
&=\lim_{\begin{subarray}{c}m\to \infty \end{subarray}}\left(\sum_{n,d}Q^nt^d\sum_{\begin{subarray}{c}I \in P^{\Theta}_n(X,d)^{T_0} \end{subarray}}
e_{T_0}(\chi_X(I,I)^{\frac{1}{2}}_0)\cdot \frac{e_{T_0\times \mathbb{C}^*}(H^0(X,F)^{\vee}\otimes e^m)}{m^n}\right) \\
&=\lim_{\begin{subarray}{c}m\to \infty \end{subarray}}\left(\sum_{n,d}Q^nt^d\sum_{\begin{subarray}{c}I \in P^{\Theta}_n(X,d)^{T_0} \end{subarray}}
e_{T_0}(\chi_X(I,I)^{\frac{1}{2}}_0)\cdot \frac{(m^n+\mathrm{l.o.t.})}{m^n}\right)\\
&=\sum_{n,d}Q^nt^d\sum_{\begin{subarray}{c}I \in P^{\Theta}_n(X,d)^{T_0} \end{subarray}}
e_{T_0}(\chi_X(I,I)^{\frac{1}{2}}_0),
\end{align*}
where `l.o.t.' means lower order terms of $m$ and we use $\chi(F)=n$ in the fourth identity. 
\end{proof}

\subsection{Wall-crossing formula} 
Let $\Theta_{\mathrm{PT}}\cneq (-1+0^+, 1)$ and consider tautological PT invariants 
\begin{align*}
	P_{n, d}(e^m) \cneq P_{n, d}^{\Theta_{\mathrm{PT}}}(e^m)
	 \in \Lambda. 
	\end{align*}
In \cite[Appendix B]{CKM1}, the following closed formula is conjectured.
\begin{conj}[\cite{CKM1}]\label{conj formula in pt chamber}
There exist choices of signs such that 
$$\sum_{n,d}P_{n,d}(e^m)q^nt^d=\prod_{k\geqslant 1}\left(1-q^k t\right)^{k\cdot \frac{m}{\lambda_3}}, $$
where $-\lambda_3$ is the equivariant parameter of $\oO_{\mathbb{P}^1}$ in $X$.
\end{conj}

The aim of this section is to 
give an interpretation of the above conjecture 
in terms of wall-crossing of $\Theta$-stable perverse coherent systems. 
Suppose that $\Theta$ lies on 
one of the walls in \eqref{all walls} except the DT/PT wall\footnote{As in \cite{NN}, we exclude the DT/PT wall here. A reason is that, at the DT/PT wall, the simple object 
in (\ref{def:Sk}) is not a line bundle on $\mathbb{P}^1$ nor its
Fourier-Mukai transform, but a skyscraper sheaf of a point, so 
a separate treatment is required.} $L_{\pm}(\infty)$, and $\Theta_{\pm}$ lie in 
its adjacent chambers.
We consider the flip type diagram of $T_0$-fixed loci of good moduli spaces: 
\begin{align}\label{diagram:wall2}
	\xymatrix{
		\bigcup_{n,d} P^{\Theta_{-}}_n(X,d)^{T_0} \ar[rd]_{\pi^-} & & \ar[ld]^{\pi^+} \bigcup_{n,d}P^{\Theta_{+}}_n(X,d)^{T_0} \\
		& \bigcup_{n,d} P^{\Theta}_n(X,d)^{T_0}. & }
\end{align}
Here $P^{\Theta}_n(X,d)^{T_0}$ consists of $\Theta$-polystable perverse coherent systems of type 
\begin{align}\notag
	I_0\oplus S_{k-1}^{\oplus r}[-1], \quad r\geqslant 0,
\end{align}
where $I_0$ is a $T_0$-fixed $\Theta$-stable perverse coherent system,
$S_{k-1}$ is a $T_0$-fixed $\Theta$-stable
perverse coherent sheaf with $\Theta(S_{k-1})=0$, 
and $r$ can be computed from the Chern character of $I_0$. 
By Proposition~\ref{classify:unframed}, 
the object $S_{k-1}$ is given by 
\begin{align}\label{def:Sk}
	S_{k-1}=\begin{cases}
		\oO_{\mathbb{P}^1}(k-1), & \Theta \in L_{-}^{-}(k), \\
		\oO_{\mathbb{P}^1}(-k-1)[1], & \Theta \in L_{-}^{+}(k), \\
		\Upsilon_0(\oO_{\mathbb{P}^1}(-k-1)[1]), & \Theta \in L_{+}^{-}(k), \\
		\Upsilon_0(\oO_{\mathbb{P}^1}(k-1)), & \Theta \in L_{+}^{+}(k).
	\end{cases} 
\end{align}
Here $\Upsilon_0$ is the derived equivalence
under flop (\ref{fequiv})
and 
$\oO_{\mathbb{P}^1}$ is 
scheme theoretically supported on the zero section 
$\mathbb{P}^1\times \{0\}\subset X$. 
For a $T_0$-fixed $\Theta$-stable perverse coherent system 
$I_0$, we consider the following 
sequence of $\Theta$-polystable objects for all $r\geqslant 0$: 
\begin{align*}
	P_{k-1, r}^{I_0}:=\left\{I_0\oplus S_{k-1}^{\oplus r}[-1]
	\right\}\in \bigcup_{n,d}P^{\Theta}_n(X,d)^{T_0}. 
\end{align*}
By Proposition \ref{dim red}, when $m=\lambda_3$ (also taking specialization $\lambda_0+\lambda_1+\lambda_2=0$), 
the invariants in Definition~\ref{taut inv} 
recover Nagao-Nakajima's counting invariants of 
perverse coherent systems on the resolved conifold
 $\oO_{\mathbb{P}^1}(-1,-1)$ \cite{NN}. 
In~\cite[Theorem~3.12]{NN}, they proved a wall-crossing formula 
of their invariants 
by stratifying $\pi^{\pm}$ into Grassmannian bundles and showed that the difference of their invariants under wall-crossing is independent 
of the choice of $I_0$. Motivated by their work, we conjecture this phenomenon extends to $\oO_{\mathbb{P}^1}(-1,-1,0)$:
\begin{conj}\label{wall cross conj}
	Let $\Theta$ lie on 
	one of the 
	walls $L_{\pm}^{-}(k)$, $L_{\pm}^{+}(k)$ in \eqref{all walls}. 
		\begin{itemize}
		\item If $\Theta=(\theta_0,\theta_1)\in L_{-}^-(k)$ or $L_{+}^-(k)$ $($$k\geqslant1$$)$ and 
		$\Theta_{\pm}=(\theta_0\mp 0^+,\theta_1)$, then there exist choices of signs such that
		$$\frac{\sum_{r}t^r\sum_{I\in \pi_+^{-1}(P^{I_0}_{k-1,r})}e_{T_0}(\chi_X(I,I)^{\frac{1}{2}}_0)\cdot e_{T_0\times \mathbb{C}^*}(\chi_X(F)^{\vee}\otimes e^m)}{ 
			\sum_{r}t^r\sum_{I\in \pi_-^{-1}(P^{I_0}_{k-1,r})}e_{T_0}(\chi_X(I,I)^{\frac{1}{2}}_0)\cdot e_{T_0\times \mathbb{C}^*}(\chi_X(F)^{\vee}\otimes e^m)}
		=(1-t)^{k\frac{m}{\lambda_3}}. $$
		\item If $\Theta=(\theta_0,\theta_1)\in L_{-}^+(k)$ or $L_{+}^+(k)$ $($$k\geqslant0$$)$ and 
		$\Theta_{\pm}=(\theta_0\mp 0^+,\theta_1)$, then there exist choices of signs such that
		$$\frac{\sum_{r}t^r\sum_{I\in \pi_+^{-1}(P^{I_0}_{k-1,r})}e_{T_0}(\chi_X(I,I)^{\frac{1}{2}}_0)\cdot e_{T_0\times \mathbb{C}^*}(\chi_X(F)^{\vee}\otimes e^m)}{ 
			\sum_{r}t^r\sum_{I\in \pi_-^{-1}(P^{I_0}_{k-1,r})}e_{T_0}(\chi_X(I,I)^{\frac{1}{2}}_0)\cdot e_{T_0\times \mathbb{C}^*}(\chi_X(F)^{\vee}\otimes e^m)}
		=(1-t^{-1})^{k\frac{m}{\lambda_3}}.$$
	\end{itemize}
\end{conj}
The formulae in Conjecture~\ref{wall cross conj}
in particular imply that the quotient series in the LHS are independent of the choice of $I_0$. 
The above conjecture implies the following wall-crossing formulae of tautological invariants given in Definition \ref{taut inv}:
\begin{prop}\label{wall crossing formula}
	Suppose that 
Conjecture \ref{wall cross conj} is true.
Then we have the following:
\begin{itemize}
\item If $\Theta=(\theta_0,\theta_1)\in L_{-}^-(k)$ or $L_{+}^-(k)$ $($$k\geqslant1$$)$ and 
$\Theta_{\pm}=(\theta_0\mp 0^+,\theta_1)$, then there exist choices of signs such that
$$\frac{\sum_{n,d}P^{\Theta_+}_{n,d}(e^m)q^nt^d}{\sum_{n,d}P^{\Theta_-}_{n,d}(e^m)q^nt^d}
=(1-q^{k}t)^{k\frac{m}{\lambda_3}}. $$
\item If $\Theta=(\theta_0,\theta_1)\in L_{-}^+(k)$ or $L_{+}^+(k)$ $($$k\geqslant0$$)$ and 
$\Theta_{\pm}=(\theta_0\mp 0^+,\theta_1)$, then there exist choices of signs such that
$$\frac{\sum_{n,d}P^{\Theta_+}_{n,d}(e^m)q^nt^d}{\sum_{n,d}P^{\Theta_-}_{n,d}(e^m)q^nt^d}
=(1-q^{k}t^{-1})^{k\frac{m}{\lambda_3}}. $$
\end{itemize}
\end{prop}
\begin{proof}
If $\Theta$ lies on $L_-^-(k)$ or $L_+^-(k)$, we have 
\begin{align*}
&\quad \,\, \sum_{n,d}P^{\Theta_{\pm}}_{n,d}(e^m)q^nt^d \\
&=\sum_{n_0, d_0}q^{n_0}t^{d_0}
\sum_{I_0 \in P_{n_0}^{\Theta \mathchar`-\rm{st}}(X, d_0)^{T_0}}
\left(\sum_{r\geqslant0}
\sum_{I\in \pi_{\pm}^{-1}(P^{I_0}_{k-1,r})}
e_{T_0}(\chi_X(I,I)^{\frac{1}{2}}_0)\cdot e_{T_0\times \mathbb{C}^*}(\chi_X(F)^{\vee}\otimes e^m)
(q^kt)^r \right),
\end{align*}
where $P_{n_0}^{\Theta \mathchar`-\rm{st}}(X, d_0)^{T_0}$ denotes the set of 
$T_0$-fixed $\Theta$-stable perverse coherent systems with numerical class $(n_0,d_0)$.
Applying Conjecture~\ref{wall cross conj}, we prove the proposition in 
this case. The other cases can be similarly obtained. 
\end{proof}
In particular, this gives a wall-crossing interpretation of Conjecture \ref{conj formula in pt chamber} and a conjectural 
formula for non-commutative tautological invariants.
\begin{cor}\label{cor on nc}
Conjecture \ref{wall cross conj} implies Conjecture \ref{conj formula in pt chamber}. 
Moreover, if we further assume the DT/PT conjecture \cite[\S 0.4]{CKM1}, then there exist choices of 
signs such that 
$$\sum_{n,d}P^{\Theta_{\mathrm{NC}}}_{n,d}(e^m)q^nt^d=
M(q)^{2\frac{m}{\lambda_3}}
\prod_{k\geqslant 1}\left(1-q^k t\right)^{k \frac{m}{\lambda_3}}
\prod_{k\geqslant 1}\left(1-q^k t^{-1}\right)^{k \frac{m}{\lambda_3}},$$
where $\Theta_{\mathrm{NC}}=(\theta_0<0,\theta_1<0)$ 
lies in the non-commutative chamber and 
$$M(q):=\prod_{k\geqslant 1}(1-q^k)^{-k}$$ 
is the MacMahon function.
\end{cor}
\begin{rmk}
By Proposition~\ref{dim red}, the substitution 
$m=\lambda_3$ and specialization $\lambda_0+\lambda_1+\lambda_2=0$ allow us to 
recover the formula of non-commutative DT invariants of resolved conifold ~\cite{Sze, Young} from Corollary~\ref{cor on nc}. 
\end{rmk}
Applying the insertion-free limit in Proposition \ref{dim red}, we obtain a wall-crossing formula for cohomological invariants without insertions:
\begin{prop}
		Suppose that 
	Conjecture \ref{wall cross conj} is true.
	Then we have the following: 
\begin{itemize}
\item If $\Theta=(\theta_0,\theta_1)\in L_{-}^-(k)$ or $L_{+}^-(k)$ $($$k\geqslant1$$)$ and 
$\Theta_{\pm}=(\theta_0\mp 0^+,\theta_1)$, then there exist choices of signs such that
\begin{align*}\frac{\sum_{n,d}q^nt^d\sum_{\begin{subarray}{c}I \in P^{\Theta_+}_n(X,d)^{T_0} \end{subarray}}e_{T_0}(\chi_X(I,I)^{\frac{1}{2}}_0)}{\sum_{n,d}q^nt^d\sum_{\begin{subarray}{c}I \in P^{\Theta_-}_n(X,d)^{T_0} \end{subarray}}e_{T_0}(\chi_X(I,I)^{\frac{1}{2}}_0)}
=\left\{\begin{array}{cl}
	\exp\left(-\frac{qt}{\lambda_3}\right), & \mbox{ if } k=1, \\
	& \\
	1, & \mbox{ otherwise}. 
	\end{array}\right. 
\end{align*}
\item If $\Theta=(\theta_0,\theta_1)\in L_{-}^+(k)$ or $L_{+}^+(k)$ $($$k\geqslant0$$)$ and 
$\Theta_{\pm}=(\theta_0\mp 0^+,\theta_1)$, then there exist choices of signs such that
\begin{align*}\frac{\sum_{n,d}q^nt^d\sum_{\begin{subarray}{c}I \in P^{\Theta_+}_n(X,d)^{T_0} \end{subarray}}e_{T_0}(\chi_X(I,I)^{\frac{1}{2}}_0)}{\sum_{n,d}q^nt^d\sum_{\begin{subarray}{c}I \in P^{\Theta_-}_n(X,d)^{T_0} \end{subarray}}e_{T_0}(\chi_X(I,I)^{\frac{1}{2}}_0)}
=\left\{\begin{array}{cl}
	\exp\left(-\frac{qt^{-1}}{\lambda_3}\right), & \mbox{ if } k=1, \\
	& \\
	1, & \mbox{ otherwise}. 
\end{array}\right. 
\end{align*}
\end{itemize}
\end{prop}
\begin{proof}
	Applying
the insertion-free limit in Proposition \ref{dim red} to the LHS of Proposition \ref{wall crossing formula}, we obtain the LHS of the above
formula. The RHS is obtained as 
\begin{align*} 
\lim_{\begin{subarray}{c}Q\, \mathrm{fixed} \\ m\to \infty \end{subarray}}(1-q^{k}t)^{k\frac{m}{\lambda_3}}\big|_{Q=qm}&=
\lim_{\begin{subarray}{c}m\to \infty \end{subarray}}\left(1-\frac{Q^{k}t}{m^k}\right)^{\frac{k}{\lambda_3}\cdot\frac{m^k}{m^{k-1}}}
=
\left\{\begin{array}{cl}
	\exp\left(-\frac{Qt}{\lambda_3}\right), & \mbox{ if } k=1, \\
	& \\
	1, & \mbox{ otherwise}. \qedhere
\end{array}\right. 
 \end{align*}  
\end{proof} 

\subsection{Computations}\label{sect on computation}
In this section, we compute examples to support Conjecture \ref{wall cross conj}.

${}$ \\
\textbf{$\bullet\, I_0=\oO_X$ case.} By Proposition \ref{Z_t chamber}, when $\Theta=(-n+0^+,n+d)$, 
the moduli space 
$P^{\Theta}_n(X,d)$ parametrizes Joyce-Song type stable pairs introduced in \cite[Definition 1.10]{CT1}:
\begin{align*}
	P_n^{\mathrm{JS}}(X,d)
	=\big\{\mbox{JS stable pairs } (F, s) \mbox{ with }
	(d(F), \chi(F))=(d, n)\big\}. 
\end{align*} 
Here in our setting, the 
JS stability for $(F, s)$
is defined by: 
\begin{itemize}
	\item $F$ is a compactly supported one dimensional semistable sheaf, 
	\item $s \neq 0$, and for any subsheaf $\Imm(s) \subset F' \subsetneq F$
	we have $\chi(F')/d(F')<\chi(F)/d(F)$. 
\end{itemize}
\begin{lem}\label{fiber for JS}
Suppose that 
 $\Theta=(\theta_0,\theta_1)\in L_{-}^-(k)$ $($$k\geqslant1$$)$ 
 and take $I_0=\oO_X$, we consider 
\begin{align*}
	P^{\oO_X}_{k-1,d}:=\left\{\oO_X\oplus \oO_{\mathbb{P}^1}(k-1)^{\oplus d}[-1]\right\}\in P^{\Theta}_{kd}(X,d)^{T_0}. 
	\end{align*}
Then fibers of maps $\pi_{\pm}$ in \eqref{diagram:wall2}
at the above points satisfy
$$\pi_{+}^{-1}(P^{\oO_X}_{k-1,d})=P^{\mathrm{JS}}_{kd}(X,d)^{T_0}, \quad 
\pi_{-}^{-1}(P^{\oO_X}_{k-1,d})=\emptyset. $$
\end{lem}
\begin{proof}
By Proposition \ref{Z_t chamber}, $I=(\oO_X\to F)\in \pi_{\pm}^{-1}(P^{\oO_X}_{k-1,d})$ is a $Z_t$-stable pair 
for $t=k+0^{\pm}$ with $d(F)=d$ and $\chi(F)=kd$. By \cite[Proposition 1.11]{CT1},  they are JS stable pairs when $t=k+0^{+}$, 
and there exists no such a pair when $t=k+0^{-}$.
\end{proof}
Torus fixed JS stable pairs are classified as follows:
\begin{lem}\label{T-fixed JS pairs}\emph{(\cite[Lemma 6.6]{CT1})}
Let $k\geqslant 0$, $n=d(k+1)$ and $\{Z_0,Z_{\infty}\}=(\mathbb{P}^1)^{T_0}$ be the torus fixed points.
Then 
a $T_0$-fixed JS stable pair 
$I=(\oO_X\stackrel{s}{\to}  F)\in P^{\mathrm{JS}}_{n}(X,d)^{T_0}$ is precisely of the form
\begin{align*}
F=\bigoplus_{i=0}^k
\oO_{\mathbb{P}^1}\big((k-i)Z_{\infty} +iZ_{0}\big)
\Big(\sum_{j=0}^{d_i-1} t_3^{j} \Big),
\end{align*}
for some $d_0,\ldots, d_k\geqslant 0$ with $\sum_{i=0}^k d_i=d$, and $s$ is given by a canonical section.
For other $n,d$, $P^{\mathrm{JS}}_{n}(X,d)^{T_0}=\emptyset$.
\end{lem}
In this case, all $F$'s are scheme theoretically supported on the Fano 3-fold $Y=\oO_{\mathbb{P}^1}(-1,0)$ and 
we use the sign rule in Remark \ref{sign rmk} for the following:
\begin{thm}\label{JS taut invs}
If $n=d(k+1)$ for an integer 
$k\geqslant 0$, using the sign rule in \eqref{sign rule}, we have  
\begin{align*}
P^{\mathrm{JS}}_{n,d}(e^m)&=\frac{(-1)^{n}}{1!\,2!\,\cdots k!}\cdot
\sum_{\begin{subarray}{c}d_0+\cdots+d_k=d  \\  d_0,\ldots, d_k\geqslant 0 
\end{subarray}}\frac{1}{d_0!\cdots d_k!}\cdot \prod_{\begin{subarray}{c}i<j  \\  0\leqslant i,j \leqslant k \end{subarray}}\left((j-i)+(d_i-d_j)\frac{\lambda_3}{\lambda_0}\right) \\
&\times \prod_{i=0}^k\left(\prod_{\begin{subarray}{c}0\leqslant a\leqslant d_i-1  \\  -i\leqslant b\leqslant k-i  \end{subarray}}\left(\frac{m}{\lambda_3}-a-b\frac{\lambda_0}{\lambda_3}\right)\cdot \prod_{\begin{subarray}{c}1\leqslant a\leqslant d_i  \\  1\leqslant b\leqslant k-i  \end{subarray}}\frac{1}{a+b\frac{\lambda_0}{\lambda_3}}\cdot
\prod_{\begin{subarray}{c}1\leqslant a\leqslant d_i  \\  1\leqslant b\leqslant i  \end{subarray}}\frac{1}{a-b\frac{\lambda_0}{\lambda_3}}\right).
\end{align*}
\end{thm}
\begin{proof}
As in the proof of \cite[Theorem 6.9]{CT1}, we compute directly using Lemma \ref{T-fixed JS pairs}. 
\end{proof}
By Lemma~\ref{fiber for JS}, 
Conjecture \ref{wall cross conj} for $I_0=\oO_X$
is reduced to showing the following identity
for all $n, d \in \mathbb{Z}_{\geqslant 1}$ with 
$d\mid n$,
\begin{align}\label{JS taut conj}P^{\mathrm{JS}}_{n,d}(e^m)=(-1)^d {\frac{n}{d}\cdot\frac{m}{\lambda_3} \choose d}.
\end{align}
This is quite a clean formula although the expression in Theorem \ref{JS taut invs}
looks rather complicated. 
In fact, the authors did not know this formula 
for sometime until we investigate Nagao-Nakajima's 
wall-crossing formula \cite{NN} and arrive 
at Conjecture \ref{wall cross conj}. 
By Proposition \ref{dim red}, the formula recovers the cohomological invariants 
without insertions \cite[Conjecture 6.10]{CT1}.
The identity (\ref{JS taut conj}) is proved in~\cite{CT3}
using its compact analogue and Atiyah-Bott localization.
\begin{thm}\label{thm on js}\emph{(\cite[Theorem 3.10]{CT3})}
The identity~\eqref{JS taut conj} holds, i.e. Conjecture \ref{wall cross conj} holds when $\Theta=(\theta_0,\theta_1)\in L_{-}^-(k)$ $($$k\geqslant1$$)$ and $I_0=\oO_X$.
\end{thm}
\begin{rmk}
Therefore to prove Conjecture \ref{wall cross conj} for $L_{-}^-(k)$, it is enough to show the quotient series in the LHS is independent of the choice of $I_0$.
\end{rmk}

${}$ \\
\textbf{$\bullet\, I_0=I_{l\mathbb{P}^1},$ $k=2$ case.}
We denote 
\begin{align}\label{def of I_lC}
I_{l\mathbb{P}^1}:=\left(\oO_X\twoheadrightarrow \oO_{\mathbb{P}^1}\otimes \sum_{j=0}^{l-1} t_3^{j}\right), \quad l\geqslant 1
\end{align}
where $\oO_{\mathbb{P}^1}$ is scheme theoretically supported on the zero section 
$\mathbb{P}^1 \hookrightarrow X$
of the projection $\pi \colon X \to \mathbb{P}^1$. 
Namely $I_{l\mathbb{P}^1}$ is the ideal sheaf of the $l$-th 
thickening of  
$\mathbb{P}^1$ in the normal direction of $Z$ inside $X$ \eqref{Z:iota}. 
\begin{lem}\label{fiber for I_lC}
Suppose that $\Theta=(\theta_0,\theta_1)\in L_{-}^-(2)$ and take $I_0=I_{l\mathbb{P}^1}$, we consider
$$P^{I_{l\mathbb{P}^1}}_{1,d}:=\left\{I_{l\mathbb{P}^1}\oplus \oO_{\mathbb{P}^1}(1)^{\oplus d}[-1]\right\}\in P^{\Theta}_{2d+l}(X,d+l)^{T_0}. $$

(1) An element in $\pi_{+}^{-1}(P^{I_{l\mathbb{P}^1}}_{1,d})$
is precisely of the form 
\begin{align}\label{equ plus fiber1}
s \colon 
I_{l\mathbb{P}^1}\to \bigoplus_{i=1}^4 F_i\otimes \sum_{j=0}^{d_i-1}t_3^j,
\end{align}
where $d_1, \ldots, d_4 \geqslant 0$
with $\sum_{i=1}^4 d_i=d$, 
$F_i$ are the following $T_0$-equivariant sheaves and $s$
is the canonical $T_0$-equivariant morphism
\begin{align}\label{F:i}
	F_i=\left\{ \begin{array}{ll}
		\oO_{\mathbb{P}^1}(Z_{\infty})\otimes t_1, & i=1, \\
		\oO_{\mathbb{P}^1}(Z_{\infty})\otimes t_2, & i=2, \\
		\oO_{\mathbb{P}^1}(Z_{\infty})\otimes t_3^l, & i=3, \\
		\oO_{\mathbb{P}^1}(Z_{0})\otimes t_3^l, & i=4.
	\end{array}  \right. 
\end{align}

(2) The set
$\pi_{-}^{-1}(P^{I_{l\mathbb{P}^1}}_{1,d})$ is 
empty for $d>0$, and consists of $I_{l\mathbb{P}^1}$ for $d=0$. 
\end{lem}
\begin{proof}
(1) 
The fiber of $\pi_{+}$ consists of isomorphism classes of 
$T_0$-fixed pairs
\begin{align*} 
	(s \colon I_{l\mathbb{P}^1} \to F), \ F \in \langle \oO_{\mathbb{P}^1}(1) \rangle_{\rm{ex}},
	\end{align*}
with no morphism to $(0 \to \oO_{\mathbb{P}^1}(1))$ by the 
$Z_{2+0^+}$-stability. 
From the $T_0$-equivariant Koszul resolution
\begin{align}\label{T-equi resol}
	\cdots \to 
	\oO_X(Z_{\infty})\otimes t_1\oplus \oO_X(Z_{\infty})\otimes t_2\oplus \oO_X\otimes t_3^l \to I_{l\mathbb{P}^1} \to 0,
	\end{align}
we see that
$\mathbb{P}(\Hom(I_{l\mathbb{P}^1}, \oO_{\mathbb{P}^1}(1))$
consists of four $T_0$-fixed points, 
namely canonical morphisms 
$s_i \colon \oO_X \to F_i$ where 
$F_i$ is one of (\ref{F:i}). 
Let us consider the composition
\begin{align*}
	I_{l\mathbb{P}^1} \to F \to F|_{Z},
	\end{align*}
where $Z\subset X$ is given by (\ref{Z:iota}).
The above composition is $T_0$-equivariant
and $F|_{Z}$ is a direct sum of $\oO_{\mathbb{P}^1}(1)$, so 
it is of the form 
\begin{align*}
(s_i^{\oplus k_i})
 \colon I_{l\mathbb{P}^1} \to \bigoplus_{i=1}^4 F_i^{\oplus k_i}, \quad
 k_i \in \mathbb{Z}_{\geqslant 0}. 
\end{align*} 
Then $Z_{2+0^+}$-stability forces
$k_i \leqslant 1$. 
Hence $F$ is a direct sum of thickenings of $F_i$ for $k_i=1$ in
the normal direction of $Z$ inside $X$, so we obtain the 
desired description for fiber of $\pi_+$. 

(2)
By Proposition~\ref{Z_t chamber}, a pair $(F,s)$ in the fiber of $\pi_{-}$ is a $Z_{2-0^{+}}$-stable pair. By the wall-chamber structures of $Z_t$-stable pairs in 
Lemma \ref{lem on zt pair}, 
it is also a $Z_{1+0^+}$-stable pair. 
Since 
we have the inequality 
\begin{align*}
	1 \leqslant \frac{\chi(F)}{d(F)}=1+\frac{d}{d+l},
\end{align*}
which is strict for $d>0$, 
we obtain 
the desired description of the fiber of $\pi_{-}$
by Lemma~\ref{lem on zt pair}. 
\end{proof}

By the above lemma, we can explicitly compute the LHS of the 
formula in Conjecture~\ref{wall cross conj} when $I_0=I_{l\mathbb{P}^1}$ and $k=2$.
Once we know the relevant 
classification of torus fixed loci as in Lemma~\ref{fiber for I_lC}, 
the computations are similar to~\cite[Theorem 6.9]{CT1}
which are direct applications of those formulae 
in Remark \ref{sign rmk}, \ref{rmk on computation}.
Here we omit details and give one example.
\begin{exam}\label{combi identi}
Let $I_0=I_{\mathbb{P}^1}$ and $\Theta\in L_{-}^-(2)$. The degree $d$-term of LHS in Conjecture~\ref{wall cross conj} is
\begin{align*}
	(-1)^d\sum_{\begin{subarray}{c}d_1+d_2+d_3+d_4=d \\
	d_i \geqslant 0
\end{subarray}}
&\prod_{\begin{subarray}{c}1\leqslant i, j \leqslant 4 \\
		0\leqslant k \leqslant d_i-1
	\end{subarray}}
((k-d_j)\lambda_3+\lambda_i-\lambda_j)^{-1}
\cdot \prod_{\begin{subarray}{c}1\leqslant i\leqslant 4 \\
		0\leqslant k \leqslant d_i-1 \\
		1\leqslant j \leqslant 2
\end{subarray}}
((k-1)\lambda_3+\lambda_i-\lambda_j) \\
&\cdot \prod_{\begin{subarray}{c}1\leqslant i \leqslant 4 \\
		0\leqslant k \leqslant d_i-1
\end{subarray}}
(m-k\lambda_3-\lambda_i)
\cdot \prod_{\begin{subarray}{c}1\leqslant i\leqslant 4 \\
		0\leqslant k \leqslant d_i-1
\end{subarray}}
(m-k\lambda_3-\lambda_i+\lambda_1+\lambda_2+\lambda_3),
	\end{align*}
where $\lambda_4:=\lambda_1+\lambda_2+2\lambda_3$. Conjecture~\ref{wall cross conj} predicts that this expression is equal to 
$(-1)^d{\frac{2m}{\lambda_3} \choose d}$. In the following Proposition \ref{check conj muti C case}, we verify this non-trivial 
identity up to $d\leqslant 16$.
\end{exam}
A computer program\footnote{Our use of computer program is simply a brute force checking of whether two rational functions are equal.
We list all cases that our computers can do.} enables us to check Conjecture \ref{wall cross conj} in the following cases.
\begin{prop}\label{check conj muti C case}
Let $\Theta=(\theta_0,\theta_1)\in L_{-}^-(2)$, and take $I_0=I_{l\mathbb{P}^1}$. 
Then Conjecture \ref{wall cross conj} holds in the following cases:
\begin{itemize}
\item $l=1$, up to degree $t^{16}$,
\item $l=2$, up to degree $t^{10}$,
\item $l=3,4$, up to degree $t^9$,
\item $l=5$, up to degree $t^8$,
\item $l=6$, up to degree $t^7$,
\item $l=7,8,9,10$, up to degree $t^6$,
\item any $l$, up to degree $t^5$.
\end{itemize}
Here the sign rule \eqref{sign rule} is as follows: 
we take $\mathrm{sign}(F)=1$
for fibers of $\pi_{+}$ in \eqref{equ plus fiber1}
with $d_2>0$, and $\mathrm{sign}(F)=0$ otherwise.
\end{prop}

${}$ \\
\textbf{$\bullet\, I_0=I_{\mathbb{P}^1},$ $k\geqslant 3$ case.}
\begin{lem} 
Suppose that 
 $\Theta=(\theta_0,\theta_1)\in L_{-}^-(k)$
 and take $I_0=I_{\mathbb{P}^1}$, we consider  
$$P^{I_{\mathbb{P}^1}}_{k-1,d}:=\left\{I_{\mathbb{P}^1}\oplus \oO_{\mathbb{P}^1}(k-1)^{\oplus d}[-1]\right\}\in P^{\Theta}_{kd+1}(X,d+1)^{T_0}. $$

(1)
An element in $\pi_{+}^{-1}(P^{I_{\mathbb{P}^1}}_{k-1,d})$ is precisely of form:
\begin{align}\label{equ plus fiber2}
&\bigoplus_{i=1}^{k-1} \oO_{\mathbb{P}^1}\big((k-1-i)Z_{0}+iZ_{\infty}\big)\otimes t_1 \sum_{j=0}^{d_i-1}t_3^j \nonumber \\ 
s:I_{\mathbb{P}^1}\to \quad  & \bigoplus_{i=1}^{k-1} \oO_{\mathbb{P}^1}\big((k-1-i)Z_{0}+iZ_{\infty}\big)\otimes t_2 \sum_{j=0}^{e_i-1}t_3^j \\
& \bigoplus_{i=0}^{k-1} \oO_{\mathbb{P}^1}\big((k-1-i)Z_{0}+iZ_{\infty}\big)\otimes t_3 \sum_{j=0}^{f_i-1}t_3^j, \nonumber
\end{align}
where $d_i,e_i,f_i\geqslant 0$ with $\sum_{i=1}^{k-1} d_i+\sum_{i=1}^{k-1}e_i+\sum_{i=0}^{k-1} f_i=d$
and $s$ is given by the canonical map.

(2)
An element in $\pi_{-}^{-1}(P^{I_{\mathbb{P}^1}}_{k-1,d})$ is precisely of 
the form
$(s\colon  \oO_X\to \eE)$, 
where $\eE$ fits into the canonical $T_0$-equivariant extension 
\begin{align*}
0\to \oO_{\mathbb{P}^1} \to \eE \to \bigoplus_{\begin{subarray}{c} 1\leqslant i\leqslant k-2 \\ d_i\in\{0,1\} \\ d_1+\cdots+d_{k-2}=d \end{subarray}} \oO_{\mathbb{P}^1}\big((k-1-i)Z_{0}+iZ_{\infty}\big)\cdot d_i\to 0, \end{align*}
and $s$ is given 
by the composition $\oO_X \twoheadrightarrow \oO_{\mathbb{P}^1} \hookrightarrow 
\eE$. 
\end{lem}
\begin{proof} 
(1)
The fiber of $\pi_+$ can be described similarly as in Lemma \ref{fiber for I_lC}, so we omit details. 

(2) The fiber of $\pi_-$ consists of $T_0$-equivariant 
exact sequences of the form
\begin{align}\label{equ on dist triangle for IP1}0\to I_{\mathbb{P}^1} \to *\to  F[-1]\to 0, \end{align}
where $F\in \langle \oO_{\mathbb{P}^1}(k-1) \rangle_{\mathrm{ex}}$
satisfies the $Z_{k+0^-}$-stability. 
Note that we have $\Hom(F[-2],\oO_X)=0$
by the Serre duality.
Therefore 
the map $F[-2]\to I_{\mathbb{P}^1}$ in \eqref{equ on dist triangle for IP1} factors through as
\begin{align*}
	F[-2] \to \oO_{\mathbb{P}^1}[-1] \to I_{\mathbb{P}^1}.
	\end{align*}
By taking cones and a diagram chasing, we obtain an extension 
\begin{align}\label{equ on F tilte}0\to \oO_{\mathbb{P}^1}\to \eE
	\to F\to 0, \end{align}
and $\ast$ is isomorphic to a pair $(s \colon \oO_X\to \eE)$, 
where $s$ is the composition $\oO_X \twoheadrightarrow \oO_{\mathbb{P}^1}
\hookrightarrow \eE$. 
The $Z_{k+0^-}$-stability is equivalent to the condition
\begin{align}\label{vanish:stab}
	\Hom(\oO_{\mathbb{P}^1}(k-1), \eE)=0.
\end{align}
The sheaf $\eE$ is obtained as a $T_0$-equivariant extension 
of $F$ by $\oO_{\mathbb{P}^1}$. Using Serre duality together with Koszul 
resolution~\eqref{T-equi resol}, one calculates
\begin{align*}
\Ext^1_X(\oO_{\mathbb{P}^1}(k-1), \oO_{\mathbb{P}^1})&\cong \ \Hom_{\mathbb{P}^1}(\oO_{\mathbb{P}^1}(Z_{0}+Z_{\infty}),
\oO_{\mathbb{P}^1}(k-1))^{\vee}.  \end{align*}
In the RHS, the $T_0$-fixed morphisms are given by canonical morphisms
$$\oO_{\mathbb{P}^1}(Z_{0}+Z_{\infty})\to \oO_{\mathbb{P}^1}((k-1-i)Z_{0}+iZ_{\infty}), \quad 1\leqslant i \leqslant k-2. $$
Therefore $F$ is of form 
$$F=\bigoplus_{\begin{subarray}{c} 1\leqslant i\leqslant k-2 \\ d_1+\cdots+d_{k-2}=d \end{subarray}} \oO_{\mathbb{P}^1}\big((k-1-i)Z_{0}+iZ_{\infty}\big)\cdot (1+t_3+\cdots+t_3^{d_i-1}). $$
We claim that 
the condition (\ref{vanish:stab})
forces $d_i \leqslant 1$. 
Suppose that $d_i\geqslant 2$
for some $i$. 
By writing 
 $F_i:=\oO_{\mathbb{P}^1}\big((k-1-i)Z_{0}+iZ_{\infty}\big)$ and 
$F^{d_i}:=F_i \otimes (\sum_{j=0}^{d_i-1}t_3^{j})$, 
we have the exact sequence
$$0\to F_i\stackrel{t_3^{d_{i}-1}}{\to} F^{d_i}\to F^{d_{i}-1}\to 0. $$
By applying 
$\Hom(-,\oO_{\mathbb{P}^1})$ to the above exact sequence, 
we obtain the long exact sequence
\begin{align*}
	\cdots \to \Ext^1_X(F^{d_i},\oO_{\mathbb{P}^1})\stackrel{\lambda}{\to} \Ext^1_X(F_{i},\oO_{\mathbb{P}^1}) \stackrel{\nu}{\to} \Ext^2_X(F^{d_{i}-1},\oO_{\mathbb{P}^1}) \to \cdots.
	\end{align*} 
If $\lambda$ is the zero map, 
then
the composition 
$$F_i \stackrel{t_3^{d_{i}-1}}{\hookrightarrow} F^{d_i} \stackrel{ }{\hookrightarrow} F $$ 
factors through a map $F_i\to \eE$, hence violating the
condition (\ref{vanish:stab}). 

We are left to show that $\lambda=0$. 
It is enough to show $\nu$ is injective. 
By the local-to-global spectral sequence, we have 
$\Ext^1_X(F_{i},\oO_{\mathbb{P}^1})\cong H^1(X,\oO_{\mathbb{P}^1}(-k))$
and 
\begin{align*}
\Ext^2_X(F^{d_{i}-1},\oO_{\mathbb{P}^1})&\cong H^0(X,\eE xt^2(F^{d_{i}-1},\oO_{\mathbb{P}^1}))\oplus H^1(X,\eE xt^1(F^{d_{i}-1},\oO_{\mathbb{P}^1})) \\
&\cong H^0(X,\eE xt^2(F^{d_{i}-1},\oO_{\mathbb{P}^1}))\oplus H^1\bigg(X,\eE xt^1\Big(\oO_{\mathbb{P}^1}\otimes \sum_{j=0}^{d_i-1}t_3^{j},\oO_{\mathbb{P}^1}(-k)\Big)\bigg) \\
&\cong H^0(X,\eE xt^2(F^{d_{i}-1},\oO_{\mathbb{P}^1}))\oplus H^1\bigg(X,(\oO_{\mathbb{P}^1}(-1)^{\oplus 2}\oplus \oO_{\mathbb{P}^1})\otimes \oO_{\mathbb{P}^1}(-k)\bigg),
\end{align*} 
where we have used \eqref{T-equi resol} in the last isomorphism. 
So $\Ext^2_X(F^{d_{i}-1},\oO_{\mathbb{P}^1})$ contains $H^1(X,\oO_{\mathbb{P}^1}(-k))$ as a direct summand and 
one can show $\nu$ is the inclusion of this summand. 
\end{proof}
As before, we explicitly compute invariants using Remark \ref{sign rmk}, \ref{rmk on computation}. A computer program enables us to check Conjecture \ref{wall cross conj} in the following cases.
\begin{prop}\label{check conj single C case}
Suppose that 
$\Theta=(\theta_0,\theta_1)\in L_{-}^-(k)$ $($$k\geqslant 3$$)$. 
Then Conjecture \ref{wall cross conj} holds for
 $I_0=I_{\mathbb{P}^1}$ in the following cases:
\begin{itemize}
\item  $k=3$, up to degree $t^5$,
\item  $k=4,5$, up to degree $t^2$,
\item  $k\leqslant 12$, up to degree $t^1$.
\end{itemize}
Here we use the following sign rule in \eqref{sign rule}: 
for fibers of $\pi_{+}$ in \eqref{equ plus fiber2}, 
we take $\mathrm{sign}(F)$ to be the number of $e_i$'s which are positive;
for fibers of $\pi_{-}$, we take $\mathrm{sign}(F)=0$.
\end{prop}
As a corollary of the above computations, we can prove Conjecture \ref{wall cross conj} for the `first wall' and provide several 
checks for the `second wall' of \eqref{all walls}:
\begin{cor}
Let $\Theta=(\theta_0,\theta_1)\in L_{-}^-(k)$. Then Conjecture \ref{wall cross conj} holds in the following cases:
\begin{itemize}
\item  $k=1$ and any $I_0$,
\item  $k=2$ and any $I_0$ up to degree $t^5$.
\end{itemize}
\end{cor} 
\begin{proof}
By Lemma \ref{lem on zt pair}, the only $Z_1$-stable pair is $\oO_X$. When $k=1$, the only possible choice of $I_0$ is $\oO_X$.
Using Lemma \ref{fiber for JS}, we are reduced to Theorem \ref{thm on js}.

By the openness of stability and wall-chamber structures of $Z_t$-stable pairs (see Lemma \ref{lem on zt pair}), 
$Z_2$-stable pairs are $Z_{1+0^+}$-stable pairs, which are JS type stable pairs $(F,s)$ with $\chi(F)=d(F)$. 
By Lemma \ref{T-fixed JS pairs}, they are of the form \eqref{def of I_lC}, which are $Z_{t}$-stable for any $t>1$.
Therefore, the $k=2$ case is reduced to Proposition \ref{check conj muti C case}.
\end{proof}

\begin{rmk}
Finally we remark that one can also study $K$-theoretic generalization of tautological invariants considered in this paper, following 
\cite[Definition 0.2]{CKM1}, and lift the formula in \cite{MMNS} to CY 4-folds.
It may be interesting to pursue this direction in the future. 
\end{rmk}




\providecommand{\bysame}{\leavevmode\hbox to3em{\hrulefill}\thinspace}
\providecommand{\MR}{\relax\ifhmode\unskip\space\fi MR }
\providecommand{\MRhref}[2]{%
  \href{http://www.ams.org/mathscinet-getitem?mr=#1}{#2}}
\providecommand{\href}[2]{#2}


\begin{thebibliography}{MNOP06}


\bibitem[BF]{BF}
K.~Behrend and B.~Fantechi, \emph{The intrinsic normal cone}, Invent.~Math.~
  \textbf{128} (1997), 45--88.

\bibitem[BBD]{BBD}
A. Beilinson, J. Bernstein and P. Deligne, \emph{Faisceaux Pervers}, Ast\'erisque 100,
Soc. Math de France (1983).

\bibitem[Boj]{boj}A. Bojko, \emph{Orientations on the moduli stack of compactly supported perfect complexes over a non-compact Calabi-Yau 4-fold}, arXiv:2008.08441.

\bibitem[BO]{BO}A. Bondal and D. Orlov, \emph{Semiorthogonal decomposition for algebraic varieties}, arXiv:alg-geom/9506012.

\bibitem[BJ]{BJ}
D.~Borisov and D.~Joyce, \emph{Virtual fundamental classes for moduli spaces of
  sheaves on {C}alabi-{Y}au four-folds}, Geom. Topol. (21), (2017) 3231--3311.
  
\bibitem[Bri02]{Bri1}
T. Bridgeland, \emph{Flops and derived categories}, Invent. Math. 147 (2002), 613--632.

\bibitem[Bri07]{Bri2} 
T. Bridgeland, \textit{Stability conditions on triangulated categories}, Ann. of Math. (2) 166 (2007), no. 2, 317--345. 

\bibitem[Bri11]{Bri3}
T. Bridgeland, \emph{Hall algebras and curve-counting invariants}, J. Amer. Math. Soc. 24 (2011), no. 4, 969--998.


\bibitem[Cala]{Cala}
J. Calabrese, \emph{Donaldson-Thomas invariants and flops},
J. Reine Angew. Math. 716 (2016), 103--145.

\bibitem[Cao]{CaoFano}Y.~Cao, \textit{Genus zero Gopakumar-Vafa type invariants for Calabi-Yau 4-folds II: Fano 3-folds},
Commun. Contemp. Math. 22 (2020), no. 7, 1950060, 25 pages.   

\bibitem[CGJ]{CGJ}
Y.~Cao, J.~Gross, and D.~Joyce, \emph{Orientability of moduli spaces of
  {S}pin(7)-instantons and coherent sheaves on {C}alabi-{Y}au 4-folds},
  Adv. Math. \textbf{368}, (2020), 107134. 

\bibitem[CK18]{CK1}
Y.~Cao and M.~Kool, \emph{Zero-dimensional {D}onaldson-{T}homas invariants of
  {C}alabi-{Y}au 4-folds}, Adv. Math. \textbf{338} (2018), 601--648.
 

\bibitem[CK19]{CK2}
Y.~Cao and M.~Kool, \emph{Curve counting and {DT/PT} correspondence for
  {C}alabi-{Y}au 4-folds}, Adv. Math. \textbf{375} (2020), 107371. 

\bibitem[CKM19]{CKM1}
Y.~Cao, M.~Kool, and S.~Monavari, \emph{K-theoretic {DT/PT} correspondence for
  toric {C}alabi-{Y}au 4-folds}, arXiv:1906.07856.
  

\bibitem[CL14]{CL1}
Y.~Cao and N.~C. Leung, \emph{Donaldson-{T}homas theory for {C}alabi-{Y}au
  4-folds}, arXiv:1407.7659.

\bibitem[CL17]{CL2}
Y.~Cao and N.~C. Leung,  \emph{Orientability for gauge theories on
  {C}alabi-{Y}au manifolds}, Adv. Math. \textbf{314} (2017), 48--70.
 
\bibitem[CMT18]{CMT1}
Y.~Cao, D.~Maulik, and Y.~Toda, \emph{Genus zero {G}opakumar-{V}afa type
  invariants for {C}alabi-{Y}au 4-folds}, Adv. Math. \textbf{338} (2018),
  41--92.  

\bibitem[CMT19]{CMT2}
Y.~Cao, D.~Maulik, and Y.~Toda, \emph{Stable pairs and {G}opakumar-{V}afa type invariants for
  {C}alabi-{Y}au 4-folds}, J. Eur. Math. Soc. (JEMS) 10.4171/JEMS/1110.
  
\bibitem[CT19]{CT1}
Y.~Cao and Y.~Toda, \textit{Curve counting via stable objects in derived categories of {C}alabi-{Y}au 4-folds}, arXiv:1909.04897. 

\bibitem[CT20a]{CT2}Y.~Cao and Y.~Toda, \textit{Gopakumar-Vafa type invariants on Calabi-Yau 4-folds via descendent insertions}, Comm. Math. Phys. 383 (2021), no. 1, 281--310.

\bibitem[CT20b]{CT3} 
Y.~Cao and Y.~Toda, \textit{Tautological stable pair invariants of Calabi-Yau 4-folds}, Adv. Math. 396 (2022) 108176.

\bibitem[CJ]{CJ} W. Chuang and D. Jafferis, \textit{Wall crossing of BPS states on the conifold from Seiberg duality and pyramid partitions}, 
Comm. Math. Phys. 292 (2009), 285-301.

\bibitem[GJT]{GJT}J. Gross, D. Joyce and Y. Tanaka, \emph{Universal structures in $\mathbb{C}$-linear enumerative invariant theories I}, arXiv:2005.05637.


\bibitem[HRS]{HRS} D. Happel, I. Reiten, and S. O. Smal\o, \textit{Tilting in abelian categories and quasitilted algebras}, Mem. Amer. Math. Soc, vol. 120, 1996.

\bibitem[King]{King}A. D. King, \textit{Moduli of representations of finite-dimensional algebras}, Q. J. Math. Oxford Ser. (2) 45 (1994), no. 180, 515--530.

\bibitem[KP]{KP}
A.~Klemm and R.~Pandharipande, \emph{Enumerative geometry of {C}alabi-{Y}au
  4-folds}, Comm. Math. Phys. 281 (2008), no. 3, 621--653.
  
  
\bibitem[MMNS]{MMNS} A.~Morrison, S.~Mozgovoy, K.~Nagao, and B.~Szendr\"oi, \textit{Motivic Donaldson-Thomas invariants of the conifold and the refined topological vertex}, Adv.~Math.~230 (2012) 2065--2093. 

\bibitem[MR]{MR}S. Mozgovoy and M. Reineke, \textit{On the noncommutative Donaldson-Thomas invariants arising from brane tilings}, 
Adv.~Math. 223 (2010) 1521--1544.
   
\bibitem[NN]{NN}
K. Nagao and H. Nakajima, \emph{Counting invariants of perverse coherent sheaves and its wall-crossing}, arXiv:0809.2992.
Int. Math. Res. Not. IMRN (2011), no. 17, 3885--3938. 

\bibitem[NY]{NY}
H. Nakajima and K. Yoshioka, \emph{Perverse coherent sheaves on blow-up. I. A quiver description}, Exploring new structures and natural constructions in mathematical physics, 349--386, Adv. Stud. Pure Math., 61, Math. Soc. Japan, Tokyo, 2011.
    

\bibitem[OT]{OT} J.~Oh and R.~P.~Thomas, \textit{Counting sheaves on Calabi-Yau 4-folds, I}, arXiv:2009.05542.

\bibitem[PT]{PT} R. Pandharipande and R.P. Thomas, \textit{Curve counting via stable pairs
in the derived category}, Invent. Math. 178, (2009) 407--447.

\bibitem[PTVV]{PTVV}
T.~Pantev, B.~To\"en, M.~Vaquie, and G.~Vezzosi, \emph{Shifted
  symplectic structures}, Publ.~Math.~IHES \textbf{117} (2013), 271--328.

\bibitem[Sze]{Sze}B. Szendr\"oi, \emph{Non-commutative Donaldson-Thomas invariants and the conifold}, Geom. Topol. 12 (2) (2008)
1171--1202. 

\bibitem[Th]{DT}R. P. Thomas, \textit{A holomorphic Casson invariant for Calabi-Yau 3-
folds, and bundles on K3 fibrations}, J. Diff. Geom. 54 (2000), no. 2, 367--438.

\bibitem[Toda10]{Toda1}Y. Toda, \emph{Curve counting theories via stable objects I: DT/PT correspondence}, J. Amer. Math. Soc. 23 (2010), 1119--1157.

\bibitem[Toda13]{Toda2}Y. Toda, \emph{Curve counting theories via stable objects II: DT/ncDT flop formula},  
J. Reine Angew. Math. 675 (2013), 1--51.


\bibitem[VB]{VB}M. Van den Bergh, \emph{Three dimensional flops and noncommutative rings}, Duke Math. J. 122 (2004), 423--455.

\bibitem[Yos]{Yos}
K. Yoshioka, \emph{Perverse coherent sheaves and Fourier-Mukai transforms on surfaces}, I. Kyoto J. Math. 53 (2013), no. 2, 261--344.

\bibitem[Young]{Young}B. Young, \emph{Computing a pyramid partition generating function with dimer shuffling}, J. Combin. Th. (A) 116 (2009), 334--350.

\end{thebibliography}
\end{document}